\documentclass[11pt, a4paper]{amsart}
\usepackage{amssymb, array, amsmath, amscd, pdfpages, enumerate, amsthm, setspace, hyperref,mathtools}

\usepackage[utf8]{inputenc}
\usepackage{amssymb}
\usepackage{bm}

\newcommand{\eq}[1]{\begin{align*}#1\end{align*}}
\newcommand{\eqn}[1]{\begin{align}#1\end{align}}
\newcommand{\ieq}[1]{$#1$}
\newcommand*{\abs} [1]{\lvert#1\rvert}
\newcommand*{\norm}[1]{\lVert#1\rVert}
\newcommand*{\set} [1]{\{#1\}}
\newcommand*{\setm}[2]{\{\,#1\mid#2\,\}}
\newcommand*{\Setm}[2]{\biggl\{\,#1\biggm| #2\,\biggr\}}
\newcommand*{\iprod}[2]{\langle#1,#2\rangle}
\newcommand*{\Abs}[2][default]{\ifthenelse{\equal{#1}{default}}{\left\lvert#2\right\rvert}{\ldelim{#1}{\lvert}#2\rdelim{#1}{\rvert}}}
\newcommand*{\Norm}[2][default]{\ifthenelse{\equal{#1}{default}}{\left\lVert#2\right\rVert}{\ldelim{#1}{\lVert}#2\rdelim{#1}{\rVert}}}

\newcommand{\ga}{\alpha}
\newcommand{\gb}{\beta}
\renewcommand{\gg}{\gamma}
\newcommand{\gd}{\delta}
\newcommand{\gl}{\lambda}
\newcommand{\gw}{\omega}
\newcommand{\gs}{\sigma}
\newcommand{\eps}{\varepsilon}
\newcommand*{\C}{{\mathbb{C}}}
\newcommand*{\R}{{\mathbb{R}}}
\newcommand*{\N}{{\mathbb{N}}}

\DeclareMathOperator{\re}{Re}

\DeclareMathOperator{\sign}{sign}
\DeclareMathOperator{\supp}{supp}
\DeclareMathOperator*{\essup}{ess\,sup}
\DeclareMathOperator*{\essinf}{ess\,inf}
\DeclareMathOperator*{\dist}{dist}

\newcommand{\citel}[2]{\cite[#2]{#1}}

\newcommand*{\inv}{^{-1}}
\newcommand{\pinv}{^+}
\newcommand*{\Lp}[1][p]{L^{#1}}
\newcommand*{\Lin}{{\mathcal{L}}}
\newcommand{\Dom}{D}

\newcommand{\pmat}[1]{\begin{pmatrix}#1\end{pmatrix}}
\newcommand{\pmatsmall}[1]{\begin{psmallmatrix}#1\end{psmallmatrix}}

\newcommand{\ran}{\textup{Ran}}
\renewcommand{\ker}{\textup{Ker}}

\newcommand{\RB}[1][is]{(#1-A_B)\inv}
\newcommand{\RA}[1][is]{(#1-A)\inv}

\newcommand{\Azhalf}{\cramped{L^{1/2}}}
\renewcommand{\Azhalf}[1][1]{\smash{L\raisebox{1ex}{\tiny$#1/2$}}}
\newcommand{\Azmhalf}{\cramped{L^{-1/2}}}
\renewcommand{\Azmhalf}[1][1]{\smash{L\raisebox{1ex}{\tiny$-#1/2$}}}

\newcommand{\SG}{(T(t))_{t\geq 0}}
\newcommand{\SGB}[1][B]{(T_{#1}(t))_{t\geq0}}

\newcommand{\WP}[3][A]{\textup{WP}_{#2, #3(s)}(#1)}

\newtheorem{theorem}{Theorem}[section]
\newtheorem{lemma}[theorem]{Lemma}
\newtheorem{proposition}[theorem]{Proposition}
\newtheorem{corollary}[theorem]{Corollary}

\theoremstyle{definition}
\newtheorem{definition}[theorem]{Definition}
\newtheorem{assumption}[theorem]{Assumption}
\newtheorem{remark}[theorem]{Remark}

\numberwithin{equation}{section}

\begin{document}

\title{Non-uniform Stability of Damped Contraction Semigroups}

\thispagestyle{plain}

\author[R. Chill, L. Paunonen, D. Seifert, R. Stahn,   and Yu. Tomilov]{Ralph Chill, Lassi Paunonen, David Seifert, Reinhard Stahn and Yuri Tomilov}

\address[R. Chill]{Institut f\"{u}r Analysis, Fakult\"{a}t f\"{u}r Mathematik, TU Dresden, 01062 Dresden, Germany}
\email{ralph.chill@tu-dresden.de}

\address[L. Paunonen]{Mathematics, Tampere University, PO.\ Box 692, 33101 Tampere, Finland}
\email{lassi.paunonen@tuni.fi}

\address[D. Seifert]{School of Mathematics, Statistics and Physics, Newcastle University, Newcastle upon Tyne, NE1 7RU United Kingdom}
\email{david.seifert@ncl.ac.uk}

\address[R. Stahn]{Formerly with Institut f\"{u}r Analysis, Fakult\"{a}t f\"{u}r Mathematik, TU Dresden, 01062 Dresden, Germany}
\email{ReinhardStahn@t-online.de}

\address[Yu. Tomilov]{Institute of Mathematics, Polish Academy of Sciences, \'{S}niadeckich 8, 00956 Warsaw, Poland}
\email{ytomilov@impan.pl}

\thanks{The research of L. Paunonen is funded by the Academy of Finland grants 298182 and 310489.
The work of Yu.\ Tomilov was partially supported the NCN grant UMO-2017/27/B/ST1/00078}

\begin{abstract}
  We investigate the stability properties of strongly continuous semigroups generated by operators of the form $A-BB^\ast$, where $A$ is the  generator of a contraction semigroup and $B$ is a possibly unbounded operator. Such systems arise naturally in the study of hyperbolic partial differential equations with damping on the boundary or inside the spatial domain. As our main results we present general sufficient conditions for non-uniform stability of the semigroup generated by $A-BB^\ast$ in terms of selected observability-type conditions on the pair $(B^\ast,A)$. 
The core of our approach consists of deriving resolvent estimates for the generator expressed in terms of these observability properties.
We apply the abstract results to obtain rates of energy decay in one-dimensional and two-dimensional wave equations, a damped fractional Klein--Gordon equation and a weakly damped beam equation.
\end{abstract}


\maketitle

\section{Introduction}
\label{sec:intro}

In this paper we study the stability properties of abstract differential equations of the form
\eqn{
  \label{eq:ACPintro}
  \dot{x}(t)&=(A-BB^\ast)x(t), \qquad x(0)=x_0\in X.
}
Here $A$ generates a strongly continuous contraction semigroup,
 or typically a unitary group,
on the Hilbert space $X$ and $B$ is a possibly unbounded operator, 
defined on a Hilbert space $U$.
This class of dynamical systems includes
several types of partial differential equations with damping, especially wave equations~\cite{Leb96,AmmTuc01,AnaLea14} and other hyperbolic PDE models~\cite{LiuZha15,DelPat21}. Equations of this form are also often encountered in control theory as a result of feedback interconnections and output feedback stabilisation~\cite{Sle74,Ben78a,GuoLuo02a,LasTri03,CurWei06,CurWei19}. 
Our main interest is in studying
stability properties of the semigroup $\SGB$ generated by $A-BB^\ast$ and the asymptotic behaviour of the solution $x(\cdot)=T_B(\cdot)x_0$ of~\eqref{eq:ACPintro}.
One of the key results concerning
equations of the form~\eqref{eq:ACPintro} is that  stability of $\SGB$ can be characterised in terms of  \emph{observability}
of the pair $(B^\ast,A)$; see~\cite{Sle74,Ben78a,CurWei06,CurWei19}. This relationship is well understood in the context of \emph{exponential stability} and \emph{strong stability}. In this paper we investigate this relationship for semigroups $\SGB$ which are \emph{polynomially stable} or more generally \emph{non-uniformly stable}. Our main results introduce new observability-type conditions
which can be used to guarantee and verify the precise non-uniform stability properties of the differential equation~\eqref{eq:ACPintro}.

The problem in~\eqref{eq:ACPintro} and the associated semigroup $\SGB$ are said to be \emph{(uniformly) exponentially stable} if $\norm{x(t)}\leq Me^{-\gw t}\norm{x_0}$
for all $x_0\in X$ and $t\geq 0$ and for some constants $M$, $\gw>0$.
The weaker notion of \emph{strong stability} requires only that $\norm{x(t)}\to 0$ for $t\to\infty$ for all $x_0\in X$.
The main benefit of exponential stability over strong stability is that the decay of the solutions takes place at a guaranteed \emph{rate} as $t\to \infty$.
In this paper we focus on
\emph{non-uniform stability}~\cite{BatDuy08,BorTom10,RozSei19,Chill_et_al},
where $\SGB$ is strongly stable and all \emph{classical solutions} of~\eqref{eq:ACPintro} decay at a specific rate.
Non-uniform and polynomial stability  have been investigated in detail especially for damped wave equations on multidimensional domains~\cite{Leb96,LiuRao05,BurHit07,AnaLea14,Sta17,CavMa19arxiv,DatKle20}, coupled partial differential equations~\cite{Duy07}, and plate equations~\cite{LiuZha15,LauLea21}.

Under suitable assumptions on $A$ and $B$, exponential stability of the semigroup $\SGB$ is equivalent to ``exact observability''~\citel{TucWei09book}{Ch.~6} of the pair $(B^\ast,A)$~\cite{Sle74,CurWei06}.
In addition, strong stability can be characterised in terms of ``approximate observability'' of $(B^\ast,A)$~\cite{Ben78a}.
In this paper we show that several modified concepts, each of which may be seen as ``quantified approximate observability'' of the pair $(B^\ast,A)$, lead to non-uniform stability of the semigroup $\SGB$.
In particular,
we say that $(B^\ast,A)$ satisfies the \emph{non-uniform Hautus test} if
there exist functions $M, m :\R\to [r_0,\infty)$ with $r_0>0$ such that~\citel{Mil12}{Sec.~2.3}
\eq{
  \norm{x}_X^2\leq M(s) \norm{(is-A)x}_X^2 + m(s) \norm{B^\ast x}_U^2 , \qquad  x\in\Dom(A),\  s\in\R.
}
In addition, if $A$ is skew-adjoint we say that the pair $(B^\ast,A)$ satisfies the \emph{wavepacket condition} if there exist bounded functions
$\gg,\gd:\R\to (0,\infty)$
such that~\citel{Mil12}{Sec.~2.5}
  \eqn{
    \label{eq:WPcondIntro}
    \norm{B^\ast x}_U\geq \gg(s) \norm{x}_X, \qquad  x\in \WP{s}{\gd}, ~ s\in\R.
  }
  Here $\WP{s}{\gd}$ denotes the spectral subspace of $-iA$ associated with the interval $(s-\gd(s),s+\gd(s))$ (elements of $\WP{s}{\gd}$ are called \emph{wavepackets} of~$A$).

  The following theorem summarises our main results on these two observability concepts.
  The precise assumptions of Theorem~\ref{thm:FreqDomIntro}
are stated in Assumption~\ref{ass:ABBass} in Section~\ref{sec:wellposedness}, and they are automatically satisfied whenever $A$ generates a strongly continuous contraction semigroup and $B\in \Lin(U,X)$.
The results employ a function $\mu:\R\to [r_0,\infty)$, $r_0>0$, such that 
\eqn{
\label{eq:BRBboundIntro}
\norm{B^\ast (1+is-A)\inv B}\leq \mu(s), \qquad s\in\R.
}
As shown in Section~\ref{sec:wellposedness}, we may always choose $\mu$ in such a way that $\mu(s)\lesssim 1+s^2$, $s\in\R$. Moreover, 
in the case where 
$B\in \Lin(U,X)$
and in many concrete applications
$\mu$ may be taken to be constant.
Finally, a measurable function $N:[0,\infty)\to (0,\infty)$ is said to have \emph{positive increase} if there exist $\ga$, $c_\ga$, $s_0>0$ such that $N(\gl s)/N(s)\geq c_\ga\gl^\ga$ for all $\gl\geq 1$ and $s\geq s_0$.

  \begin{theorem}
    \label{thm:FreqDomIntro}
    Assume that the operators $A$ and $B$ satisfy Assumption~\textup{\ref{ass:ABBass}}
 and that  $\mu:\R\to[r_0,\infty)$, $r_0>0$, is an even function such that~\eqref{eq:BRBboundIntro} holds.

    If the pair $(B^\ast,A)$ satisfies the non-uniform Hautus test for some continuous and even functions $M$ and $m$, and if the function 
$N: [0,\infty)\to (0,\infty)$ defined by
$N(\cdot):=M(\cdot)\mu(\cdot)+m(\cdot)\mu(\cdot)^2$
is strictly increasing and has positive increase,
then $\SGB$ is non-uniformly stable and
    \eqn{
      \label{eq:NUStabIntro2}
  \norm{T_B(t)x_0}\leq \frac{C}{N\inv(t)} \norm{(A-BB^\ast)x_0}, \quad  x_0\in \Dom(A-BB^\ast),\ t\geq t_0,
    }
 for some $C$, $t_0>0$,
 where $N\inv$ is the inverse function of $N$.

If $A$ is skew-adjoint and $(B^\ast,A)$ satisfies the wavepacket condition~\textup{\eqref{eq:WPcondIntro}} for continuous and even functions  $\gg,\gd$ such that $\gg(\cdot)\inv \gd(\cdot)\inv$ is strictly increasing and has positive increase, then
$\SGB$ is non-uniformly stable and~\eqref{eq:NUStabIntro2}
is satisfied for
$N(\cdot):=\gg(\cdot)^{-2}\gd(\cdot)^{-2}\mu(\cdot)^2$.
  \end{theorem}

Equations of the form~\eqref{eq:ACPintro}
 in particular include
the damped second-order equation
\eqn{
  \label{eq:AbsWEIntro}
  \ddot{w}(t)+L w(t)+DD^\ast \dot{w}(t)=0 , \qquad w(0)\in H_{1/2},\  \dot{w}(0)\in H,
}
for a positive operator  $L$ on a Hilbert space $H$ and $D\in \Lin(U,H_{-1/2})$,
 where $H_{1/2}$  is the domain of the fractional power $L^{1/2}$ and $H_{-1/2}$ is its dual with respect to the pivot space $H$.
Non-uniform stability of
such systems
has been studied in the literature in the case where $D\in \Lin(U,H)$,
and in particular
it was shown in~\cite{AnaLea14},~\citel{JolLau20}{App.~B}
that for such operators $D$ the problem~\eqref{eq:ACPintro} is non-uniformly stable whenever the
``Schr\"odinger group'' generated by $i L$
with the observation operator $D^\ast$ is observable
in a  certain generalised sense.
We subsequently refer to this property as the 
\emph{Schrödinger group associated with the pair} $(D^\ast,iL)$ \textit{being observable}.
In this paper we show that the same observability condition for the
Schr\"odinger group  generated by $i L$
serves as a sufficient condition for the wavepacket condition and the non-uniform Hautus test for the pair $(B^\ast,A)$. Moreover, our results generalise the results in~\citel{AnaLea14}{Thm.~2.3} and~\citel{JolLau20}{App.~B} to the case of general damping operators $D\in \Lin(U,H_{-1/2})$.
Finally, the second part of Theorem~\ref{thm:FreqDomIntro} was proved in~\citel{Pau17b}{Thm.~6.3} in the special case where $A$ is a diagonal operator with uniform spectral gap and $B\in \Lin(U,X)$.

As our last observability-type concept we introduce \emph{non-uniform observability} of the pair $(B^\ast,A)$, which requires that
  there exist $ \gb\geq 0$ and $\tau,c_\tau>0$ such that
  \eqn{
   \label{eq:NUObsIntro}
 c_\tau \norm{(I-A)^{-\gb}x}_X^2\leq \int_0^\tau \norm{B^\ast T (t) x}_U^2 \,dt, \qquad  x\in \Dom(A),
  }
  where $\SG$ is the contraction semigroup generated by $A$.
	Note that if $\beta=0,$ then non-uniform observability reduces to the classical notion of \emph{exact observability} of  $(B^\ast,A)$.
  The main result of Section~\ref{sec:timedomresults}, Theorem~\ref{thm:NUObstoResgrowth}, shows that if 
$(B^\ast,A)$ is non-uniformly observable with parameter $\gb\in(0,1]$ 
and if $B\in\Lin(U,X)$,
 then the semigroup $\SGB$ is polynomially stable
  and~\eqref{eq:NUStabIntro2} holds for $N\inv(t)=t^{1/(2\beta)}$.
Related generalisations of exact observability have previously been used as sufficient conditions for non-uniform stability of damped second-order systems of the form~\eqref{eq:AbsWEIntro} in~\cite{AmmTuc01,AmmNic15,AmmBch17}.
Moreover, in the 
special case $\beta=1/2$, 
similar generalised observability conditions were used in~\cite{Rus75} and~\citel{Duy07}{Sec.~5} to prove polynomial stability of~\eqref{eq:ACPintro}.
Finally, 
non-uniform stability of~\eqref{eq:AbsWEIntro} 
for a special class of dampings satisfying  $\norm{L^{-\gb}x}\lesssim \norm{D^\ast x}\lesssim \norm{L^{-\gb}x}$ for some $\gb>0$ and all $x\in X$
was studied in~\cite{LiuZha15}, and
 for $DD^\ast = f(L)$ with some function $f$
 in~\cite{DelPat21}.
In Section~\ref{sec:timedomresults} we show that the assumptions in~\cite{LiuZha15} imply non-uniform observability of the pair $(B^\ast,A)$, and our results in particular establish a new proof of~\citel{LiuZha15}{Thm.~2.1}.

  The core of our approach in Sections~\ref{sec:freqdomresults} and~\ref{sec:timedomresults} consists of deriving upper bounds
  for the resolvent norms $\norm{(is-A+BB^\ast)\inv}$, $s\in \R$,
  in terms of the different types of observability-type condition.
In
Section~\ref{sec:Optimality}
we address optimality of our results. In particular, we present an abstract result which describes how sharpness of the resolvent bound can be used to deduce optimality of the decay rate~\eqref{eq:NUStabIntro2} of the semigroup $\SGB$.
In addition, in the case where $A$ is  skew-adjoint we prove a lower bound for resolvent norms of $A-BB^\ast$ in terms of the restrictions of $B^\ast$ to eigenspaces of $A$. Combining these two results allows us to prove that Theorem~\ref{thm:FreqDomIntro} is optimal in several situations of interest, and 
 in particular if $A$ has compact resolvent and uniformly separated eigenvalues.

  In the last part of the paper  we apply our main results to derive rates of energy decay for solutions of selected PDE models, namely wave equations on one- and two-dimensional spatial domains with different types of damping, a fractionally damped Klein--Gordon equation,
  and a weakly damped Euler--Bernoulli beam equation. In most of these examples the wavepackets are simply finite linear combinations of eigenfunctions~\citel{TucWei09book}{Sec.~6.9}.
  In our one-dimensional  wave and beam equations, the eigenvalues of $A$ have a uniform spectral gap and, as a result, we obtain a particularly simple form of the wavepacket condition~\eqref{eq:WPcondIntro}. 
Moreover, our general optimality results in Section~\ref{sec:Optimality} guarantee that the decay estimates we obtain in these cases are sharp.
   On the other hand, for two-dimensional wave equations with viscous damping our results are typically suboptimal.
   This is due to the phenomenon that in certain cases the smoothness of the damping profile improves the degree of polynomial stability~\cite{BurHit07,AnaLea14,DatKle20}, whereas  observability-type conditions do not in general distinguish between smooth and rough dampings.
  Indeed, comparing different types of viscous damping reveals
  natural limitations to optimality of decay rates derived from observability conditions, and we discuss this topic in detail
  in Section~\ref{sec:2Dwaves}.

The paper is organised as follows. In Section~\ref{sec:assumptions} we state the main assumptions on the operators $A$ and $B$ and recall essential results concerning non-uniform stability of strongly continuous semigroups. In Section~\ref{sec:freqdomresults} we present the main results showing that the non-uniform Hautus test and the wavepacket condition imply non-uniform stability of $\SGB$.
In particular, in the second part of Section~\ref{sec:freqdomresults} we reformulate these results specifically for damped second-order systems, and present sufficient conditions for non-uniform stability of~\eqref{eq:AbsWEIntro} based on observability of the Schr\"odinger group.
Next, in Section~\ref{sec:timedomresults} we show that non-uniform observability in the sense of~\eqref{eq:NUObsIntro} implies polynomial stability of~$\SGB$.
In Section~\ref{sec:Optimality} we present a series of abstract results concerning optimality of the stability results in the previous sections.
Finally, in Section~\ref{sec:PDEmodels} we study energy decay for several PDE models.

\textbf{Notation.}
If $X$ and $Y$ are Banach spaces and $A: \Dom(A)\subseteq X\rightarrow Y$ is a linear operator, we denote by $\Dom(A)$, $\ker(A)$ and $\ran(A)$ the domain, kernel and range of $A$, respectively. Moreover, $\gs(A)$, $\gs_p(A)$,
and $\rho(A)$ denote the spectrum, the point spectrum
and the resolvent set of $A$, respectively. 
 The space of bounded linear operators from $X$ to $Y$ is denoted by $\Lin(X,Y)$.   
The notation $X \hookrightarrow Y$ will mean that $X\subseteq Y$ with continuous and dense embedding.
We denote the norm on a space $X$ by $\norm{\cdot}_X$ and its inner product by $\iprod{\cdot}{\cdot}_X$, and we omit the subscripts when there is no risk of ambiguity.
We assume all our Banach and Hilbert spaces to be complex.

Let $\R_+:=[0,\infty)$, and let $\C_\pm$ stand for the open right and left half-planes $\setm{\lambda\in\C }{ \re\, \lambda \gtrless0}$, respectively.
We denote by $\chi_E$  the characteristic function of a set $E.$
For two functions $f: E\subseteq \R\to \R_+$ and $g: \R_+\to \R_+$ we write $f(t)=O(g(\abs{t}))$ if there exist $C$, $t_0>0$ such that $f(t)\leq C g(\abs{t})$ whenever $\abs{t}\geq t_0$. If in addition $g(t)>0$ whenever $\abs{t}\geq t_0$, we write $f(t)=o(g(\abs{t}))$ if $ f(t)/g(\abs{t})\to 0$ as $\abs{t}\to \infty$.
For real-valued quantities $p$ and $q$, we use the notation $p\lesssim q$ if $p\leq C q$ for some constant $C > 0$ which is independent of all the parameters that are free to vary in the given situation.
The notation $p\gtrsim q$ is defined analogously.

\section{Preliminaries}
\label{sec:assumptions}

\subsection{Standing assumptions and well-posedness}
\label{sec:wellposedness}

Let $A:  \Dom(A)\subseteq X\to X$ be the generator of a contraction semigroup $(T (t))_{t \ge 0}$ on a Hilbert space $X$.
All semigroups considered in this paper are strongly continuous.
 For $\gl_0\in \rho(A)$ we equip $\Dom(A)$ with the graph norm  $\norm{x}_{1}=\|(\lambda_0-A)x\|_X$, $x \in \Dom(A),$ and denote the Hilbert space defined in this way by $X_1.$ 
Defining $X_{-1}$ as the completion of $X$ with respect to the norm 
$\norm{x}_{-1}=\|(\lambda_0-A)^{-1}x\|_X$, we obtain a Hilbert space $X_{-1}$
such that $X_1 \hookrightarrow X \hookrightarrow X_{-1}$.
 The operator $A$ has a unique extension $A_{-1}$ to $X_{-1},$ with domain $\Dom(A_{-1})=X$, and $A_{-1}$ generates a contraction semigroup $(T_{-1} (t))_{t \ge 0}$
on $X_{-1}$ which is unitarily equivalent to $(T (t))_{t \ge 0}.$ In particular, $A_{-1}\in \Lin(X,X_{-1})$ and the operators $A$, 
$A_{-1}$ are unitarily equivalent and thus have the same spectrum.   
Moreover, any $S \in \Lin(X)$ commuting with $A$ has a (unique) continuous extension to an operator in
$\Lin(X_{-1}),$ unitarily equivalent to $S$; see~\citel{TucWei09book}{Sec.~2.10}.

To state our main assumptions, we let
$V$ be a Hilbert space such that
$X_1\subseteq V\subseteq X$ with continuous embeddings. 
In particular,
 $V$ is dense in $X$ and
  we consider the Gelfand triple $V\hookrightarrow X\hookrightarrow V^\ast$, where $V^\ast$ is the dual of $V$ with respect to the pivot space $X$~\cite[Sec.~2.9]{TucWei09book}. 
We denote by $\iprod{\cdot}{\cdot}_{V^\ast,V}:  V^\ast\times V\to \C$ the unique continuous extension of the inner product of $X$, and we define
$V_A:=\setm{x \in V}{ A_{-1}x \in V^\ast}$.
In the following we state our standing assumptions on the operators $A: \Dom(A)\subseteq X\to X$ and 
$B\in \Lin(U,X_{-1})$, where $U$ is another Hilbert space.

\begin{assumption}
  \label{ass:ABBass}
 The operators $A: \Dom(A)\subseteq X\to X$ and $B\in \Lin(U,X_{-1})$ have the following properties.
\begin{itemize}
  \item[\textup{(H1)}] 
The generator $A$ of the contraction semigroup $\SG$ satisfies
    $\re \iprod{A_{-1}x}{x}_{V^\ast,V}\leq 0$ for all 
$x \in V_A$.
  \item[\textup{(H2)}] We have $B\in \Lin(U,V^\ast)$ and 
	    $\ran( (\gl_0-A_{-1})\inv B)\subseteq V$ for some (or equivalently all)  $\gl_0\in\rho(A)$.
  \end{itemize}
\end{assumption}

Assumption~\ref{ass:ABBass} in particular requires that $\ran(B)\subseteq X_{-1}\cap V^\ast$.
Note that when $A$ is not skew-adjoint, the space $V^\ast$ is not necessarily contained in $X_{-1}$; it is instead a subspace of $X_{-1}^d$, the first extrapolation space for the adjoint $A^\ast$~\citel{TucWei09book}{Sec.~2.10}. 
If $B\in \Lin(U,X)$, which we will refer to as $B$ being \emph{bounded},
then Assumption~\ref{ass:ABBass} is automatically satisfied for any generator $A$ of a contraction semigroup $\SG$ with the choices $V=V^\ast=X$.

We write $B^\ast \in \Lin(V,U)$ for the adjoint of $B\in \Lin(U,V^\ast)$, where $V$ is identified with $(V^\ast)^\ast$ via the pivot duality through $X$.
In particular, 
\eq{
\iprod{Bu}{x}_{V^\ast,V}
=\iprod{u}{B^\ast x}_U, \qquad x\in V,\; u\in U.
}
Moreover, (H2) in Assumption~\ref{ass:ABBass} and the closed graph theorem imply that $B^\ast (\gl-A_{-1})\inv B\in \Lin(U)$ for all $\gl\in\rho(A)$.
We formally define the operator $A_B=A_{-1}-BB^\ast$ on  $X$ by 
\begin{subequations}
\label{defn_of_ab}
\begin{align}
A_Bx&=A_{-1}x - BB^\ast x, \qquad x \in \Dom(A_B),\\
\Dom(A_B)&=\setm{x \in V}{ A_{-1}x - BB^\ast x \in X}.
\end{align}
\end{subequations}
As shown in the following lemma, Assumption~\ref{ass:ABBass} guarantees that $A_B$ generates a contraction semigroup $\SGB$ on $X$. In particular, the orbits of this semigroup are the solutions of the abstract Cauchy problem
\begin{subequations}
\label{eq:ACPmain}
  \begin{align}
   \dot{x}(t) &= A_B x(t), \qquad t\ge0, \\x(0) &= x_0\in X.
  \end{align}
\end{subequations}
For  $x_0\in X$ the orbit  $x(\cdot)=T_B(\cdot)x_0$ is a \emph{mild solution} of~\eqref{eq:ACPmain}, and it is a \emph{classical solution} if and only if
 $x_0\in \Dom(A_B)$~\citel{AreBat11book}{Ch.~3}.
  
\begin{lemma}
  \label{lem:ABBsemigroup}
	Let $A$ and $B$ satisfy Assumption~\textup{\ref{ass:ABBass}}. Then the operator $A_B$ defined in~\eqref{defn_of_ab} generates a strongly continuous contraction semigroup $\SGB$ on $X$.
Moreover, we have $\rho(A)\cap \overline{\C_+} \subseteq \rho(A_B)\cap \overline{\C_+},$ 
  \eqn{
    \label{eq:ABBiprod}
    \re \iprod{(is-A_B)x}{x} \geq \norm{B^\ast x}^2, \qquad s\in \R, ~x\in \Dom(A_B),
  }
and
  \eqn{
\label{eq:RBboundwrtBRB}
    \norm{(\gl-A_{-1})\inv B}^2\leq \frac{1}{\re \gl}\norm{B^\ast (\gl-A_{-1})\inv B}, \qquad \gl\in\C_+.
  }
\end{lemma}

\begin{proof}
  First note that if $x\in X$ and $u\in U$ are such that $A_{-1}x+Bu=:y\in X$, then condition (H2) implies that for any $\gl_0\in\rho(A)$ we have $x=(\lambda_0-A_{-1})^{-1}(\gl_0 x-y+Bu)\in V$ and $A_{-1}x=y-Bu\in V^\ast$. 
	Thus $x\in V_A$ and condition (H1) implies that
\begin{subequations}
    \label{eq:ABBdissipativity}
    \eqn{
      \re \iprod{A_{-1}x+Bu}{x}_X
      &=\re \iprod{A_{-1} x}{x}_{V^\ast,V}
      +\re \iprod{Bu}{x}_{V^\ast,V}\\
      &\leq\re \iprod{B^\ast x}{u}_U.
    }
\end{subequations}
Let $s\in\R$ and $x\in \Dom(A_B)$, and choose $u=-{B^\ast}x$. Then~\eqref{eq:ABBdissipativity} immediately implies~\eqref{eq:ABBiprod}.
  In particular, $A_B$ is dissipative.
  
	To prove that $\rho(A)\cap \overline{\C_+} \subseteq \rho(A_B)\cap \overline{\C_+}$,
 fix  
$\gl\in\rho(A)\cap \overline{\C_+},$ let $u\in U$
 and choose $x=(\lambda-A_{-1})^{-1}Bu$. Then $A_{-1}x+Bu =\lambda (\lambda-A_{-1})^{-1}Bu \in X$ and~\eqref{eq:ABBdissipativity} implies that
  \eq{
 (\re \gl) \norm{(\lambda-A_{-1})^{-1}Bu}^2\leq
    \re \iprod{B^\ast(\lambda-A_{-1})^{-1}Bu}{u}.
  }
In particular, this inequality implies~\eqref{eq:RBboundwrtBRB}.
  Moreover, this estimate shows that the operator $G(\lambda):=B^\ast(\lambda-A_{-1})^{-1}B\in \Lin(U)$ satisfies
  $\re G(\gl)\geq 0,$
and consequently  $I+G(\gl)$ is boundedly invertible in $\Lin(U).$
A direct verification shows  that $\gl-A_B$ has bounded inverse given by
  \eqn{
    \label{eq:KatoPert}
   (\lambda-A_B)^{-1} = (\lambda-A_{-1})^{-1}(I-B(I+G(\gl))\inv B^\ast \RA[\gl]),
  }
	and we deduce the required spectral inclusion $\rho(A)\cap \overline{\C_+} \subseteq \rho(A_B)\cap \overline{\C_+}$.
  	In particular, $A_B$ is closed. Since $A_B$ is dissipative and $\C_+ \subseteq \rho(A_B)$, its domain is dense in $X$ by~\citel{TucWei09book}{Prop.~3.1.6}.
	Hence
  $A_B$ is $m$-dissipative, and  by the Lumer--Phillips theorem it generates a strongly continuous contraction semigroup on $X$.
 \end{proof}

\begin{remark}\label{extension}
If Assumption~\ref{ass:ABBass} holds, then for every $\lambda \in \C_+$
 the right-hand side of~\eqref{eq:KatoPert} extends uniquely to a mapping from the (not necessarily closed) subspace $X + \ran (B)$ of $X_{-1}$ to $X$, simply by replacing $(\gl-A)\inv$ by $(\gl-A_{-1})\inv$.
We use this formula to define the extension of 
$(\lambda-A_B)^{-1}$
to an operator
 $(\lambda-A_B)^{-1}: X + \ran (B) \to X$.
In particular, we have
\begin{align*}
 (\lambda -A_B)^{-1} B 
 & = (\lambda -A_{-1})^{-1} B (I+G(\lambda ))^{-1} \in {\mathcal L} (U,X) 
\end{align*}
for $\gl\in\C_+$.
The identity $(\lambda -A_B )^{-1} = (I + (1 - \lambda )(\lambda -A_B)^{-1} ) (1 -A_B)^{-1}$ 
 shows that also for arbitrary $\gl\in\rho(A_B)$ the operator
 $(\lambda -A_B)^{-1}$ extends uniquely to a mapping from $X+{\rm Ran}\,  B$ into $X$, and that $(\lambda -A_B)^{-1} B \in {\mathcal L} (U,X)$. 
For $\gl\in\rho(A_B)$ and $u\in U$ we have 
$(\gl-A_B)\inv Bu\in V$ and 
\eq{
(\gl-A_{-1}+BB^\ast )(\gl-A_B)\inv Bu = Bu,
}
and if $x\in V$ is such that $(\gl - A_{-1}+BB^\ast) x\in X+\ran(B) $ (in particular, if $x\in \Dom(A)$), then 
\eq{
(\gl-A_B)\inv (\gl-A_{-1}+BB^\ast )x = x.
}
\end{remark}

\begin{remark}
\label{rem:ABBAltDomain}
Define $X_B := \Dom(A)+\ran((\gl_0-A_{-1})^{-1} B)$,
where $\gl_0\in\rho(A)$.
The space $X_B$ is
  independent of the choice of $\gl_0$, and $X_B\subseteq V$ by Assumption~\ref{ass:ABBass}. 
Moreover, 
the domain of $A_B$ has the useful alternative characterisation
\eq{
\Dom(A_B)=\setm{x \in X_B}{ A_{-1}x + BB^\ast x \in X}.
}
Here the non-trivial inclusion can be verified as in the beginning of the proof of Lemma~\ref{lem:ABBsemigroup}.
\end{remark}

Our results in Section~\ref{sec:freqdomresults} employ a parameter which describes the growth of the operator-valued function $\gl\mapsto B^\ast (\gl-A_{-1})\inv B$ on a vertical line in~$\C_+$. In particular, we take $\mu:\R\to [r_0,\infty)$, $r_0>0$, to be a function such that
\eqn{
\label{eq:BRBboundPrelims}
\norm{B^\ast (1+is-A_{-1})\inv B}\leq \mu(s), \qquad s\in\R,
}
and the rate of growth of $\mu$ affects the resolvent estimates in our results.
The following lemma shows that $\mu$ can be taken to be uniformly bounded whenever 
 $B\in \Lin(U,X)$, and that estimate~\eqref{eq:BRBboundPrelims} always holds for a quadratic function $\mu$.

\begin{lemma}
\label{lem:BRBboundProps}
If $A$ and $B$ satisfy Assumption~\textup{\ref{ass:ABBass}}, then the following hold.
\begin{itemize}
\setlength{\itemsep}{.2ex}
\item[\textup{(a)}] The estimate~\eqref{eq:BRBboundPrelims} holds for $\mu(s)= c(1+s^2)$, $s\in\R$, for some $c>0$.
\item[\textup{(b)}] If $B\in \Lin(U,X)$, then~\eqref{eq:BRBboundPrelims} holds for $\mu(s)\equiv c$ with some $c>0$.
\item[\textup{(c)}] If~\eqref{eq:BRBboundPrelims} holds, then $\norm{(1+is-A_{-1})\inv B}\leq \mu(s)^{1/2}$ for $s\in\R$. 
\end{itemize}
\end{lemma}

\begin{proof}
Part (b) follows directly from the assumption that $A$ generates a contraction semigroup, which implies that $\norm{(1+is-A)\inv}\leq 1$ for all $s\in\R$. Moreover, part (c) follows from~\eqref{eq:RBboundwrtBRB} in Lemma~\ref{lem:ABBsemigroup}.
To prove part (a), fix $s\in\R$ and let
$R = (1+is-A_{-1})\inv$.
Using the identity
$R = (I-A_{-1})\inv -is (I-A)\inv R$ we see that
\eq{
\norm{B^\ast R B}
&\leq \norm{B^\ast (I-A_{-1})\inv B} + \abs{s}\norm{B^\ast(I-A)\inv}\norm{ R B}
\lesssim 1 + \abs{s}\norm{ R B}
}
and similarly
\eq{
\norm{ R B}
& \leq \norm{ (I-A_{-1})\inv B} + \abs{s}\norm{(1+is-A)\inv}  \norm{ (I-A_{-1})\inv B}
\lesssim 1+\abs{s}.
}
Together these estimates give $\norm{B^\ast (1+is-A_{-1})\inv B}\lesssim 1+s^2$, $s\in\R$.
\end{proof}

Estimates of the form~\eqref{eq:BRBboundPrelims} have been studied extensively in the control theory literature. In particular,
for a \emph{bounded} function $\mu$ the estimate in~\eqref{eq:BRBboundPrelims} is known as the property of 
\emph{well-posedness} of the operator-valued ``transfer function'' $\gl\mapsto B^\ast (\gl-A_{-1})\inv B$; see~\cite{Sal87a,GuoLuo02a,Sta02,TucWei14}. 
This property has been verified in the literature for several different types of PDE systems; see for instance~~\cite{AmmTuc01,GuoLuo02a,LasTri03, TucWei14, AmmNic15}.
As shown in the next lemma, validity of \eqref{eq:BRBboundPrelims} for a bounded function $\mu$ moreover implies that $B^\ast$ is an \emph{admissible observation operator} for the semigroup $(T (t))_{t \ge 0}$, which is to say that
$B^* T(\cdot)x \in L^2(0,\tau; U)$ for all $x \in D(A)$ and $\tau >0$. This property will be useful in discussing the relationship between our results and existing results in the literature.
In addition, the following lemma shows that under the same assumption $B$ is an \emph{admissible control operator}
in the sense that
$\int_{0}^{\tau}T_{-1}(\tau-t)Bu(t)\, dt\in X$ for all $u \in L^2(0,\tau; U)$ and $\tau>0$.

\begin{lemma}\label{lem:Bstaradmissible}
 Let $A$ and $B$ satisfy Assumption~\textup{\ref{ass:ABBass}}.
If~\eqref{eq:BRBboundPrelims} is satisfied for a bounded function $\mu$, then $B$ and $B^\ast$ are, respectively, admissible control and observation operators for the semigroup $\SG$ generated by $A$.
\end{lemma}

\begin{proof}
Since $A$ and $B$ satisfy Assumption~\ref{ass:ABBass}, it is straightforward to verify that the operator $S: \Dom(S)\subseteq X\times U\to X\times U$
defined by
\eq{
 S=\pmat{A_{-1} & B \\ B^\ast & 0}, \qquad  \Dom(S)=\left \{\pmat{x\\u}\in X\times U: A_{-1}x + Bu \in X\right \}
}
 is a system node on $(U,X,U)$ in the sense of~\citel{Sta02}{Def.~2.1}. 
Moreover, estimate~\eqref{eq:ABBdissipativity} for $(x,u)\in \Dom(S)$ and~\citel{Sta02}{Thm.~4.2} imply that the system node $S$ is impedance passive in the sense of~\citel{Sta02}{Def.~4.1}.
The transfer function of the system node $S$ is given by $G(\gl)= B^\ast (\gl-A_{-1})\inv B$ for $\gl\in\rho(A)$.
Hence the assumption that~\eqref{eq:BRBboundPrelims} is satisfied for a bounded function $\mu$ together with~\citel{Sta02}{Thm.~5.1} imply that the system node $S$ is well-posed in the sense of~\citel{Sta02}{Def.~2.6}. 
In particular, 
$B\in \Lin(U,X_{-1})$ and 
 $B^\ast\in \Lin(X_1,U)$ are, respectively, admissible control and observation operators for the semigroup generated by $A$.
\end{proof}

\subsection{Damped second-order problems}\label{second_order_setup}

In this section we wish to use the framework introduced in Section~\ref{sec:wellposedness} to study a class of abstract
 second-order equations with damping. 
To this end, we consider a positive self-adjoint and boundedly invertible operator $L: \Dom(L)\subseteq H\to H$ on a Hilbert space $H$. 
We write $H_1$ for the domain of $L$ equipped with the norm $\norm{x}_{H_1} = \norm{Lx}_H$, $x\in H_1$, and define $H_{1/2}$ to be the domain of the fractional power $L^{1/2}$ equipped with the norm $\norm{x}_{H_{1/2}} = \norm{L^{1/2}x}_H$, $x\in H_{1/2}$.
 We denote by $H_{-1/2}$
 the dual of $H_{1/2}$
with respect to the pivot space $H$.
For an operator
$D\in \Lin(U,H_{-1/2})$, where $U$ is another Hilbert space, we consider the differential equation 
\begin{subequations}
\label{dwe}
  \begin{align}
    &\ddot{w}(t)+L w(t)+DD^\ast \dot{w}(t)=0, \qquad t \ge 0, \\
		&w(0)=w_0\in H_{1/2}, \qquad \dot{w}(0)=w_1\in H.
  \end{align}
\end{subequations}
Such systems have been studied extensively; see for instance  
	\cite{LaTri00,GuoLuo02a,AnaLea14,AmmNic15} and the references therein. 
This class of systems in particular contains the wave equation with viscous damping on a two-dimensional bounded and convex domain $\Omega\subseteq \R^2$ with (necessarily Lipschitz) boundary $\partial \Omega$,
\eq{
    w_{tt}(\xi,t)-\Delta w(\xi,t)+b(\xi)^2 w_t(\xi,t)=0, \qquad t > 0,
}
where 
   $b\in \Lp[\infty](\Omega)$ is a non-negative function and we impose Dirichlet boundary conditions.
	In this situation we may choose $H=U=\Lp[2](\Omega)$, let 
$L=-\Delta$ be the (negative) Laplacian on $H$ with Dirichlet boundary conditions, and define
  $D\in \Lin(U,H)$  by $D u= bu$ for all $u\in U$.
This partial differential equation will be studied in detail in Section~\ref{sec:2Dwaves}.

In order to formulate the abstract system~\eqref{dwe} as a first-order abstract Cauchy problem of the form~\eqref{eq:ACPmain}, 
 we proceed as in~\citel{TucWei14}{Sec.~6}.
In particular, we let  $x(\cdot)=(w(\cdot),\dot{w}(\cdot))$ 
and take $X$ to be the Hilbert space  $X = H_{1/2}\times H$ equipped with the inner product  
\ieq{
\iprod{x}{y}_{X}=\iprod{x_1}{y_1}_{H_{1/2}} + \iprod{x_2}{y_2}_{H}
}
for $x=(x_1, x_2)$, $y=(y_1,y_2)\in X$.
The operators $A: \Dom(A)\subseteq X\to X$ and $B:U\to X_{-1}$ in Section~\ref{sec:wellposedness} are defined as
\eq{
   A = \pmat{0 & I \\ -L & 0} 
\qquad \mbox{and} \qquad 
 B = \pmat{0 \\ D} 
}
with $\Dom(A)=H_{1}\times H_{1/2}$ and
  $X_{-1}=H\times H_{-1/2}$.
Then $A$ is a skew-adjoint operator and thus it generates a unitary group	$(T (t))_{t \in \mathbb R}$ on  $X$.
  We may choose 
    $V=H_{1/2}\times H_{1/2}$, which has the corresponding dual space $V^\ast  =H_{1/2}\times H_{-1/2}$.
  The dual pairing of $V$ and $V^\ast$ is given by
  \eq{
    \iprod{x}{y}_{V^\ast,V} = \iprod{x_1}{y_1}_{H_{1/2}} + \iprod{x_2}{y_2}_{H_{-1/2},H_{1/2}}
  }
	for $x=(x_1, x_2)\in V^\ast$, $y=(y_1,y_2)\in V$.

Condition (H1) is satisfied since
  $\re \iprod{A_{-1}x}{x}_{V^\ast, V}=0$ for $x\in V = V_A$, as is easily verified.
In addition, we have  both
  $B\in \Lin (U,X_{-1})$ and
  $B\in \Lin (U,V^\ast)$.
For $\gl\in\rho(A)$ the resolvent of $A$ has the form
  \eq{
    (\gl-A)\inv = \pmat{\gl(\gl^2+L)\inv & (\gl^2+L)\inv\\-L(\gl^2+L)\inv&\gl(\gl^2+L)\inv},
  }
 and an analogous formula holds for $(\gl-A_{-1})\inv$. Therefore we in particular have
$\ran (A_{-1}^{-1} B) \subseteq V$, and thus 
 condition (H2) in Assumption~\ref{ass:ABBass} is satisfied. 
By Lemma~\ref{lem:ABBsemigroup} the operator $A_B$ defined in~\eqref{defn_of_ab} generates a contraction semigroup on $X$, as also shown
in~\cite[Prop.~7.6.1]{LaTri00} and~\citel{GuoLuo02a}{Thm.~1}.

It is straightforward to see that $B^\ast = (0,D^\ast)\in \Lin(V,U)$, where $D^\ast \in \Lin(H_{1/2},U)$ is the adjoint of $D\in \Lin(U,H_{-1/2})$.
Therefore the formula for $(\gl-A_{-1})\inv$ implies that
  \eq{
    B^\ast (\gl-A_{-1})\inv B = \gl D^\ast (\gl^2+L_{-1})\inv D, \qquad \gl\in\C_+.
  }
Moreover, 
$\norm{D^\ast ( (1+is)^2+L_{-1})\inv D}
=\norm{D^\ast ( (1-is)^2+L_{-1})\inv D} $, $s\in\R$.
Hence if 
\eqn{
\label{eq:BRBboundSecondOrder}
 s\,\norm{D^\ast ( (1+is)^2+L_{-1})\inv D}\leq \mu_0(s), \qquad s\in\R_+,
}
for some $\mu_0:\R_+\to [r_0',\infty)$, $r_0'>0$, then 
condition~\eqref{eq:BRBboundPrelims} holds for some even function $\mu:\R\to [r_0,\infty)$, $r_0>0$, satisfying $\mu(s)\lesssim \mu_0(\abs{s})$, $s\in\R$. 
Conversely, property~\eqref{eq:BRBboundPrelims} implies the above estimate for $\mu_0:\R_+\to [r_0,\infty)$ defined by  $\mu_0(s)=\mu(s)$, $s\in\R_+$.
The estimate~\eqref{eq:BRBboundSecondOrder} has been shown to hold for a bounded function 
 $\mu_0$
for several PDE models having   
our second-order form~\eqref{dwe};
 see for instance~\cite{AmmTuc01,GuoLuo02a,LasTri03}. On the other hand, as shown in~\cite{LasTri81} and~\citel{Wei03}{Sec.~4}, unbounded functions $\mu_0$ are needed in some cases including wave equations with boundary damping.
In the case where $D\in \Lin(U,H)$, we have $B\in \Lin(U,X)$ and, in particular,~\eqref{eq:BRBboundPrelims} holds for a bounded function~$\mu$ by  Lemma~\ref{lem:BRBboundProps}.
 
\subsection{Resolvent estimates and non-uniform stability}

Throughout the paper we are interested in finding sufficient conditions for the spectrum of the operator $A_B$ defined in~\eqref{defn_of_ab} to be contained in $\C_-$ and in obtaining a resolvent estimate of the form
\eqn{
  \label{eq:ABBresestimate}
  \norm{(is-A_B)\inv} \leq N(s), \qquad   s\in\R,
}
for an explicit function $N: \R\to (0,\infty)$. 

In order to pass from the resolvent estimate~\eqref{eq:ABBresestimate} to sharp rates of decay for the semigroup $\SGB$
   we make use of the following abstract result from~\citel{RozSei19}{Thm.~3.2}; see~\citel{BorTom10}{Thm.~2.4} for the case where $N$ is a polynomial.
Recall that a measurable function $N:\R_+\to(0,\infty)$ is said to have \emph{positive increase} if there exist constants $\alpha,s_0>0$ and $c_\alpha\in(0,1]$ such that
\begin{equation}\label{eq:pos_inc}
\frac{N(\lambda s)}{N(s)}\ge c_\alpha\lambda^\alpha,\qquad\lambda\ge1,\ s\ge s_0.
\end{equation}
When $N:\R_+\to(0,\infty)$ is non-decreasing but not necessarily strictly increasing we take $N\inv$ to denote the right-continuous right-inverse of $N$ defined by $N\inv(t)=\sup\setm{s\ge0}{N(s)\le t}$ for $t\ge N(0)$.

\begin{theorem}[{\textup{\cite[Thm.~3.2]{RozSei19}}}]
  \label{thm:sg_decay}
  Let $\SG$ be a strongly continuous contraction semigroup on a Hilbert space $X$, with generator $A$.	
If $i\R\subseteq \rho(A)$
and if $\norm{(is-A)\inv}\leq N(\abs{s})$  for all $s\in\R$, where $N:\R_+\to(0,\infty)$ is a continuous non-decreasing  function of  positive increase, then
\begin{equation}\label{eq:Minv}
  \|T (t)A^{-1}\|=
O\biggl(\frac{1}{N\inv(t)}\biggr)
,\qquad t\to\infty.
\end{equation}
\end{theorem}

The class of functions satisfying~\eqref{eq:pos_inc} contains all regularly varying functions $N:\R_+\to(0,\infty)$ which have positive index~\cite[Sec.~2]{RozSei19}, and in particular it contains any measurable function $N:\R_+\to(0,\infty)$ defined for all sufficiently large values of $s\ge0$ by $N(s)=s^\alpha\log(s)^\beta$, where $\alpha>0$ and  $\beta\in\R$.
As discussed in~\cite{BorTom10,RozSei19,DebSei19b}, Theorem~\ref{thm:sg_decay} is optimal in several senses, and  for a large class of semigroups the condition of positive increase is even a \emph{necessary} condition for~\eqref{eq:Minv} to hold.

\begin{remark}\label{rem:BT}
If
$N(s)=C(1+\abs{s})^\alpha$ in Theorem~\ref{thm:sg_decay} for some constants $C,\alpha>0$, then~\eqref{eq:Minv} becomes $\|T (t)A^{-1}\|=O(t^{1/\alpha})$ as $t\to\infty$. It is shown in~\cite[Thm.~2.4]{BorTom10} that
for individual orbits of $\SG$
one obtains the even better decay rate $\|T (t)x\|=o(t^{-1/\alpha})$ as $t\to\infty$ for all $x\in \Dom(A)$.
\end{remark}

In subsequent sections we shall repeatedly make use of  the following lemma when proving resolvent estimates;
 see e.g.~\citel{AreBat11book}{Prop.~4.3.6} for a proof of a more general result.

\begin{lemma}\label{sp_contr}
Let $A$ be the generator of a contraction semigroup on a Hilbert space $X$ and let $s \in \R$.  
If there exists $c_s>0$ such that
\begin{equation}\label{lower_bound}
\norm{x}\leq c_s \norm{(is-A)x},
\qquad x \in \Dom(A), 
\end{equation}
then $is \in \rho (A)$ and $\| (is -A)^{-1} \|\le c_s.$
\end{lemma}

We shall also make use of the following lemma on adjoints in the case where $A$ is a skew-adjoint operator.
Here the composition $(\gl-A_B)\inv B$ in part~(b) is defined as in Remark~\ref{extension}.

\begin{lemma}
\label{lem:RBadjoints}
Let $A$ and $B$  satisfy Assumption~\textup{\ref{ass:ABBass}} and assume that $A$ is skew-adjoint.
\begin{itemize}
\item[\textup{(a)}] We have
\eq{
((\gl-A_{-1})\inv B)^\ast = B^\ast (\overline{\gl}+A)\inv , \qquad \gl\in\rho(A).
}
\item[\textup{(b)}] If $\re \iprod{A_{-1}x}{x}_{V^\ast,V}=0$ for all $ x\in V_A$, then the adjoint $A_B^\ast$ of $A_B$ defined in~\eqref{defn_of_ab} is given by 
\begin{subequations}
\label{eq:ABast}
\eqn{
A_B^\ast x&= -A_{-1} x-BB^\ast x, \qquad  x\in \Dom(A_B^\ast),\\
\Dom(A_B^\ast ) &= \setm{x\in V}{A_{-1}x+BB^\ast x\in X}.
 }
\end{subequations}
Moreover,
$((\gl-A_B)\inv B)^\ast =B^\ast (\overline{\gl}-A_B^\ast)\inv$ for $\gl\in\rho(A_B)\cap\overline{\C_+}$.
\end{itemize}
\end{lemma}

\begin{proof}
To prove part (a),
let $\gl\in\rho(A)$,  $x\in X$ and $u\in U$.
By density of $X$ in $X_{-1}$, we may find a sequence
$(y_k)_{k\in\N}\subseteq X$ such that $\norm{y_k-Bu}_{X_{-1}}\to 0$ as $k\to\infty$. 
Since $(\overline{\lambda}+A_{-1})\inv \in \Lin(X_{-1},X)$, we  also have
\eq{
\norm{(\overline{\lambda}+A_{-1})\inv Bu-(\overline{\lambda}+A)\inv y_k}_{X}\to 0,\qquad k\to\infty.
}
Hence
 the definition of $B^\ast$ and skew-adjointness of $A$ imply that
\eq{
\MoveEqLeft\iprod{u}{B^\ast(\gl-A)\inv x}_U
=\iprod{Bu}{(\gl-A)\inv x}_{V^\ast,V}
=\iprod{Bu}{(\gl-A)\inv x}_{X_{-1},X_1}\\
&=\lim_{k\to \infty}\iprod{y_k}{(\gl-A)\inv x}_{X_{-1},X_1}
=\lim_{k\to \infty}\iprod{y_k}{(\gl-A)\inv x}_X\\
&=\lim_{k\to \infty}\iprod{(\overline{\gl}+A)\inv y_k}{ x}_X
=\iprod{(\overline{\gl}+A_{-1})\inv Bu}{ x}_X.
}
Since $x$ and $u$ were arbitrary, we have $( B^\ast (\gl-A)\inv)^\ast =(\overline{\gl}+A_{-1})\inv B $.

To prove (b),
we define 
\eq{
\widetilde{A}_B x&= -A_{-1} x-BB^\ast x, \qquad  x\in \Dom(\widetilde{A}_B),\\
\Dom(\widetilde{A}_B) &= \setm{x\in V}{A_{-1}x+BB^\ast x\in X}.
 }
Since $-A$ and $B$ satisfy Assumption~\ref{ass:ABBass} (with the same choice of $V$), $\widetilde{A}_B$ generates a contraction semigroup on $X$ by Lemma~\ref{lem:ABBsemigroup}.
The assumption that $\re \iprod{A_{-1}x}{x}_{V^\ast,V}=0$ for $x\in V_A$ and a simple polarisation argument imply that
 $ \iprod{A_{-1}x}{y}_{V^\ast,V}=-\iprod{x}{A_{-1}y}_{V,V^\ast}$ for $x,y\in V_A$, where we define $\iprod{z_1}{z_2}_{V,V^\ast} := \overline{\iprod{z_2}{z_1}_{V^\ast,V}}$ for $z_1\in V$, $z_2\in V^\ast$.
Hence if $x\in \Dom(A_B)\subseteq V_A$ and $y\in \Dom(\widetilde{A}_B)\subseteq V_A$, then
\eq{
\iprod{A_B x}{y}_X
=\iprod{A_{-1}x-BB^\ast x}{y}_{V^\ast,V}
=\iprod{x}{(-A_{-1}-BB^\ast)y}_{V,V^\ast}
=\iprod{x}{\widetilde{A}_By}_X.
}
Thus $A_B^\ast $ is an extension of $\widetilde{A}_B$,
and since $\rho(A_B^\ast)\cap \rho(\widetilde{A}_B)\neq \varnothing$ we further see that
$A_B^\ast =\widetilde{A}_B$.

Now let $\gl\in\rho(A_B)\cap\overline{\C_+}$, $x\in X$ and $u\in U$.
We have $(\overline{\gl}-A_B^\ast)\inv x\in \Dom(A_B^\ast)\subseteq V_A$. Moreover,
by Remark~\ref{extension} we have $(\gl-A_B)\inv Bu\in V_A$
and
\eq{
\MoveEqLeft\iprod{u}{B^\ast (\overline{\gl}-A_B^\ast )\inv x}_U
=\iprod{Bu}{ (\overline{\gl}-A_B^\ast)\inv x}_{V^\ast,V}\\
&=\iprod{(\gl-A_{-1}+BB^\ast)(\gl-A_B)\inv Bu}{ (\overline{\gl}-A_B^\ast)\inv x}_{V^\ast,V}\\
&=
\iprod{(\gl-A_B)\inv Bu}{ (\overline{\gl}+A_{-1} + BB^\ast)(\overline{\gl}-A_B^\ast)\inv x}_{V,V^\ast}\\
&=\iprod{(\gl-A_B)\inv Bu}{  x}_X.
}
Since $\gl\in\rho(A_B)\cap\overline{\C_+}$, $x\in X$ and $u\in U$ were arbitrary, the proof is complete.
\end{proof}

The following proposition 
 presents some general consequences of resolvent estimates of the form~\eqref{eq:ABBresestimate}.
In particular, the last part concerns the effect of scaling the operator $B$ on the resulting resolvent estimate.
Once again, the composition $(is-A_B)\inv B$ for $s\in\R$ is defined as in Remark~\ref{extension}.
As noted in Section~\ref{second_order_setup}, the additional assumptions in~(b) are in particular satisfied for the class of second-order systems considered there.

\begin{lemma}
  \label{lem:ABBresgrowthlemma}
  	Let $A$ and $B$  satisfy Assumption~\textup{\ref{ass:ABBass}} and let $A_B$ be as defined in~\eqref{defn_of_ab}.
	If $i\R\subseteq \rho(A_B)$ 
	and 
if $N:\R\to (0,\infty)$ is such that 
\textup{\eqref{eq:ABBresestimate}} holds,
 then the following are true.
  \begin{enumerate}
    \item[\textup{(a)}] For  $s\in\R$, we have
      \eq{
	\norm{B^\ast (is-A_B)^{-1}}&\leq N(s)^{1/2},\\
\norm{\RB[is]B}&
\lesssim 1+ N(s), \\
	\norm{B^\ast \RB[is]B}&\leq 1.
      }

    \item[\textup{(b)}] If either $B\in \Lin(U,X)$, or 
\eq{
A^\ast =-A \qquad \mbox{and} \qquad \re \iprod{A_{-1}x}{x}_{V^\ast,V}=0, \quad x\in V_A,
}
then
\ieq{
\norm{(is-A_B)^{-1}B}\leq N(s)^{1/2} 
} for all $s\in\R$.
    \item [\textup{(c)}]  
Let $\kappa>0$ and consider the operator $A_{B, \kappa }: \Dom(A_{B,\kappa })\subseteq X\to X$ defined by
\eq{
A_{ B,\kappa} x &= A_{-1}x-\kappa^2 BB^\ast x, \qquad x\in \Dom(A_{ B,\kappa}), \\
\Dom(A_{ B,\kappa}) &= \setm{x\in V}{A_{-1}x-\kappa^2 BB^\ast x \in X}.
}
Then 
       $i\R\subseteq \rho(A_{ B,\kappa})$ and 
	$\norm{(is-A_{B,\kappa })\inv}\lesssim 1+N(s)^2$ for $s\in\R$. If the assumptions in part \textup{(b)} hold, then $\norm{(is-A_{B,\kappa })\inv}\lesssim N(s)$ for  $s\in\R$.
  \end{enumerate}
\end{lemma}

\begin{proof}
To prove the first estimate in~(a), fix $s\in\R$ and $y\in X$, and let $x=(is-A_B)^{-1}y\in\Dom(A_B)$. Then $\norm{x}\leq N(s)\norm{y}$ and
    $(is-A_B)x=y$, and hence, by~\eqref{eq:ABBiprod} in Lemma~\ref{lem:ABBsemigroup},
    \eq{
      \norm{B^\ast x}^2
      \leq \re \iprod{y}{x}
      \leq \norm{y}\norm{x}
      \leq N(s)\norm{y}^2.
    }
    Since  $s\in\R$ and $y\in X$ were arbitrary, the first estimate in part (a) follows.

To prove the second and third estimates in~(a), we begin by deriving a preliminary estimate. 
Let $\gl\in\overline{\C_+}$ and $u\in U$. 
If we define the composition $(\gl-A_B)\inv B$ as in Remark~\ref{extension} and let
$x = (\gl-A_B)\inv Bu\in X$, then
Remark~\ref{extension} implies that $x\in V$ and
\ieq{
A_{-1} x + B(u-B^\ast x)
= \gl x\in X.
}
Estimate~\eqref{eq:ABBdissipativity} in the proof of Lemma~\ref{lem:ABBsemigroup} shows that
    \eq{
      (\re \gl) \norm{x}^2& = \re \iprod{A_{-1}x+B(u-B^\ast x)}{x}_X
       \leq \re \iprod{B^\ast x}{u-B^\ast x}_U\\
      &= \re \iprod{B^\ast x}{u}_U - \norm{B^\ast x}_U^2.
    }
    In particular,  
$\norm{B^\ast (\gl-A_B)\inv Bu}=\norm{B^\ast x}\leq \norm{u}$ for all $\gl\in \overline{\C_+}$, which implies the third estimate in~(a).
On the other hand, for $\gl = 1+is$ with $s\in\R$, the same estimate shows that 
\eq{
\norm{(1+is-A_B)\inv Bu}^2
&\leq
\re \iprod{B^\ast x}{u}_U - \norm{B^\ast x}_U^2\\
&\leq \re \iprod{B^\ast (1+is-A_B)\inv Bu}{u}_U 
\leq 1.
}
This inequality together with the property that (see Remark~\ref{extension})
\eq{
(is-A_B)\inv Bu
= \bigl(I+(is-A_B)\inv\bigr)(1+is-A_B)\inv Bu , \qquad s\in\R,
}
finally implies the second estimate in~(a).

In order to prove (b), we first note that 
under the additional assumptions it follows either from boundedness of $B$ or from Lemma~\ref{lem:RBadjoints}(b) that
the adjoint $A_B^\ast$ is given by~\eqref{eq:ABast}
and that 
$ ((is-A_B)\inv B)^\ast=B^\ast (-is-A_B^\ast)\inv  $, $s\in\R$.
Proceeding as in the case of the first estimate in part (a), we may use the structure of $A_B^\ast$ to show that $\norm{B^\ast (-is-A_B^\ast)\inv}^2\leq \norm{(-is-A_B^\ast)\inv}$ for $s\in\R$.
Hence for all $s\in\R$ we have
\eq{
\norm{(is-A_B)\inv B} 
= \norm{B^\ast (-is-A_B^\ast)\inv}
\leq
 \norm{(is-A_B)\inv}^{1/2}
\leq N(s)^{1/2}.
}

To show (c), let $\kappa>0$ and $s\in\R$ be fixed. 
Moreover, let $x\in \Dom(A_{ B,\kappa})$ and $y=(is-A_{ B,\kappa})x\in X$. 
Estimate~\eqref{eq:ABBiprod} in Lemma~\ref{lem:ABBsemigroup} (applied to the operators $A$ and $\kappa B$) implies that
$  \norm{B^\ast x}^2 \leq \kappa^{-2} \norm{x}\norm{y} $.
We have
\eq{
      y
      =(is-A_{-1}+\kappa^2 BB^\ast)x
      =(is-A_{-1}+ BB^\ast)x + (\kappa^2-1)BB^\ast x,
    }
and since $x\in V$ and $(is-A_{-1}+ BB^\ast)x\in X+\ran(B)$, 
     Remark~\ref{extension} gives
      $$x= \RB[is]y + (1-\kappa^2)\RB[is]BB^\ast x.$$
Using Young's inequality we obtain
    \eq{
      \norm{x}^2
      &\leq 2 N(s)^2 \norm{y}^2 + 2(1-\kappa^2)^2 \norm{\RB[is]B}^2 \norm{B^\ast x}^2\\
      & \leq 2 N(s)^2 \norm{y}^2 + 2\frac{(1-\kappa^2)^2}{\kappa^2} \norm{\RB[is]B}^2 \norm{x}\norm{y} \\
      & \leq 2 N(s)^2 \norm{y}^2 + \frac{1}{2}\norm{x}^2 +\frac{2(1-\kappa^2)^4}{\kappa^4}\norm{\RB[is]B}^4 \norm{y}^2 .
    }
 Since  $A_{ B,\kappa}$ generates a contraction semigroup by Lemma~\ref{lem:ABBsemigroup}, the claims follow from parts (a) and (b) together with Lemma~\ref{sp_contr}.
\end{proof}
  
The estimate $\norm{B^\ast \RB[is]B}\leq 1$, $s\in\R$, in part (a) was proved in~\cite[Lem.~2.2.6, P6]{Oos00} in the case where  $B\in \Lin(U,X)$, and 
a similar result for general $B$ in the case of 
second-order systems was presented in~\citel{WeiTuc03}{Thm.~1.3}.

\section{Frequency domain criteria for resolvent bounds and non-uniform stability }
\label{sec:freqdomresults}

\subsection{Criteria for first-order problems}

In this section we consider the semigroup $\SGB$
generated by the operator $A_B$ defined in~\eqref{defn_of_ab}, and present 
sufficient conditions for non-uniform stability of this semigroup
 in terms of observability properties of the pair $(B^\ast,A)$.
Theorem~\ref{thm:sg_decay} allows us to focus on estimating the resolvent of $A_B$ on the imaginary axis, and shows that whenever $\norm{(is-A_B)^{-1}}\leq N(\abs{s})$, $s\in\R$, for some continuous non-decreasing $N: \R_+\to(0,\infty)$ with positive increase, the classical solutions $x(\cdot)=T_B(\cdot)x_0,$  $x_0\in \Dom(A_B),$ of~\eqref{eq:ACPmain}
satisfy
\eqn{
  \label{eq:NUStabFreq}
  \norm{T_B(t)x_0}\leq \frac{C}{N\inv(t)}\norm{A_B x_0}, \qquad t\geq t_0,
}
for some constants $C,t_0>0$. 

Our first main result is based on the following Hautus-type condition with variable parameters. 
The same condition with bounded functions $M$ and $m$ was used in~\cite{Mil12} to study observability properties of the pair $(B^\ast,A)$.

\begin{definition}
  \label{def:NUHautus}
  The pair $(B^\ast,A)$ is said to satisfy the \emph{non-uniform Hautus test} if
  there exist $M$, $m: \R\to [r_0,\infty), r_0>0,$ such that
\eqn{
\label{eq:NUHautus}
  \norm{x}_X^2\leq M(s) \norm{(is-A)x}_X^2 + m (s) \norm{B^\ast x}_U^2 , \qquad  x\in\Dom(A),\ s\in\R.
}
\end{definition}

The following theorem presents a norm bound for the resolvent of $A_B$ on $i\mathbb R$
when the pair $(B^\ast,A)$ satisfies the non-uniform Hautus test.
General properties of the function $\mu$ in condition~\eqref{eq:BRBboundHautusThm} were discussed in Section~\ref{sec:wellposedness} and in Lemma~\ref{lem:BRBboundProps}.

\begin{theorem}
  \label{thm:HautustoResgrowth}
  Let $A$ and $B$ satisfy Assumption~\textup{\ref{ass:ABBass}}.
Assume further that 
$M, m,\mu: \R\to [r_0,\infty), r_0>0,$ are 
such that 
the pair $(B^\ast,A)$
satisfies the non-uniform Hautus test for the functions $M$ and $m$,
and 
\eqn{
\label{eq:BRBboundHautusThm}
\norm{B^\ast (1+is-A_{-1})\inv B}\leq \mu(s), \qquad s\in\R.
}
Then the operator
		 $A_B$ defined in~\eqref{defn_of_ab} satisfies
 $i\R\subseteq \rho(A_B)$ and
  \eq{
    &\norm{(is-A_B)\inv}\lesssim M(s)\mu(s)+m(s)\mu(s)^2, \qquad s\in\R.
  }

Conversely, if 
$N:\R\to (0,\infty)$ is such that
 $\norm{(is-A_B)\inv}\leq N(s)$ for all $s\in\R$,
then~\eqref{eq:NUHautus} holds for $M(\cdot)=2 N(\cdot)^2$ and a function $m$ such that $m(s)\lesssim 1+N(s)^2$ for $s\in\R$.
If, in addition, either $B\in \Lin(U,X)$, or  $A^\ast =-A$ 
and $\re \iprod{A_{-1}x}{x}_{V^\ast,V}=0$ for all $ x\in V_A$,
 then
one may choose $m=2 N$.
\end{theorem}

  \begin{proof}
 Since $A_B$ generates a contraction semigroup on $X$ by Lemma~\ref{lem:ABBsemigroup},  	
	Lemma~\ref{sp_contr} shows that 
the inclusion $i\R\subseteq \rho(A_B)$ and the resolvent estimate 
 will follow from
		a suitable lower bound for $is-A_B,$  $s\in\R$.
To this end, let $s\in\R$ and $x\in \Dom(A_B)$ be fixed and let $y=(is-A_B)x$.
If we let $R=(1+is -A_{-1})^{-1}$ and 
define $x_1=x+R BB^\ast x$, then 
$(is-A_{-1})x_1 
=y-R BB^\ast x\in X$ and hence  $x_1\in \Dom(A)$. Applying~\eqref{eq:NUHautus} and using the identity $B^\ast x_1=(I+B^\ast RB)B^\ast x$ shows that
\eq{
	\norm{x_1}^2
	& \leq M(s) \norm{(is-A)x_1}^2 + m(s) \norm{B^\ast x_1}^2\\
	& \leq M(s) (\norm{y} + \norm{RB}\norm{B^\ast x})^2 + m(s) (1+\norm{B^\ast RB})^2\norm{B^\ast x}^2\\
	& \lesssim M(s) \norm{y}^2 + \bigl(M(s) \norm{ RB}^2+m(s)(1+\norm{B^\ast RB}^2) \bigr)\norm{B^\ast x}^2.
      }
      Since $\norm{B^\ast x}^2\leq \re \iprod{y}{x}\leq  \norm{y}\norm{x}$ by Lemma~\ref{lem:ABBsemigroup}, 
      we may further estimate the norm of $x=x_1-RBB^\ast x$ by
      \eq{
	\norm{x}^2 
	&\lesssim \norm{x_1}^2 + \norm{RB}^2 \norm{B^\ast x}^2 \\
&  \lesssim M(s) \norm{y}^2 + \bigl(M(s) \norm{ RB}^2+m(s)(1+\norm{B^\ast RB}^2) \bigr)\norm{ x}\norm{y} \\
	& \leq M(s) \norm{y}^2 + \eps \norm{x}^2  + \frac{1}{4\eps}\bigl(M(s) \norm{ RB}^2+m(s)(1+\norm{B^\ast RB}^2) \bigr)^2\norm{y}^2,
      }
where $\eps>0$.
We have $\norm{B^\ast RB}\leq \mu(s)$ by assumption, and 
Lemma~\ref{lem:ABBsemigroup} further implies that $\norm{RB}^2\leq \norm{B^\ast RB}\leq \mu(s)$.
      Letting $\eps$ be sufficiently small
we obtain
\eq{
\norm{x}^2 
& \lesssim \left(M(s) +  M(s)^2 \norm{ RB}^4+m(s)^2(1+\norm{B^\ast RB}^2)^2\right) \norm{y}^2\\
& \lesssim \left(M(s)^2 \mu(s)^2+m(s)^2\mu(s)^4\right) \norm{y}^2\\
& \lesssim \left(M(s) \mu(s)+ m(s) \mu(s)^2\right)^2 \norm{(is-A_B)x}^2.
}
Since $x\in \Dom(A_B)$ was arbitrary,  Lemma~\ref{sp_contr} implies that $is\in \rho(A_B)$ and $\norm{(is-A_B)\inv }\lesssim M(s) \mu(s)+ m(s)\mu(s)^2$.

      To prove the other claims,
      assume that $\norm{\RB}\leq N(s)$ and let $s\in\R$ and $x\in \Dom(A)$ be arbitrary.
Using the properties in Remark~\ref{extension}, 
  the claims follow from the estimate
\eq{
  \norm{x}^2
  &=\norm{\RB[is](is-A)x+\RB[is]BB^\ast x}^2\\
  &\leq2\norm{\RB[is]}^2\norm{(is-A)x}^2+2\norm{\RB[is]B}^2 \norm{B^\ast x}^2
   }
and Lemma~\ref{lem:ABBresgrowthlemma}.
\end{proof}

\begin{remark}
\label{rem:HautusSharperBound}
In the case where $\mu$ is a bounded function the resolvent estimate in Theorem~\ref{thm:HautustoResgrowth} takes the form 
$\norm{(is-A_B)\inv}\lesssim M(s)+m(s)$, $s\in\R$.  
As shown in Lemma~\ref{lem:BRBboundProps}, if
 $A$ and $B$ satisfy Assumption~\ref{ass:ABBass}, then condition~\eqref{eq:BRBboundHautusThm} is always satisfied for 
 $\mu(s)=c(1+s^2)$, $s\in\R$, with some $c>0$.
However, in the absence of a more precise bound for $\norm{B^\ast (1+is-A_{-1})\inv B}$
 the proof of Theorem~\ref{thm:HautustoResgrowth} can be modified to derive an alternative resolvent growth bound.  
 Indeed, if the operator $R$
 in the proof is redefined as $R=(I-A_{-1})\inv$ and if $x_1$ is defined as before, then we have $(is-A_{-1})x_1 = y + (is-1)RBB^\ast x$, and
 estimates analogous to those in the original proof
 show that $i\R\subseteq\rho(A_B)$ and 
\eq{
\norm{(is-A_B)\inv}& \lesssim 
M(s) (1+s^2)+ m(s), \qquad s\in\R.
} 
This estimate is in general sharper than what is obtained from Theorem~\ref{thm:HautustoResgrowth} with a quadratic upper bound for $\mu$. 
Finally, 
for general $\mu$
 the estimates in the proof of Theorem~\ref{thm:HautustoResgrowth} also establish the more precise bound 
\eq{
\norm{(is-A_B)\inv}& \lesssim 
M(s)^{1/2}
+M(s) \norm{ (1+is-A_{-1})\inv B}^2+ m(s)\mu(s)^2
}
for $s\in\R$.
This improves on the original estimate if
 $\norm{ (1+is-A_{-1})\inv B}\to 0$ as $\abs{s}\to\infty$. The latter holds, for instance, if $B\in \Lin(U,X)$ is compact.
\end{remark}

Recall that the pair $(B^\ast,A)$ is said to be \emph{exactly observable} if 
 
\eq{
\int_0^\tau \| B^\ast T (t) x \|^2\, dt \ge c_\tau \| x\|^2, \qquad  x\in \Dom(A),
}
for some $\tau>0$ and $c_\tau>0$~\citel{TucWei09book}{Def.~6.1.1}.
If~\eqref{eq:BRBboundHautusThm} is satisfied for a bounded function $\mu$, then Lemma~\ref{lem:Bstaradmissible} and~\citel{Mil12}{Thm.~2.4} imply that the non-uniform Hautus test is satisfied for some bounded functions $M$ and $m$ if and only if the pair $(B^\ast,A)$ is exactly observable. 
In this situation Theorem~\ref{thm:HautustoResgrowth} and the Gearhart--Pr\"uss theorem imply that $\SGB$ is exponentially stable, similarly as in~\cite{Sle74,CurWei06}.

Our next resolvent estimate for a skew-adjoint operator $A$ is based on lower bounds for $B^\ast$
restricted to so-called \emph{wavepackets} of $A$.
Similar conditions have previously been used to study exact observability of the pair $(B^\ast,A)$, for example in~\cite{Chen,RTTT05,Mil12}.

\begin{definition}
  \label{def:WP}
	Let $A$ be a self-adjoint operator on $X$.
  For $s\in\R$ and $\gd(s)>0$ we define $\WP{s}{\gd}$ to be the spectral subspace of $A$ associated with the interval $(s-\gd(s),s+\gd(s))\subseteq \R$.
  The elements $x\in \WP{s}{\gd}$ are called \emph{$(s,\gd(s))$-wavepackets of $A$}.
		If $A$ is skew-adjoint, then we define  $\WP{s}{\gd}$ to be  $\mathrm{WP}_{s,\delta(s)}(-iA).$
\end{definition}

The following proposition presents a sufficient condition for non-uniform stability of $\SGB$ given in terms of the action of $B^*$ on  wavepackets of $A$.
In the case where $\mu$ is a bounded function
and the pair $(B^\ast,A)$ is exactly observable,
it is possible by Lemma~\ref{lem:Bstaradmissible} and~\citel{Mil12}{Cor.~2.17}   to choose $\gd(s)\equiv\gd_0>0$ and $\gg(s)\equiv \gg_0>0$, and our result then implies exponential stability of $\SGB$.

\begin{theorem}
  \label{thm:WPtoRG}
  Let $A$ and $B$ satisfy Assumption~\textup{\ref{ass:ABBass}} and suppose that $A$ is skew-adjoint.
Suppose further that $\mu: \R\to [r_0,\infty)$, $r_0>0$, is such that 
\eq{
\norm{B^\ast (1+is-A_{-1})\inv B}\leq \mu(s), \qquad s\in\R.
}
If there exist
 bounded functions $\gg,\gd: \R\to (0,\infty)$
 such that
  \eqn{
    \label{eq:WPcond}
    \norm{B^\ast x}_U\geq \gg(s) \norm{x}_X,
\qquad  x\in \WP{s}{\gd}, \ s\in\R,
  }
  then $i\R\subseteq \rho(A_B)$ and
  $$\norm{(is-A_B)\inv} \lesssim \frac{\mu(s)^2}{\gg(s)^{2} \gd(s)^{2}},\qquad s\in\R.$$
\end{theorem}

\begin{proof}
    By Lemma~\ref{lem:ABBsemigroup},  $A_B$ generates a contraction semigroup on $X$.
	Thus by Lemma~\ref{sp_contr}
    the claims will follow from
    suitable lower bounds for the operators $is-A_B$, $s\in\R$.
Let $s\in\R$ and $x\in \Dom(A_B)$ be fixed and let $y=(is-A_B)x$.
Further let $P_0\in \Lin(X)$ be the orthogonal projection onto 
 $\WP{s}{\gd}$, and let $P_\infty = I-P_0$. Define 
\eq{
x_0=P_0x, \quad x_\infty=P_\infty x, \quad y_0=P_0y, \quad \text{and}\quad y_\infty = P_\infty y.
}
Since $x_0\in \WP{s}{\gd}$ and $B^\ast x_0=B^\ast x-B^\ast x_\infty$,  \eqref{eq:WPcond} implies that
    \eqn{
      \label{eq:WPRGxEst}
      \norm{x}^2
      &= \norm{x_0}^2 + \norm{x_\infty}^2
      \lesssim \gg(s)^{-2} \bigl(\norm{B^\ast x}^2 + \norm{B^\ast x_\infty}^2\bigr) + \norm{x_\infty}^2 .
    }
We now estimate  $\norm{x_\infty}$ and $\norm{B^\ast x_\infty}$ in turn. 
We begin by introducing the operator $R=(1+is-A_{-1})^{-1}$, noting that $\norm{R}\leq 1$ since $A$ generates a contraction semigroup.
Applying $P_\infty R$ to both sides of the identity $y=(is-A_B)x$ we obtain
\eqn{
\label{eq:WPisAR1}
        (is-A)R x_\infty = Ry_\infty - P_\infty RBB^\ast x,
}
 and hence
\eqn{
\label{eq:WPxinftyform}
    x_\infty = Rx_\infty + Ry_\infty - P_\infty RBB^\ast x.      
}
    Now since $R$ and $P_\infty$ commute, we have $Rx_\infty \in \ran(P_\infty)$, and
    the spectral theorem for self-adjoint operators implies that $\norm{Rx_\infty}\leq \gd(s)\inv \norm{(is-A)Rx_\infty}$. Thus
\eq{
 \norm{x_\infty}
 \lesssim \gd(s)\inv\norm{(is-A)Rx_\infty} +\norm{y}+\norm{ RB}\norm{B^\ast x}.}
By~\eqref{eq:WPisAR1} we have
$$\norm{(is-A)Rx_\infty}\le \norm{Ry_\infty} + \norm{P_\infty RBB^\ast x}\le  \|y\|+ \norm{ RB}\norm{B^\ast x},$$
and therefore 
\begin{equation}\label{eq:xinf_bound}
 \norm{x_\infty}
 \lesssim \gd(s)\inv\big(\norm{y}+\norm{ RB}\norm{B^\ast x}\big).
 \end{equation}

In order to estimate $\norm{B^*x_\infty}$ we begin by observing that, by~\eqref{eq:WPxinftyform},
\begin{equation}\label{eq:B*_bound}
\norm{B^\ast x_\infty}
      \leq \norm{B^\ast R}\norm{x_\infty}+\norm{B^\ast R}\norm{y}+\norm{B^\ast (I-P_0) RB}\norm{B^\ast x}.
      \end{equation}
Since $A$ is skew-adjoint, 
we have
$ B^\ast (1+is-A)\inv=((1-is+A_{-1})\inv B)^\ast $
by Lemma~\ref{lem:RBadjoints}.
Hence the resolvent identity gives
\eq{
\norm{B^\ast R}
= \norm{(1-is+A_{-1})\inv B}
=\norm{R B - 2 (1-is+A)\inv RB} 
\leq 3\norm{R B},
}
and since $\norm{(1+is-A)P_0}\lesssim 1+\gd(s)\lesssim 1$ we see using~\eqref{eq:RBboundwrtBRB} in
 Lemma~\ref{lem:ABBsemigroup}  that
\eq{
\norm{B^\ast (I-P_0) R B}  
&\leq  \norm{B^\ast RB} + \norm{B^\ast R(1+is-A)P_0 R B}  \\
&\lesssim \norm{B^\ast RB} + \norm{ R B}^2  
\lesssim \norm{B^\ast RB}.
}     
Using these estimates and~\eqref{eq:xinf_bound}, we obtain from~\eqref{eq:B*_bound} that
   \eq{     \norm{B^\ast x_\infty}
      &\lesssim \norm{RB}\norm{x_\infty}+\norm{ RB}\norm{y}+\norm{B^\ast RB}\norm{B^\ast x}\\
      &\lesssim \gd(s)\inv\norm{ RB}\norm{y}+ \big(\gd(s)\inv\norm{ RB}^2 + \norm{B^\ast  RB}\big)\norm{B^\ast x} .
    }
Inserting  our bounds for $\|x_\infty\|$ and $\|B^*x_\infty\|$
  into~\eqref{eq:WPRGxEst}, and using the  estimate {$\norm{B^\ast x}^2\leq \norm{x}\norm{y}$ implied  by~\eqref{eq:ABBiprod} in
 Lemma~\ref{lem:ABBsemigroup}, we deduce after a straightforward calculation that }
\eq{
\norm{x}^2 
      &\lesssim \gg(s)^{-2} \bigl(\norm{B^\ast x}^2 + \norm{B^\ast x_\infty}^2\bigr) + \norm{x_\infty}^2 \\
            &          \lesssim
\gd(s)^{-2}\bigl(1+\gg(s)^{-2}\norm{ RB}^2\bigr)\norm{y}^2 \\
& \quad 
+\Bigl(\gg(s)^{-2}\big(1+ \gd(s)^{-2}\norm{ RB}^4 + \norm{B^\ast  RB}^2\big)
+ \gd(s)^{-2}\norm{ RB}^2
\Bigr)\norm{x}\norm{y}.
}
Since $\norm{RB}^2\leq \norm{B^\ast RB}\leq \mu(s) $ by Lemma~\ref{lem:ABBsemigroup} and our assumption 
we obtain, after dropping dominated terms, the estimate 
$$\|x\|^2\lesssim \gamma(s)^{-2}\delta(s)^{-2}\mu(s)\|y\|^2+\gamma(s)^{-2}\delta(s)^{-2}\mu(s)^2\|x\|\|y\|.$$
An application of Young's inequality now yields
$$\|x\|^2\lesssim \gg(s)^{-4}\gd(s)^{-4}\mu(s)^4  \norm{y}^2,$$
and the claim follows from Lemma~\ref{sp_contr}.
\end{proof}

\begin{remark}
In the situation where $\mu$ is a bounded function,  Theorem~\ref{thm:WPtoRG} can alternatively be proved by combining Theorem~\ref{thm:HautustoResgrowth}, Lemma~\ref{lem:Bstaradmissible} and results in~\cite{Mil12}.
Indeed, 
in this case Lemma~\ref{lem:Bstaradmissible} implies that $B^\ast$ is admissible and
by~\citel{Mil12}{Prop.~2.16} the pair $(B^\ast,A)$ satisfies the non-uniform Hautus test~\eqref{eq:NUHautus} for some functions $M$ and $m$ such that
    $M(s)\lesssim \gg(s)^{-2}\gd(s)^{-2}$ and $m(s) \lesssim \gg(s)^{-2}$ for $s\in\R$.
The claim of Theorem~\ref{thm:WPtoRG} then follows from Theorem~\ref{thm:HautustoResgrowth}.
Similarly as in Remark~\ref{rem:HautusSharperBound}, 
the end of 
the proof of Theorem~\ref{thm:WPtoRG} can be modified to  establish the potentially sharper resolvent estimate
\eq{
\norm{(is-A_B)\inv }
& \lesssim 
\nu(s) +   \nu(s)^2\norm{(1+is-A_{-1})\inv B}^2+ \frac{\mu(s)^2}{\gg(s)^2}, \quad s\in\R,
}
where $\nu(s)=\gd(s)^{-1}(1+\gg(s)\inv \norm{(1+is-A_{-1})\inv B}) $.
\end{remark}

\begin{remark}\label{rem:asymp}
It is easy to see from the proofs of Theorems~\ref{thm:HautustoResgrowth} and~\ref{thm:WPtoRG} that if the assumptions
are satisfied only for $\abs{s}\geq s_0$ for some $s_0>0$, then $i\R\setminus (-is_0,is_0)\subseteq \rho(A_B)$ and the resolvent estimate will hold for $\abs{s}\geq s_0$.
The same comment applies to the results in the remainder of this paper.
Since the non-uniform decay rate
is determined only by the resolvent norms for large values of $\abs{s}$, this property
is useful in situations where $i\R\subseteq \rho(A_B)$ is already known or can be shown using other methods.
\end{remark}

\subsection{Criteria for second-order problems}
\label{sec:freqdomSecondOrder}

In this section we focus on studying the resolvent growth for the operator $A_{B}$
defined in~\eqref{defn_of_ab} in the case where the operators
\eq{
  A = \pmat{0 & I \\ -L & 0} \qquad \mbox{and} \qquad B=\pmat{0\\D}
}
on $X$ and $U$, respectively, satisfy the assumptions in Section~\ref{second_order_setup}.
In particular, $L:  H_1\subseteq H\to H$ is a positive self-adjoint and boundedly invertible operator and  $D\in \Lin(U,H_{-1/2})$. 
We shall reformulate the conditions of Theorems~\ref{thm:HautustoResgrowth} and~\ref{thm:WPtoRG}  
in terms of the operators $L$ and $D$. In addition, we shall present further sufficient conditions for non-uniform stability in terms of generalised observability properties of the ``Schr\"odinger group'' generated by $iL$.

In the proofs of our results 
we shall employ a change of variables which transforms $A$ into a block-diagonal operator $A_{\rm diag}$; see for instance the proof of~\citel{Mil12}{Thm.~3.8}.
Recalling that 
$V=H_{1/2}\times H_{1/2}$, we define a unitary operator $Q\in \Lin(V, X)$ by
\eqn{
  \label{eq:WEtransform}
  Q = \frac{1}{\sqrt{2}}\pmat{I&I\\iL^{1/2}&-iL^{1/2}}, \quad \text{with}\quad Q\inv = \frac{1}{\sqrt{2}}\pmat{I&-iL^{-1/2}\\I & iL^{-1/2}}.
}
We then have 
$A=QA_{\rm diag} Q^{-1}$, where
\eq{
 A_{\rm diag} = \pmat{i\Azhalf&0\\0&-i\Azhalf}:   \Dom(A_{\rm diag})\subseteq V\to V 
}
 with domain $\Dom(A_{\rm diag})=H_1\times H_1$.
The following lemma describes the  wavepackets of $A$ in terms of the wavepackets of $\Azhalf$.

\begin{lemma}
  \label{lem:WE_WPs}
Let $L$ and $A$ be as in Section~\textup{\ref{second_order_setup}}
and let $\gd: \R\to (0,\infty)$ be such that  $\sup_{s\in\R}\gd(s)\leq \norm{\Azmhalf }$.
  Then for every $s\in\R$ we have
   \eqn{
    \label{eq:WE_WPform}
  \textup{WP}_{s, \delta(s)}(A)= 
	\Setm{\pmat{w\\ i\sign(s)L^{1/2}w}}{ w \in \textup{WP}_{|s|, \delta(s)}(L^{1/2}) }.
  }
  \end{lemma}

\begin{proof}
Let $s>0$ be fixed.
We have  $\textup{WP}_{s, \delta(s)}(A)=\ran(\chi_{I_{s,\delta (s)}}(-iA))$, where  $I_{s,\delta (s)}=(s-\gd(s),s+\gd(s))$.
  Using the decomposition $A=QA_{\rm diag}Q\inv$ 
	and the upper bound for $\gd$ we see that
   \eq{
    \chi_{I_{s,\delta_s}}(-iA)
        = Q \pmat{\chi_{I_{s,\delta (s)}}(L^{1/2})&0\\0&0}Q^{-1}
    = \frac{1}{\sqrt{2}} \pmat{\chi_{I_{s,\delta (s)}}(L^{1/2})&0\\iL^{1/2}\chi_{I_{s,\delta (s)}}(L^{1/2})&0}Q\inv.
  }
The functional calculus for 
the positive and boundedly invertible operator
$L$ implies that
\[
\chi_{I_{s,\delta (s)}}(\Azhalf)H_{1/2}=\ran\big(\chi_{I_{s,\delta (s)}}\left(\Azhalf\right)\big),
\]
and hence~\eqref{eq:WE_WPform} follows from surjectivity of $Q\inv $.
The proof in the case $s<0$ is analogous.
\end{proof}

The next result is a counterpart of Theorem~\ref{thm:WPtoRG} for damped second-order systems.
We refer to~\citel{Rus75}{Sec.~3} 
for a related result on polynomial stability of second-order systems in the case where $L$ has discrete spectrum and $D\in \Lin(U,H)$.

\begin{theorem}
  \label{thm:WPtoRG_WE}
Let $L$, $D$, $A$ and $B$ be 
 as in Section~\textup{\ref{second_order_setup}} and assume that 
 $\mu_0: \R_+\to [r_0,\infty)$, $r_0>0$, is such that 
\eq{
s\,
\norm{D^\ast ( (1+is)^2+L_{-1})\inv D}\leq \mu_0(s), \qquad s\in\R_+.
}
 	  If there exist bounded functions $\gg_0,\gd_0: \R_+\to(0,\infty)$
such that
\eq{
    \norm{D^\ast w}_U\geq \gg_0(s) \norm{w}_{H},
\qquad  w \in \textup{WP}_{s,\delta_0(s)}(L^{1/2}), s\geq 0,
}
  then $i\R\subseteq \rho(A_{B})$ and
\eq{
\norm{(is-A_{B})\inv} \lesssim \frac{\mu_0(\abs{s})^2}{\gg_0(\abs{s})^2 \gd_0(\abs{s})^2}, \qquad s \in \R.
}
\end{theorem}

\begin{proof}
  If we let $s_0=\min\set{\norm{\Azmhalf },1}$ then $\gs(\Azhalf)\subseteq [s_0,\infty)$.
 Define $\gd:\R\to(0,\infty)$ by
	\begin{equation}\label{def_delta}
		\gd(s)=\frac{s_0\gd_0(\abs{s})}{2\sup_{s\geq 0}\gd_0(s)}, \qquad s \in \R.
	\end{equation}
	Fix $s\in\R$ and let $x \in \textup{WP}_{s,\delta(s)}(A)$ be arbitrary.  Lemma~\ref{lem:WE_WPs} implies that  
	$x=(w,  i\sign(s)L^{1/2}w)$ for some $w \in \textup{WP}_{|s|, \delta(s)}(L^{1/2})$.
Noting that $\Azhalf w \in \textup{WP}_{|s|,\delta(s)}( L^{1/2})$, our assumptions imply that
	\eq{
    \norm{B^\ast x}_U
    = \norm{D^\ast L^{1/2}w}_U
    \geq \gg_0(\abs{s}) \norm{L^{1/2}w}_{H}
    = \frac{\gg_0(\abs{s})}{\sqrt{2}} \norm{x}_{X}.
  }
    Thus the conditions of Theorem~\ref{thm:WPtoRG} hold for $\gd:\R_+\to (0,\infty)$ defined in~\eqref{def_delta} and for $\gg:\R_+\to (0,\infty)$ defined by
$\gg(s)=\gg_0(\abs{s})/\sqrt{2}$ for $ s \in \R$. 
Since \eqref{eq:BRBboundSecondOrder} holds by assumption, the arguments in Section~\ref{second_order_setup} show that $\norm{B^\ast (1+is-A_{-1})\inv B}\lesssim \mu_0(\abs{s})$, $s\in\R$. Thus the claims follow from Theorem~\ref{thm:WPtoRG}.
\end{proof}

The recent literature contains several studies of non-uniform stability
for second-order systems based on observability properties of the
Schr\"odinger group associated with $(D^\ast,iL)$ when $D\in \Lin(U,H)$ is a bounded operator.
In particular, the Hautus-type condition~\eqref{eq:SchObs} in the following proposition was used as a starting point for deriving resolvent estimates for $A_B$ in~\citel{AnaLea14}{Thm.~2.3} in the case of constant parameters $M_0$ and $m_0$, and with variable parameters in~\citel{JolLau20}{App.~B}; see also~\cite{LauLea21}. In both cases the results were used to prove non-uniform stability of wave equations with viscous damping.
The following result generalises the results on resolvent growth in~\citel{JolLau20}{App.~B} to operators $L$ with possibly non-compact resolvent and operators~$D\in \Lin(U,H_{-1/2})$.

\begin{proposition}
  \label{prp:SchObstoRG}
Let $L$, $D$, $A$ and $B$ be 
 as in Section~\textup{\ref{second_order_setup}}.
Moreover, let $M_0: \R_+\to(0,\infty)$ and $m_0 : \R_+\to [r_0,\infty), r_0>0,$ 
be such that
\eqn{
\label{eq:SchObs}
\norm{w}_H^2 \leq M_0(s) \norm{(s^2-L)w}_H^2 + m_0(s) \norm{D^\ast w}_U^2, \quad  w\in H_1,\  s\geq 0,
}
and define $\eta:=\inf_{s\geq 0}M_0(s)(1+s)^2>0.$  Then the conditions of Theorem~\textup{\ref{thm:WPtoRG_WE}} are satisfied for the functions $\gg_0,\gd_0: \R_+\to(0,\infty)$
defined by
\eqn{
\label{eq:SchObsWPparam}
\gd_0(s) = \frac{\min\set{\sqrt{\eta},1/2}}{\sqrt{2 M_0(s)}(1+s)}
\qquad \mbox{and} \qquad
\gg_0(s) =\frac{1}{\sqrt{2m_0(s)}}
} 
for $s\ge0$. If, in addition,
 $\mu_0: \R_+\to [r_0,\infty)$, $r_0>0$, is such that 
\eqn{\label{eq:mu0}
s\,
\norm{D^\ast ( (1+is)^2+L_{-1})\inv D}\leq \mu_0(s), \qquad s\in\R_+,
}
then
$i\R\subseteq \rho(A_{B})$ and 
\eq{
\norm{(is-A_{B})^{-1}}\lesssim (1+s^2) M_0(\abs{s})m_0(\abs{s})\mu_0(\abs{s})^2, \qquad s\in\R.
}
\end{proposition}

\begin{proof}
Let $s \ge 0$.
  The function
  $\gd_0$ in~\eqref{eq:SchObsWPparam} is bounded and for every  $r\in (s-\gd_0(s),s+\gd_0(s))$ we have
  \eq{
    \abs{s^2-r^2}
    =\abs{s-r}\abs{s+r}
    \leq \frac{\min\set{\sqrt{\eta},1/2}(2s+\gd_0(s))}{\sqrt{M_0(s)}(1+s)} \leq \frac{1}{\sqrt{2M_0(s)}}.
  }
If $w \in \textup{WP}_{s,\delta_0(s)}( L^{1/2})$,
this estimate
 and the 
functional calculus for $L$
 imply that $\norm{(s^2-L)w}^2\leq (2M_0(s))\inv\norm{w}^2$. Hence~\eqref{eq:SchObs} yields
  \eq{
    \norm{D^\ast w}^2 \geq \frac{1}{2m_0(s)}\norm{w}^2.
  }
  Since $s\geq 0$ and the wavepacket $w$ were arbitrary, the conditions of Theorem~\ref{thm:WPtoRG_WE} are satisfied for the functions $\gd_0$ and $\gg_0$ defined by~\eqref{eq:SchObsWPparam}, and the remaining claims follow from Theorem~\ref{thm:WPtoRG_WE}.
\end{proof}

Our result shows in particular that if \eqref{eq:SchObs} holds for constant functions $M_0$ and $m_0$ and if \eqref{eq:mu0} holds for a bounded function $\mu_0$, then   $\norm{(is-A_B)\inv}\lesssim 1+s^2$ for $s\in\R$.
The same result was previously proved for $D\in\Lin(U,H)$ in~\citel{AnaLea14}{Thm.~2.3}, and we shall discuss this result further in the context of damped waves in Section~\ref{sec:2Dwaves} below. A result closely related to Proposition~\ref{prp:SchObstoRG} and, in particular, allowing non-constant functions $M_0$ and $m_0$ was proved in \cite[Prop.~B.3]{JolLau20}, once again  in the simpler setting where $D\in\Lin(U,H)$; see also~\cite{LauLea21}. Proposition~\ref{prp:SchObstoRG} not only generalises and extends these earlier results, it moreover allows us to see that  observability conditions of the type considered in \eqref{eq:SchObs} and in~\citel{JolLau20}{App.~B} serve as sufficient conditions for the wavepacket condition in Theorem~\ref{thm:WPtoRG}.
Finally, in the case where $\mu_0$ is a bounded function, Lemma~\ref{lem:Bstaradmissible} and~\citel{Mil12}{Prop.~2.16} show that the same conditions further imply the non-uniform Hautus test in Definition~\ref{def:NUHautus} 
for the associated first-order equation.

We conclude this section by presenting an equivalent characterisation for the non-uniform Hautus test of pairs $(B^\ast,A)$ stemming from second-order systems.

  \begin{proposition}
    \label{prp:HautustoResgrowth_WE}
Let $L$, $D$, $A$ and $B$ be 
 as in Section~\textup{\ref{second_order_setup}}.
If $M_0, m_0 : \R_+\to[r_0,\infty)$, $r_0>0,$ are such that
    \eqn{
      \label{eq:NUHautus_WE}
      \norm{w}_{H}^2\leq  {M_0}(s) \norm{(s-L^{1/2})w}_{H}^2 +  {m_0}(s)\norm{D^\ast w}_U^2
    }
    for all $w\in H_{1/2}$ and $s\geq 0$,
    then $(B^\ast,A)$ satisfies the non-uniform Hautus test
for some function $M$ such that 
     $M(s)\lesssim {M_0}(\abs{s})+ {m_0}(\abs{s})$ and for $m$ given by
    $m(s)=4  {m_0}(\abs{s})$, $s\in\R$. 
If, in addition,
 $\mu_0: \R_+\to [r_0,\infty)$, $r_0>0$, is such that 
\eq{
s\,
\norm{D^\ast ( (1+is)^2+L_{-1})\inv D}\leq \mu_0(s), \qquad s\in\R_+,
}
then
$i\R\subseteq \rho(A_{B})$ and
    \[
\norm{(is-A_{B})\inv} \lesssim {M_0}(\abs{s})\mu_0(\abs{s})+{m_0}(\abs{s})\mu_0(\abs{s})^2, \qquad s \in \R.
\]

Conversely, if $(B^\ast,A)$ satisfies the non-uniform Hautus test for some $M,m: \R\to [r_0,\infty)$, $r_0>0$,
    then~\eqref{eq:NUHautus_WE} holds for
$M_0$ and $m_0$ defined by
    ${M_0}(s)= M(s)$ and
    ${m_0}(s)= m(s)/2$ for $s\geq 0$.
  \end{proposition}

\begin{proof}
Since $\Azhalf $ is boundedly invertible by definition, similarly as in~\citel{Mil12}{Thm.~3.8} the decomposition $A=QA_{\rm diag}Q\inv$ with $Q$ as in~\eqref{eq:WEtransform} implies that~\eqref{eq:NUHautus} holds if and only if 
    \eq{
      \norm{y_1}_{H}^2+\norm{y_2}_{H}^2 
      &\leq M(s)\left( \norm{(s-L^{1/2})y_1}_{H}^2 + \norm{(s+L^{1/2})y_2}_{H}^2 \right)
      \\
      & \quad+ \frac{m(s)}{2} \norm{D^\ast (y_1-y_2)}_U^2
    }
    for all $y_1,y_2\in H_{1/2}$ and $s\in\R$. 
		Thus if~\eqref{eq:NUHautus} holds, then choosing
		$y_2=0$ and $s\geq 0$ in the above inequality implies the last claim of the proposition.

To prove the first claim, let $s\geq 0$ and $y_1,y_2\in H_{1/2}$ be arbitrary.
Our assumptions imply that $\Azhalf $ is boundedly invertible and $D^\ast L^{-1/2}\in \Lin (H, U)$. Thus the estimates
$\norm{\Azhalf (s+\Azhalf )\inv}\leq 1$, $\norm{(s+\Azhalf )\inv}\leq \norm{\Azmhalf }\inv$ and~\eqref{eq:NUHautus_WE}  imply that
    \eq{
      \MoveEqLeft \norm{y_1}_{H}^2 +\norm{y_2}_{H}^2
      \leq M_0(s)\norm{(s-L^{1/2})y_1}_{H}^2 + m_0(s)\norm{D^\ast y_1}_U^2 + \norm{y_2}_{H}^2\\
      &\leq M_0(s)\norm{(s-L^{1/2})y_1}_{H}^2 + 2m_0(s)\norm{D^\ast (y_1-y_2)}_U^2 \\
      &\quad
      + 2 m_0(s)\norm{D^\ast L^{-1/2}}^2\norm{L^{1/2}y_2}_{H}^2 +\norm{y_2}_{H}^2\\
      &\leq M_0(s)\norm{(s-L^{1/2})y_1}_{H}^2 + 2m_0(s)\norm{D^\ast (y_1-y_2)}_U^2 \\
      &\quad
      + \bigl(2 m_0(s)\norm{D^\ast L^{-1/2}}^2 + \norm{L^{-1/2}}^{-2}\bigr) \norm{(s+L^{1/2})y_2}_{H}^2.
    }
    Thus~\eqref{eq:NUHautus} holds for $s\geq 0$ with  $M$ and $m$ as described in the claim.
    For $s<0$ we get 
an analogous estimate
 by applying~\eqref{eq:NUHautus_WE} to $\norm{y_2}^2$ with $s$ replaced by $\abs{s}$, and combining the estimates shows that~\eqref{eq:NUHautus} holds for $s\in\R$ with functions
$M,m:\R\to [r_0,\infty)$ satisfying
$m(s)=4  {m_0}(\abs{s})$ and
     $M(s)\lesssim {M_0}(\abs{s})+ {m_0}(\abs{s})$ for $s\in\R$.
Finally,  
as shown in Section~\ref{second_order_setup},
the fact that~\eqref{eq:BRBboundSecondOrder} holds by assumption
implies 
$\norm{B^\ast (1+is-A_{-1})\inv B}\lesssim \mu_0(\abs{s})$, $s\in\R$,
and thus the remaining claims follow from Theorem~\ref{thm:HautustoResgrowth}.
\end{proof}

\section{Time-domain conditions for non-uniform stability}
\label{sec:timedomresults}

\subsection{Conditions for first-order problems}

In this section we present sufficient conditions for polynomial stability of the semigroup $\SGB$ generated by $A_B$ in terms of the following generalised observability concept. Related generalisations of exact observability have previously been used in~\cite{AmmTuc01,AmmNic15,AmmBch17} to study non-uniform stability of damped second-order systems.

\begin{definition}
  \label{def:NUObs}
Let  $(T (t))_{t \ge 0}$ be a contraction semigroup on $X,$ with generator $A$,
and let $C\in \Lin(X_1,U)$, where $X$ and $U$ are Hilbert spaces.
  The pair $(C,A)$ is said to be \emph{non-uniformly observable} (\emph{with parameters} $\beta \geq 0$ \emph{and} $\tau>0$)  if there exists $c_\tau>0$ such that
  \eqn{
    \label{eq:NUObs}
 c_\tau \norm{(I-A)^{-\gb}x}^2_X\leq \int_0^\tau \norm{C T (t) x}^2_U \,dt, \qquad  x\in \Dom(A).
  }
\end{definition}

Note that by~\cite[Corollary]{Kat61} the norm $\norm{(I-A)^{-\gb}x}$ in~\eqref{eq:NUObs} can be replaced by $\norm{(\gl_0-A)^{-\gb}x}$ for any fixed $\gl_0\in\rho(A)\cap\overline{\mathbb{C}_+}$ (and a possibly different $c_\tau>0$), and in particular the choice $\gl_0=0$ is possible if $0\in\rho(A)$.
By injectivity of $(I-A)^{-\gb}$, non-uniform observability also implies \emph{approximate observability} of the pair $(C,A)$ in the sense that if $C T (t)x= 0$ for all $ t\in[0, \tau]$, then necessarily $x=0$.
The case $\gb=0$ corresponds to exact observability of the pair $(C,A)$.

Throughout this section we consider the setting of Section~\ref{sec:wellposedness} in the case where $B$ is a bounded operator. In particular, $A: \Dom(A)\subseteq X\to X$ generates a contraction semigroup $\SG$ on a Hilbert space $X$ and $B\in \Lin(U,X)$, where $U$ is another Hilbert space.
In this situation the generator of the semigroup $\SGB$ is  $A_B=A-BB^\ast$ with $\Dom(A_B)=\Dom(A)$.
The following consequence of the Heinz inequality for dissipative operators due to Kato
will be important for the arguments in this section.
The result in particular allows us to compare fractional powers of $I-A$ and $I- A_B$.

\begin{theorem}[{\cite[Corollary]{Kat61}}]
\label{kato}
Let $A_1$ and $A_2$ be  generators of contraction semigroups on $X$, 
and suppose that $\Dom(A_1) \subseteq \Dom(A_2)$ and $\|A_2x\|\lesssim \|A_1x\|$ for all $x \in \Dom(A_1).$
Then for every $\alpha \in [0,1]$ we have
$\Dom((-A_1)^\alpha) \subseteq \Dom((-A_2)^\alpha)$ and $\|(-A_2)^\alpha x\| \lesssim \|(-A_1)^\alpha x\|$ for all $x \in \Dom((-A_1)^\alpha).$
\end{theorem}

We shall also require the following lemma.
A similar result for second-order systems of the form in Section~\ref{second_order_setup} (and a possibly unbounded operator~$B$) was presented in~\citel{AmmTuc01}{Lem.~4.1}.

\begin{lemma} \label{lem:OutputMaps}
Let $A: \Dom(A)\subseteq X\to X$ be a skew-adjoint operator generating a unitary group $\SG$  and let 
$B\in \Lin(U,X)$.
	\begin{itemize}
\item [(a)]  
    For every $\tau>0$ there exists $C_\tau>0$ such that
\eqn{
\label{eq:OutputMapsIneq}
      \int_0^\tau \norm{B^\ast T_B(t)x}^2 \,dt
      \leq
      \int_0^\tau \norm{B^\ast T (t)x}^2\,dt
      \leq
      C_\tau
      \int_0^\tau \norm{B^\ast T_B(t)x}^2\,dt
}
		 for all  $x\in X.$ 
  Moreover, the second inequality in~\eqref{eq:OutputMapsIneq} remains valid when $A$ is merely a generator of a contraction semigroup. 
		\item [(b)] 
		The pair $(B^\ast,A)$ is non-uniformly observable with parameters $ \gb\in[0,1]$ and $\tau>0$ if and only if $(B^\ast, A_B)$ is non-uniformly observable with the same parameters $\beta$ and $\tau $.
		\end{itemize}
		\end{lemma}

\begin{proof}
We begin by the second statement in (a).
Suppose therefore that $\SG$ is a contraction semigroup and
let $\tau >0$ be fixed.  Define  $\Psi$, $\Psi_B \in \Lin (X, \Lp[2](0,\tau;U))$ by
    $\Psi x: = B^\ast T(\cdot)x$ and
$\Psi_B x: = B^\ast T_B(\cdot) x$
  for all $x\in X$.
  If we define $\mathbb{F}_\tau \in \Lin(\Lp[2](0,\tau;U))$ by
  \eq{
    (\mathbb{F}_\tau u)(t) = \int_0^t B^\ast T(t-s)B u(s)\,ds, \qquad u\in \Lp[2](0,\tau;U),
  }
  then the variation of parameters formula for $\SGB$ implies that 
    \eq{
      (I+\mathbb{F}_\tau)\Psi_B =\Psi.
    }
Hence the second inequality in~\eqref{eq:OutputMapsIneq} holds with $C_\tau = (1+\norm{\mathbb{F}_\tau})^2$.
To complete the proof of (a), assume that
$A$ is skew-adjoint in which case $\SG$ is a unitary group. Direct computations may be used to show that $\re \iprod{\mathbb{F}_\tau u}{u}\geq 0$ for all $u\in \Lp[2](0,\tau;U)$, and therefore
    the operator $I+\mathbb{F}_\tau$ is boundedly invertible with $\norm{(I+\mathbb{F}_\tau)\inv}\leq 1$. 
    This implies the first inequality in~\eqref{eq:OutputMapsIneq} and thus completes the proof of (a).

To prove (b), fix $\beta \in [0,1]$ and $\tau >0$. Both $(A-I)^{-1}$ and $(A_B-I)^{-1}$ are bounded operators generating contraction semigroups on $X$. 
Since $\norm{(A-I)\inv x}\lesssim \norm{(A_B-I)\inv x}\lesssim \norm{(A-I)\inv x}$ for all $x\in X$,  Theorem~\ref{kato} implies that
 $\|(I-A)^{-\beta} x \|\lesssim \|(I-A_B)^{-\beta} x \| \lesssim \|(I-A)^{-\beta} x \|$ for all $ x \in X$. 
Now the claim follows directly from (a).
\end{proof}

As our first main result of this section we show that
if $\Dom(A^\ast)=\Dom(A)$ and $B\in \Lin(U,X)$, then 
non-uniform observability of $(B^\ast,A)$ implies polynomial stability of the semigroup $\SGB$ generated by $A_B $.
The theorem is similar in nature to the results presented 
in~\cite{AmmTuc01,AmmBch17} and \citel{AmmNic15}{Ch.~2}. In particular, these references introduce generalised versions of exact observability of $(B^\ast,A)$
for second-order equations of the form in Section~\ref{second_order_setup}, and 
deduce non-uniform stability of the semigroup $\SGB$.
If $\beta=0$ in our result, then the pair $(B^\ast,A)$ is exactly observable and we obtain exponential stability, similarly as in~\cite{Sle74}.

\begin{theorem} \label{thm:NUObstoResgrowth}
Let  $A$ 
be the generator of a contraction semigroup on $X$ such that $\Dom(A^\ast)=\Dom(A)$, and let 
 $B\in \Lin(U,X)$. 
If the pair $(B^\ast,A)$ is non-uniformly observable with parameters  $\beta \in [0,1]$ and $\tau>0$, then 
$i\R\subseteq \rho(A_B)$ and
  \eq{
    \|(is-A_B)^{-1}\| \lesssim 1+\abs{s}^{2\gb}, \qquad s \in \mathbb R.
  }
  In particular, if $0<\gb\leq 1$ then the semigroup $\SGB$ is polynomially stable and there exists a constant $C>0$ such that
\eqn{\label{eq:poly_decay}
    \norm{T_B(t)x}\leq \frac{C}{t^{1/(2\gb)}}\norm{A_B x}, \qquad  
x\in \Dom(A_B), \ t> 0.
}
If $\beta=0$ then the semigroup $\SGB$ is exponentially stable.
\end{theorem}

\begin{proof}
Let $\beta \in [0,1]$ and $\tau >0$ be such that~\eqref{eq:NUObs} holds for some $c_\tau>0$.
By Lemma~\ref{lem:ABBsemigroup} the semigroup $(T_B(t))_{t \ge 0}$  is contractive and $1\in\rho(A_B)$. 
Moreover, both $(A-I)^{-1}$ and $(A_B-I)^{-1}$ are bounded operators generating contraction semigroups on $X$. 
Since $\norm{(A_B-I)\inv x}\lesssim \norm{(A-I)\inv x}$ for all $x\in X$,
we have $\|(I-A_B)^{-\beta} x \| \lesssim \|(I-A)^{-\beta} x \|$ for all $ x \in X$, by Theorem~\ref{kato}. 
Let $\gl \in\C_+$ and $x\in \Dom(A)$.
The previous estimate together with non-uniform observability of $(B^\ast,A)$, Lemma~\ref{lem:OutputMaps}(a) and the estimate $ \re \iprod{(\gl-A_B)z}{z}\geq\norm{B^\ast z}^2, z\in \Dom(A) $, imply that 
\eq{
\norm{(I-A_B)^{-\gb}x}^2
    &\lesssim
    \norm{(I-A)^{-\gb}x}^2
    \leq \frac{C_\tau}{c_\tau}\int_0^\tau\norm{B^\ast T_B(t)x}^2\,dt\\
    &\leq \frac{C_\tau}{c_\tau}\int_0^\tau\re\iprod{T_B(t)(\gl-A_B) x}{T_B(t)x}\,dt . 
}
Since $\Dom(I-A_B^\ast) = \Dom(A)=\Dom(I-A_B),$ Theorem~\ref{kato} gives $\Dom((I-A_B^\ast)^\gb)=\Dom( (I-A_B)^\gb)$, and in particular $(I-A_B^\ast)^\gb(I-A_B)^{-\gb}\in \Lin(X)$.
Hence if $\gl\in \C_+$ and $x\in \Dom( (-A_B)^{1+2\gb}) $ are arbitrary, 
the above estimate and contractivity of $\SGB$ imply that
\eq{
  \MoveEqLeft[1]\|x\|^2 \lesssim  \frac{C_\tau}{c_\tau} \int_0^\tau\re\iprod{T_B(t)(\gl-A_B) (I-A_B)^\gb x}{T_B(t)(I-A_B)^\gb x}\,dt
	\\
  &= \frac{C_\tau}{c_\tau}\int_0^\tau\re\iprod{(I-A_B^\ast)^\gb (I-A_B)^{-\gb} T_B(t)(\gl-A_B) (I-A_B)^{2\gb} x}{T_B(t) x}\,dt \\
  &\le \frac{\tau C_\tau}{c_\tau}\norm{(I-A_B^\ast)^\gb(I-A_B)^{-\gb}}\norm{(\gl-A_B) (I-A_B)^{2\gb} x}\norm{x} .
}
Since $\C_+\subseteq \rho(A_B)$ we in particular obtain
\eq{
  \sup_{0<\re\gl<1}\; \norm{(\gl -A_B)^{-1} (I-A_B)^{-2\gb}}<\infty.
}
Thus  $\norm{(\gl -A_B)^{-1}}\lesssim 1+ \abs{\gl}^{2\gb}$ for $0<\re\gl<1$ by \citel{LatShv01}{Lem.~3.2}. In particular, 
the inequality $\norm{(\gl -A_B)^{-1}}\geq 1/\dist(\gl,\gs(A_B))$ implies that $i\R\subseteq \rho(A_B)$ and 
$\norm{(is-A_B)\inv}\lesssim 1+\abs{s}^{2\gb}$ for $s\in\R$.
Finally, for \ieq{\gb\in(0,1]}, the estimate \eqref{eq:poly_decay} follows from Theorem~\ref{thm:sg_decay}, and for $\gb=0$ the claim follows from the Gearhart--Pr\"uss theorem.
\end{proof}

As shown in the following proposition, non-uniform observability of $(B^\ast,A)$ can also be characterised in terms of the orbits of the semigroup $\SGB$.

  \begin{proposition}
    \label{lem:NUObsAltChar}
Let $A$ be skew-adjoint and $B\in \Lin(U,X)$. 
    The pair
    $(B^\ast,A)$
    is non-uniformly observable 
    with parameters $\gb\in[0,1]$, $\tau>0$
    if and only if 
    \eqn{
      \label{eq:NUObsAltChar}
      \norm{(I-A)^{-\gb} x}^2\lesssim \norm{x}^2- \norm{T_B(\tau)x}^2, \qquad x\in X.
    }
    In particular, if~\eqref{eq:NUObsAltChar} holds for some $ \gb\in[0, 1]$ and $\tau>0$, then
    $i\R\subseteq \rho(A_B)$ and $\norm{\RB}\lesssim 1+\abs{s}^{2\gb}$ for $s\in \R$.
  \end{proposition}

\begin{proof}
Fix $\gb\in[0, 1]$ and $\tau>0$. 
As in the proof of Lemma~\ref{lem:OutputMaps},
we have
    $\norm{(I-A)^{-\gb}x}\lesssim \norm{(I-A_B)^{-\gb} x}\lesssim \norm{(I-A)^{-\gb}x}$ for all $x\in X$ by Theorem~\ref{kato}.
    For every $x\in \Dom(A) = D(A_B)$ we have
    \eq{
       2\int_0^\tau \norm{B^\ast T_B(t)x}^2\,dt
      &= 2\int_0^\tau \re\iprod{(-A+BB^\ast) T_B(t)x}{T_B(t)x}\,dt\\
       &=-\int_0^\tau \frac{d}{dt} \norm{ T_B(t)x}^2\,dt
      = \norm{x}^2-\norm{T_B(\tau)x}^2.
    }
Thus~\eqref{eq:NUObsAltChar} is equivalent to  non-uniform observability of the pair $(B^\ast,A_B)$ with parameters $\beta$ and $\tau$, which in turn is equivalent to  non-uniform observability of $(B^\ast,A)$ with parameters $\beta$ and $\tau$ by Lemma~\ref{lem:OutputMaps}(b).
If~\eqref{eq:NUObsAltChar} holds, then 
   non-uniform observability of $(B^\ast,A)$ and Theorem~\ref{thm:NUObstoResgrowth} imply 
    that $i\R\subseteq \rho(A_B)$ and $\norm{\RB}\lesssim 1+\abs{s}^{2\gb}$ for $s\in\R$.
\end{proof}

Note that by Theorem~\ref{kato} the norm $\norm{(I-A)^{-\gb} x}$ on the left-hand side of~\eqref{eq:NUObsAltChar} can be replaced by $\norm{(I-A_B)^{-\gb} x}$, or by $\norm{(-A)^{-\gb}x}$ if $0\in\rho(A)$.
Estimates similar to~\eqref{eq:NUObsAltChar} have been used in the literature in order to prove polynomial decay rates for $\SGB$ based on discrete-time iterations, 
especially for  
damped wave equations~\cite{Rus75} and coupled partial differential equations~\cite{Rauch,Duy07}. 
In particular,
in the special case $\gb=1/2$ condition~\eqref{eq:NUObsAltChar} is equivalent to the observability estimate~\citel{Duy07}{Eq.~(39)}.
Thus Theorem~\ref{thm:NUObstoResgrowth} improves and generalises the stability result in~\citel{Duy07}{Sec.~5} in the case where $A$ is skew-adjoint. 
Finally, if $A$ generates a contraction semigroup  and $B \in \Lin (U,X)$, then non-uniform observability of 
    $(B^\ast,A)$ with parameters $\beta\in[0,1]$ and $\tau>0$ implies~\eqref{eq:NUObsAltChar}

\subsection{Time-domain conditions for second-order problems}

In this section we study  non-uniform observability 
for second-order systems of the form
\eqn{
\label{eq:SecondOrderTimeDom}
  \ddot{w}(t)+L w(t)+D D^\ast\dot{w}(t) =0,\qquad t\ge0.
}
Throughout the section, $L$, $D$, $A$ and $B$ are as in Section~\ref{second_order_setup}. In the proofs of our results we also make use of the operator $\abs{A_{\rm diag}}: \Dom(\abs{A_{\rm diag}})\subseteq X\to X$ defined by
\eqn{
\label{eq:Absdiag}
\abs{A_{\rm diag}}=\pmat{\Azhalf &0\\0&\Azhalf }, \qquad \Dom(\abs{A_{\rm diag}}) = \Dom(A).
}
For second-order systems the concept of non-uniform observability in Definition~\ref{def:NUObs} has the following alternative characterisation.

\begin{proposition}
  \label{prp:NUObsAltWE}
Let $L$, $D$, $A$ and $B$ be 
 as in Section~\textup{\ref{second_order_setup}}.
The pair $(B^\ast,A)$ is non-uniformly observable
  with parameter $\gb\in [0,1]$ and $\tau>0$
    if and only if
\eq{
  \norm{L^{(1-\gb)/2}w_0}_{H}^2+ \norm{L^{-\gb/2}w_1}_{H}^2\lesssim \int_0^\tau\norm{D^\ast \dot{w}(t)}_U^2 \,dt ,
}
where $w$ is the (classical) solution of
\eq{
  \ddot{w}(t)+L w(t)=0, \qquad w(0)=w_0\in H_1,  \quad \dot{w}(0)=w_1\in H_{1/2}.
}
\end{proposition}

\begin{proof}
  Fix $\gb\in[0, 1]$ and $\tau>0$.
Since $0\in \rho(A)$, the norm $\norm{(I-A)^{-\gb}x}$ in~\eqref{eq:NUObs} can be replaced by $\norm{(-A)^{-\gb}x}$.
If $\abs{A_{\rm diag}}$ is defined as in~\eqref{eq:Absdiag}, then
  for $x=(x_1,x_2)\in X=H_{1/2}\times H$ we have
  \eq{
    \norm{-A\inv x}^2_X
    &= \norm{L\inv x_2}_{H_{1/2}}^2 + \norm{x_1}_{H}^2
    = \norm{\abs{A_{\rm diag}}\inv x}_X^2.
  }
Thus Theorem~\ref{kato} implies that
    $\norm{(-A)^{-\gb}x}\lesssim \norm{\abs{A_{\rm diag}}^{-\gb} x}\lesssim \norm{(-A)^{-\gb}x}$ for all $x\in X$, 
  and hence
  \eq{
    \norm{(-A)^{-\gb}x}^2\lesssim \norm{L^{(1-\gb)/2}x_1}_{H}^2 + \norm{L^{-\gb/2}x_2}_{H}^2
    \lesssim\norm{(-A)^{-\gb}x}^2
  }
  for all $x=(x_1,x_2)\in X$.
The claims now follow from the fact that for $x=(w_0,w_1)\in \Dom(A)= H_1\times H_{1/2} $ we have $T(t)x\in \Dom(A)$ and $B^\ast T(t)x=D^\ast \dot{w}(t)$ for all $t\geq 0$.
\end{proof}

We conclude this section by studying the damped second-order equation~\eqref{eq:SecondOrderTimeDom}
for  damping operators $D\in \Lin(U,H)$ satisfying
\begin{equation}\label{compar}
\norm{\Azmhalf[\ga] w}\lesssim \norm{D^\ast w}\lesssim \norm{\Azmhalf[\ga]w}, \qquad w\in H,
\end{equation}
for some \ieq{\ga\in(0,1]}.
Non-uniform stability of such equations was studied in~\cite{LiuZha15}, and in~\cite{DelPat21} in a slightly more general setting. 
The assumptions on $D$ are satisfied in particular for the damping operator $D=\Azmhalf[\ga]$ in the wave and beam equations in~\citel{DelPat21}{Sec.~15}, as well as for the damped Rayleigh plate studied in~\citel{LiuZha15}{Sec.~3}.
We shall show that such damping implies non-uniform observability in the sense of Definition~\ref{def:NUObs}.
In particular, the following proposition reproduces the result of~\citel{LiuZha15}{Thm.~2.1} for a symmetric damping operator of the form $D D^\ast$ and for $\ga\in(0,1]$. The degree of stability was shown in~\citel{LiuZha15}{Sec.~3} to be optimal for a class of systems with a diagonal~$L$.

\begin{proposition}
  \label{prp:WEFracDamping}
Let $L$, $D$, $A$ and $B$ be 
 as in Section~\textup{\ref{second_order_setup}}
with
 $D\in \Lin(U,H)$ such that~\eqref{compar} holds
for some constant $\ga\in(0,1]$.
   Then 
   the pair $(B^\ast,A)$
   is non-uniformly observable 
   with parameter $\gb=\ga$ and for any $\tau> (\pi+2\pi^3)\norm{\Azmhalf }\inv$.
Moreover, the semigroup $\SGB$ generated by $A_B$ is polynomially stable and there exists a constant $C>0$ such that 
  \eq{
    \norm{T_B(t)x}\leq \frac{C}{t^{1/(2\ga)}}\norm{A_Bx} , \qquad  x\in \Dom(A_B), \ t>0.
  }
\end{proposition}

\begin{proof}
  We begin by showing that if we define $ (0,I)\in \Lin(X,H)$, then the pair $( (0,I),A)$ is exactly observable
for any $\tau> (\pi+2\pi^3)\norm{\Azmhalf }\inv$.
  To prove this, let $\gd_0=\norm{\Azmhalf }$. Then Lemma~\ref{lem:WE_WPs} shows that every non-trivial $(s,\gd_0)$-wavepacket $x$ of $A$ 
  has the form $x=(w,i\sign(s)\Azhalf w)$ where $w$ is a $(\abs{s},\gd_0)$-wavepacket of $\Azhalf$, and for such $x$ we have
\eq{
  \norm{(0,I) x}_{H} = \norm{L^{1/2}w}_{H}
  = \frac{1}{\sqrt{2}}\norm{x}_X.
}
Since $\norm{(0,I)}=1$, it follows from~\citel{Mil12}{Cor.~2.17} that 
the pair $( (0,I),A)$ is exactly observable for $\tau>
 (\pi+2\pi^3)\norm{\Azmhalf }\inv$.  

If $\abs{A_{\rm diag}}$ is defined as in~\eqref{eq:Absdiag}, then 
$\abs{A_{\rm diag}}\inv$ commutes with $A$, and thus the same is true for $\abs{A_{\rm diag}}^{-\ga}$.
Similarly as in the proof of Proposition~\ref{prp:NUObsAltWE} we have
  $\norm{(-A)^{-\ga}x}\lesssim \norm{\abs{A_{\rm diag}}^{-\ga}x}\lesssim \norm{(-A)^{-\ga}x}$ for all $x\in X$.
We may write 
$B^\ast=(0,D^\ast)=(0,D^\ast\Azhalf[\ga])\abs{A_{\rm diag}}^{-\ga}$, 
where the operator $D^\ast\Azhalf[\ga] $ is  bounded below by assumption. 
Thus, for any fixed $\tau>(\pi+2\pi^3) \norm{\Azmhalf }\inv$ and for all $x\in \Dom(A)$,  exact observability of $( (0,I),A)$ implies that
\eq{
  \int_0^\tau \hspace{-1ex} \norm{B^\ast T(t)x}_U^2\,dt
  &\gtrsim \hspace{-.6ex}\int_0^\tau \hspace{-1ex} \norm{(0,I) T(t) \abs{A_{\rm diag}}^{-\ga} x}_{H}^2\,dt\\
    &\gtrsim \norm{\abs{A_{\rm diag}}^{-\ga}x}_X^2
    \gtrsim \norm{(-A)^{-\ga}x}_X^2.
}
 Theorem~\ref{kato} now implies that the pair $(B^\ast,A)$ is non-uniformly observable with parameter $\gb=\ga$ and with the chosen 
$\tau> (\pi+2\pi^3)\norm{\Azmhalf }\inv$.
Since $A$ is skew-adjoint, the remaining claims follow from Theorem~\ref{thm:NUObstoResgrowth}.
\end{proof}

\section{Optimality of the decay rates}
\label{sec:Optimality}

In this section we investigate the optimality
of our non-uniform decay estimates for the damped semigroup $\SGB$. In particular,
we present lower bounds for $\norm{T_B(\cdot) A_B \inv}$,
which in turn
impose a restriction on the growth of $N\inv(t)$ as $t\to\infty$ in  estimate~\eqref{eq:NUStabFreq}. Our results will allow us to
 show that our resolvent estimates and the resulting non-uniform decay rates 
are optimal or near-optimal in several situations of interest, including various PDE models to be explored in Section~\ref{sec:PDEmodels}.
 As we shall see in Section~\ref{sec:2DwaveWPsuboptimality} below, however, there are also situations of interest in which our techniques  fail to produce sharp results and, in particular, the resolvent estimates obtained by means of non-uniform Hautus tests or wavepacket conditions are necessarily suboptimal.

Our first result of this section provides a lower bound for
the resolvent norm $\norm{\RB}$ near eigenvalues of $A$.
Here $A$ is assumed to be skew-adjoint, but it need not have compact resolvent.
In this section we define
$B_s:=(B^\ast P_s)^\ast \in \Lin(U,X)$,
 where $P_s:=\chi_{\{s\}}(-iA)$ is the orthogonal projection onto $\ker(is-A)$.
Note that $\ran(B_s)\subseteq \ker(is-A)$
and hence we subsequently consider $B_s$ as an operator
from $U$ into $\ker(is-A)$.
If $\ran(B_s)=\ker(is-A)$, 
we write $B_s\pinv \in \Lin(\ker(is-A),U)$ for the \emph{Moore--Penrose pseudoinverse} of
$B_s$.
 If $\dim \ker (is-A)=1$ and $B_s\neq 0$, then $\norm{B_s\pinv}=\norm{B_s}\inv$.

\begin{proposition}
\label{prp:RGLowerBound}
Let  $A$ 
and $B$ satisfy
 Assumption~\textup{\ref{ass:ABBass}}
 and suppose that $A$ is skew-adjoint.
Suppose, in addition, that $i\R\subseteq\rho(A_B)$	and let $N:\R\to(0,\infty)$ be a function such that $\|(is-A_{B})\inv\|\le N(s)$ for all $s\in\R$.
Then
$\ran(B_s)=\ker(is-A)$
 for all $s\in\R$, and 
		$N(s) \geq\norm{B_s\pinv }^2$ for all
    $s\in\R$ such that $is\in\gs_p(A)$.
  \end{proposition}

 \begin{proof}

Fix $is\in\gs_p (A)$ and let $y\in \ker (is-A)$ be arbitrary. Then  $\iprod{y}{z}_X=\iprod{y}{P_s z}_X$ for all $z\in X$. Hence if  $x\in \Dom(A_B)$ is such that
$(is-A_B)x=y$, then
$$\iprod{y}{z}_X=\iprod{(is-A_{-1})x}{P_sz}_{X_{-1},X_1}+\iprod{BB^\ast x}{P_sz}_{X_{-1},X_1}$$
for all $z\in X$. 
It is straightforward to show that
the first term on the right-hand side is zero, so by  definition of $B_s$ we have $\iprod{y}{z}_X=\iprod{B_s B^\ast x}{z}_X$ for all $z\in X$. Thus $B_s B^\ast x=y$.
   Since $y\in \ker (is-A)$ was arbitrary, we deduce that $\ran(B_s) =\ker (is -A)$, and in particular the Moore--Penrose pseudoinverse $B_{s}\pinv\in\Lin(\ker(is-A),U)$  of $B_s$ is well defined.
 Now $\|B_{s}\pinv y\|=\min\setm{\|u\|}{u\in U\mbox{ and }B_su=y}$,
 so by the identity $B_s B^\ast x = y$ and Lemma~\ref{lem:ABBresgrowthlemma} we have
 \eq{
   \norm{B_{s}\pinv y}^2\leq \norm{B^\ast x}^2
   = \norm{B^\ast \RB[is]y}^2
   \leq N(s)\norm{y}^2.
 }
This holds for all $y\in \ker (is-A)$, so
   $\norm{B_{s}\pinv}^2
\leq N(s)$.
\end{proof}

\begin{remark}\label{rem:gap}
  If the skew-adjoint operator $A$ in Proposition~\ref{prp:RGLowerBound} has pure point spectrum and the eigenvalues of $A$ are uniformly separated (but not necessarily simple), so that the spectral gap
	\[ 
	d_{{\rm gap}}: =\inf \,\setm{ \abs{s-s'} }{  is , is'\in\gs (A),\ s \not= s'} 
	\] 
	is strictly positive, then the norms $\norm{B_s\pinv}$ can be used to construct functions $\gd$ and $\gg$ for which Theorem~\ref{thm:WPtoRG} provides the optimal rate of resolvent growth.
  Indeed, if
  we choose a constant $\gd(s)\equiv \gd:=  d_{{\rm gap}}/4>0$, then all non-trivial $(s,\gd(s))$-wavepackets of $A$ are eigenvectors corresponding to the unique eigenvalue $is'$ in the interval $i(s-\gd,s+\gd)$. 
		If $B_{s'}$ maps surjectively onto $\ker (is'-A)$ (which is in fact necessary for $is'$ to be an element of the resolvent set $\rho(A_B)$),  then
  for every $x\in\ker(is'-A)$ we have
  \eq{
    \norm{B^\ast x}
    = \norm{B_{s'}^\ast x }
    \geq \norm{B_{s'}\pinv}\inv\norm{x}.
  }
  The wavepacket condition~\eqref{eq:WPcond} is therefore satisfied for every bounded function $\gg$ such that $\gg(s)\equiv \norm{B_{s'}\pinv}\inv$ whenever $s\in (s'-\gd,s'+\gd)$ and $is'\in\gs (A)$.
  Theorem~\ref{thm:WPtoRG} then implies that $\norm{\RB}\lesssim \gg(s)^{-2}$, and  by Proposition~\ref{prp:RGLowerBound} this estimate is sharp in the sense that $N(s') \geq \gg(s')^{-2}$ whenever $is'\in\gs (A)$ and $N$ is as in~\eqref{eq:NUStabFreq}.
\end{remark}

As Proposition~\ref{prp:RGLowerBound} provides us with  a lower bound for the resolvent of $A_B$, we proceed by showing that such a bound  implies a lower bound for 
orbits of $(T_B(t))_{t \ge 0}.$ This will be done in a more general context in anticipation of possible applications elsewhere. It was shown in~\cite[Prop.~1.3]{BatDuy08} that one cannot in general hope for a better rate of decay than that given in Theorem~\ref{thm:sg_decay}.  The following new result is a consequence of~\cite[Prop.~1.3]{BatDuy08}. More specifically, it is a variant of a claim made in~\cite[Thm.~1.1]{BatChi16} and in the discussion following~\cite[Thm.~4.4.14]{AreBat11book}, and it gives a sharp optimality statement of the same type but which, crucially, is applicable as soon as one has a lower bound for the resolvent along a (possibly unknown) unbounded sequence of points on the imaginary axis. The proof uses the same ideas as that of~\cite[Cor.~6.11]{BatChi16}.

\begin{proposition}
  \label{prp:opt}
Let $X$ be a Banach space and let $(T (t))_{t\ge0}$ be a bounded semigroup on $X$ whose generator $A$ satisfies $i\R\subseteq\rho(A)$. Suppose that $N:\R_+\to(0,\infty)$ is a continuous non-decreasing function such that $N(s)\to\infty$ as $s\to\infty$ and
\begin{equation}\label{eq:limsup}
\limsup_{|s|\to\infty}\frac{\|(is-A)\inv\|}{N(|s|)}>0.
\end{equation}
Then there exists $c>0$ such that
\begin{equation}\label{eq:lb}
\limsup_{t\to\infty}N\inv(ct)\|T (t)A\inv\|>0,
\end{equation}
and if $N$ has positive increase then~\eqref{eq:lb} holds for all $c>0$.
\end{proposition}

\begin{proof}
    Consider the continuous non-decreasing function $n:\R_+\to (0,\infty)$ defined by $n(t)=\sup_{\tau\ge t}\|T(\tau)A\inv\|$, $t\ge0$, and let $n\inv$ denote any right-inverse of $n$. Note that
    $n$ takes strictly positive values since by~\eqref{eq:limsup} the semigroup $\SG$ cannot be nilpotent, and that $n(t)\to 0$ as $t\to\infty$ by~\cite[Thm.~1.1]{BatDuy08}.
Furthermore, by~\eqref{eq:limsup} and~\cite[Prop.~1.3]{BatDuy08} we may find a constant $c>0$ and an increasing sequence $(s_k)_{k\in\N}$ of positive numbers such that $s_k\to\infty$ as $k\to\infty$ and $N(s_k)<c n\inv((2s_k)\inv)$ for all $k\in\N$. Let $t_k=n\inv((2s_k)\inv)$ for $k\in\N$. Then $t_k\to\infty$ as $k\to\infty$ because $N$ is assumed to be unbounded, and we have $s_k=(2n(t_k))\inv$, $k\in\N$. Now $N(N\inv(ct_k))=ct_k>N(s_k)$
and hence $N\inv(ct_k)> (2n(t_k))\inv$ for all $k\in\N$. Letting $K=\sup_{t\ge0}\|T (t)\|$, it follows that
$$\frac{1}{2N\inv(ct_k)}\le n(t_k)\le K \|T(t_k)A\inv\|,\qquad k\in\N,$$
 which establishes~\eqref{eq:lb}. If $N$ has positive increase then by~\cite[Prop.~2.2]{RozSei19} we have $N\inv(t)\asymp N\inv(ct)$ as $t\to\infty$ for all $c>0$, which immediately yields the second statement.
\end{proof}

\begin{remark}
If $N$ is not assumed to have positive increase then it is possible for~\eqref{eq:limsup} to be satisfied but for~\eqref{eq:lb} to hold only for certain values of $c>0$. We refer the interested reader to the discussion following~\cite[Rem.~3.3]{RozSei19} for an  example of a contraction semigroup on a Hilbert space such that~\eqref{eq:limsup} holds for $N(s)=\log(s)$, $s\ge2$, and $\|T (t)A\inv\|=O(e^{-t/2})$ as $t\to\infty$. In particular, \eqref{eq:lb} does not hold for any $c\in(0,1/2)$.
\end{remark}

The considerations above lead to the following statement, which is the main result of this section. It is an immediate consequence of Propositions~\ref{prp:RGLowerBound} and~\ref{prp:opt}, both of which are applicable under more general assumptions. The result provides lower bounds for orbits of $(T_B(t))_{t \ge 0}$
under an assumption on the action of $B^\ast$ on eigenvectors of $A$ associated with imaginary eigenvalues $is_k\in\sigma_p(A)$.
These lower bounds  will allow us to  show in Section~\ref{dwe1d} below that the non-uniform decay rates we obtain from our observability conditions are optimal (or near-optimal) in several concrete situations of interest.

\begin{theorem}
  \label{thm:OptimMain}
Let  $A$ and $B$ satisfy
Assumption~\textup{\ref{ass:ABBass}}, suppose that $A$ is skew-adjoint and that $i\R\subseteq\rho(A_B)$.
If there exist a sequence $(s_k)_{k\in\N}\subseteq \R$,
 $\abs{s_k}\to \infty$ as $k\to \infty$ and a continuous non-decreasing function $N_0: \R_+\to (0,\infty)$ of positive increase such that $\norm{B_{s_k}\pinv}^2\geq N_0(\abs{s_k})$ for all $k\in\N$, then
\eq{
  \limsup_{t\to\infty} N_0\inv(t) \|T_B(t) A_B^{-1}\|>0.
}
 Consequently, if~\eqref{eq:NUStabFreq} holds then there exists a sequence $(t_k)_{k\in\N}\subseteq (0,\infty) $ with $t_k\to \infty$ as $k\to\infty$ such that $N\inv(t_k)\lesssim N_0\inv(t_k)$ for all $k\in\N$. 
\end{theorem}

We finish this section with a result of independent interest, offering  
an asymptotic estimate for a collection of eigenvalues of $A_B$
under a uniform spectral gap condition of the type discussed in Remark~\ref{rem:gap}. 

\begin{proposition}\label{accurate}
    Let  $A$ be skew-adjoint and suppose that $B\in \Lin(U,X)$ is compact.
 Suppose further that 
 $\gs (A) = \gs_p (A)$ and that this set is infinite, that $\dim \ker (is -A) = 1$ for every $is \in\gs (A)$,
and  that $d_{{\rm gap}}>0$.
Then
 there exist a family $(\lambda_s)_{is\in\gs_p (A) }$ and $s_0 \geq 0$ such that $\lambda_s\in \gs (A_B)$ for $\abs{s}\geq s_0$ and $|\lambda_s - (is -\norm{B_s}^2 )| = o (\norm{B_s}^2)$ as $\abs{s}\to\infty$.
\end{proposition}
\begin{proof}
First, we  note that 
\[\setm{\gl\in\C_-}{\ker(I+B^\ast\RA[\gl]B)\neq \set{0}}\subseteq \gs_p(A_B).\]
 Indeed, if $\gl\in\C_-$ and $ u\in U\setminus \set{0} $ are such that $B^\ast \RA[\gl]B u=-u$, then $(\gl-A_B)\RA[\gl]Bu
=0$. Since $\RA[\gl]Bu\neq 0$ (otherwise
$u=-B^\ast\RA[\gl]Bu=0$), we conclude that $\gl\in\gs_p(A_B)$. This reduces our problem to finding suitable points $\lambda\in\C_-$ with  $\ker(I+B^\ast\RA[\gl]B)\neq \set{0}$.

Our assumptions on  $A$ and compactness of $B$ imply that $\norm{B_s} = \norm{P_sB}\to 0$ as $\abs{s}\to\infty$.
Fix $is\in\gs_p (A)$ such that $|s|\geq 9 \norm{B}^2$ and $\norm{B_s}^2 \leq d_{{\rm gap}}$. 
By Proposition~\ref{prp:RGLowerBound},
$B_s$ maps surjectively onto $\ker (is-A)$, and therefore $B_s \neq  0$.
 Let
 \[ F_s(\gl)=(\gl-is)(I+B^\ast \RA[\gl]B).
\]
Note that for $\gl\in\rho (A)$ we have
$\ker(I+B^\ast\RA[\gl]B)\neq \set{0}$
if and only if
$\ker(F_s(\gl))\neq \set{0}$.
  Our aim is to apply Rouch\'e's theorem for operator-valued functions~\citel{GohSig72}{Thm.~2.2}.
We have $F_s(\gl)=G_s(\gl)+H_s(\gl)$ with
\eq{
  G_s(\gl)  = \gl-is +B_s^\ast B_s , \qquad
  H_s(\gl) =  (\gl-is)B^\ast\RA[\gl]B-B_s^\ast B_s .
}
Since 
 $B_s^\ast B_s$ is a rank-one operator and $\dim X >1$, $G_s(\gl)$ is boundedly invertible if and only if $\gl\notin \set{is- \norm{B_s}^2,is}$. Let $r_s = \norm{B_s}^2/2$ and define the closed disk $\Omega_s = \setm{\gl\in\C}{\abs{\gl-(is-\norm{B_s}^2)}\leq r_s }\subseteq\C_-$ and $\Gamma_s = \partial \Omega_s$.  Then $G_s(\gl)$ is boundedly invertible for all $\gl\in \Omega_s\setminus \set{is-\norm{B_s}^2}$, and for all $\gl\in\Gamma_s$ we have
\eq{
  \norm{G_s(\gl)\inv}= \frac{1}{\dist(\gl,\set{is-\norm{B_s}^2,is})} = \frac{1}{r_s}.
}

Let $J_s= \setm{ s'\in\R }{ |s' -s| \leq |s|/2}$. For every $s'\in \R \setminus J_s$  and every $\gl\in\Omega_s$ we have
\eq{
|\gl -is' | & \geq |is' -is| - |\gl -is| \geq \frac{|s|}{2} - \frac{3}{2} \norm{B_s}^2 \geq \frac{|s|}{3} ,
}
where the last inequality follows from the condition $|s|\geq 9\norm{B}^2$. Hence, for every  $\lambda\in\Omega_s$,
\eq{
\norm{B^\ast \RA[\gl] \chi_{\R \setminus J_s}(-iA)B } & \leq \norm{B^\ast} \, \sup_{|s' -s|>|s|/2} \frac{1}{|\gl -is'|} \, \norm{B} \leq \frac{3 \norm{B}^2}{|s|} .
}
Thus, for every $u\in U$ with $\norm{u}\leq 1$, by the Cauchy-Schwarz inequality,  the uniform spectral gap assumption and Bessel's identity, we see that
\eq{
  \frac{\norm{H_s(\gl)u}}{\abs{\gl-is}}
&\leq \norm{B^\ast \RA[\gl] \chi_{\R \setminus J_s}(-iA)Bu} \\&\qquad\qquad +
 \Norm{B^\ast \RA[\gl] \chi_{ J_s}(-iA)Bu - \frac{B_s^\ast B_su}{\gl -is}} \\
& \leq \frac{3 \Norm{B}^2}{|s|} + \Norm{\sum_{is'\in (\gs_p (A) \cap iJ_s )  \setminus \{is\} } \frac{1}{\gl -is'} B_{s'}^\ast B_{s'} u } \\
& \leq \frac{3 \norm{B}^2}{|s|}  + \sup_{|s'| \geq |s|/2} \norm{B_{s'}^\ast}\, \left( 2\sum_{j=1}^\infty \frac{1}{d^2_{{\rm gap}}j^2} \right)^\frac12 \, \left( \sum_{is'\in\gs_p (A)} \norm{B_{s'} u}^2 \right)^\frac12 \\
&\leq \frac{3 \norm{B}^2}{|s|} + \frac{\pi \norm{B}}{\sqrt{3} d_{{\rm gap}}}\, \sup_{|s'| \geq |s|/2} \norm{B_{s'}} .
}
Thus $\norm{H_s(\gl)}\leq q_s \abs{\gl-is}$ for some $q_s\geq 0$ satisfying $q_s\to 0$ as $|s|\to \infty$. Then, for $|s|$ large enough  and $\gl\in\Gamma_s$,
\eq{
  \norm{G_s(\gl)\inv H_s(\gl)}
  \leq \frac{q_s \abs{\gl-is}}{r_s}
  \leq 3q_s
  <1.
}
Rouch\'e's theorem~\citel{GohSig72}{Thm.~2.2} now implies that for every $is\in\gs_p (A)$ with $|s|$ sufficiently large there exists $\gl_s\in\Omega_s$ such that $\ker(F(\gl_s))\neq \set{0}$, and the proof is complete.
\end{proof}

Observe that if $A$ and $B$ are as in Proposition~\ref{accurate} 
and if $i\R\subseteq\rho(A_B)$, then the result implies 
that $\liminf_{\abs{s} \to \infty}\|B_s\|^2\|(is -A_B)^{-1}\|>0.$ Then using Proposition~\ref{prp:opt}
as in Theorem~\ref{thm:OptimMain}, we obtain a lower bound for  $\|T_B(\cdot) A_B^{-1}\|$ along a sequence
$(t_k)_{k\in\N}\subseteq(0,\infty)$ with $t_k\to\infty$ as $k\to\infty$. We omit a precise formulation of the corresponding statement since it is completely analogous
to Theorem~\ref{thm:OptimMain}.

\section{Non-uniform stability of damped partial differential equations}
\label{sec:PDEmodels}

In this section we apply our
general results to several concrete partial differential equations of different types. In particular, we consider damped wave equations on one- and two-dimensional spatial domains, a one-dimen\-sion\-al fractional Klein--Gordon equation, and a damped Euler--Bernoulli beam equation.
We also refer to a recent article~\cite{SuTuc20} for an application of Theorem~\ref{thm:WPtoRG} in the study of a coupled PDE system describing the dynamics of linearised water waves.

\subsection{Wave equations on two-dimensional domains}
\label{sec:2Dwaves}

In this section we consider wave equations on bounded simply connected domains $\Omega\subseteq \R^2$
 which are either convex or have sufficiently regular (say $C^2$) boundary to ensure
that the domain
of the Dirichlet Laplacian on $\Omega$ is included in $H^2(\Omega)$.
The wave equation with viscous damping and Dirichlet boundary conditions is given by
\begin{subequations}
  \label{eq:2Dwave}
  \eqn{
    &w_{tt}(\xi,t)-\Delta w(\xi,t)+ b(\xi)^2w_t(\xi,t) =0, \qquad \xi\in \Omega, ~t>0,\\
    & w(\xi,t)=  0, \hspace{5.35cm} \xi\in\partial \Omega,~ t> 0,\\
    & w(\cdot,0)= w_0(\cdot)\in H^2(\Omega)\cap H_0^1(\Omega), \qquad w_t(\cdot,0)=w_1(\cdot)\in H_0^1(\Omega).
  }
\end{subequations}
Here $b\in \Lp[\infty](\Omega)$ is the non-negative damping coefficient. It is well known that the geometry of $\Omega$ and the region where $b(\cdot)>0$ have great impact on the asymptotic properties of the wave equation. In the framework of Section~\ref{second_order_setup} we set $H=\Lp[2](\Omega)$,  $L=-\Delta$ with domain $H_1=H^2(\Omega)\cap H_0^1(\Omega)$, and
 define $U=\Lp[2](\Omega)$ and $D\in \Lin(\Lp[2](\Omega))$ by $D u = bu$ for all $u\in\Lp[2](\Omega)$.
Since $D\in \Lin(U,H)$, the function $\mu_0$
 in  Section~\ref{sec:freqdomSecondOrder} can be chosen to be bounded.

\subsubsection{Exact observability of the Schr\"odinger group}

In order to apply Proposition~\ref{prp:SchObstoRG} to  the damped wave equation~\eqref{eq:2Dwave} we need to understand the observability properties of the Schr\"odinger group on $\Omega$. Of particular interest here is the case of exact observability of the Schr\"odinger group, which corresponds to~\eqref{eq:SchObs} being satisfied for constant functions $M_0$ and $m_0$. In such cases Proposition~\ref{prp:SchObstoRG} immediately yields the resolvent bound $\|(is-A_{B})\inv\|\lesssim 1+s^2,\ s \in \R,$ so by Theorem~\ref{thm:sg_decay} (and Remark~\ref{rem:BT}) classical solutions of the corresponding abstract Cauchy problem decay like (and in fact faster than) $t^{-1/2}$ as $t\to\infty$. This was first proved in~\cite{AnaLea14}, but we mention that, similarly as in~\citel{JolLau20}{App.~B}, Proposition~\ref{prp:SchObstoRG} also allows us to deal with the much more general situation where~\eqref{eq:SchObs} is satisfied for functions $M_0$ and $m_0$ which satisfy suitable lower bounds but need not be constant. We take advantage of this added generality in Section~\ref{sec:BurZui} below.

The study of energy decay of damped waves via observability conditions has a long history~\cite{Sle74, Rus75,Ben78a,Leb96,AmmTuc01,BurHit07,CavMa19arxiv,LetSun21,LauLea21},
and in particular it predates the resolvent approach.
It is not surprising, therefore, that there is a rich literature on exact observability of the Schr\"odinger group, giving many concrete examples to which our abstract theory may be applied.
For instance, if $\Omega$ is a rectangle then it follows from a classical result due to Jaffard~\cite{Jaf90} that the Schr\"odinger group corresponding to our system is exactly observable for every non-negative $b\in L^\infty(\Omega)$ such that $\essup_{\xi\in\omega}b(\xi)>0$ for some non-empty open set $\omega\subseteq\Omega$; see~\cite{BurZwo19} for an even stronger result on the torus.
Similarly, it follows from~\cite[Thm.~9]{BurZwo04} that if $\Omega$ is the Bunimovich stadium then the corresponding Schr\"odinger group is exactly observable provided the damping $b$ has strictly positive essential infimum on a neighbourhood of one of the sides of the rectangle meeting a half-disk and also at one point on the opposite side.
This allows us to recover under a slightly weaker assumption the decay rate obtained in~\cite[Thm.~1.1]{BurHit07}.
Finally, if $\Omega$ is a disk then by~\cite[Thm.~1.2]{AnaLea16}  the Schr\"odinger group is exactly observable whenever $\essup_{\xi\in\omega}b(\xi)>0$ for some  open subset $\omega$ of $\overline{\Omega}$ such that $\omega\cap\partial\Omega\not=\emptyset$.
In fact, this condition is also necessary for exact observability, as can be seen by considering so-called \emph{whispering gallery modes}. We thus recover the decay rate for classical solutions obtained in~\cite[Rem.~1.7]{AnaLea16}. Further examples of when the Schr\"odinger group is exactly observable, including also higher-dimensional situations, may be found in~\cite[Sec.~2A]{AnaLea14}.
We point out in passing that there is also scope to apply directly the wavepacket result Theorem~\ref{thm:WPtoRG_WE}, which underlies Proposition~\ref{prp:SchObstoRG}. One case in which this is possible is if one knows that $\essup_{\xi\in\omega}b(\xi)>0$ for some open set $\omega\subseteq\Omega$ such that $\|w\|_{L^2(\omega)}\ge c\|w\|_{L^2(\Omega)}$ for some constant $c>0$ and all eigenfunctions $w$ of the Dirichlet Laplacian on $\Omega$.
This would allow us to take $\gamma_0$ to be constant in Theorem~\ref{thm:WPtoRG_WE}, provided we know how to choose $\delta_0$ in such a way that the $(s,\gd_0(s))$-wavepackets of $(-\Delta)^{1/2}$ are eigenfunctions associated with a single eigenvalue of $\Delta$. The appropriate lower bound is obtained in~\cite{HasHil09} in the case where $\Omega$ is a polygonal region and $\omega$  contains a neighbourhood  of each of the vertices of $\Omega$, and in fact these assumptions can be relaxed somewhat; see~\cite[Rem.~4]{HasHil09}. Choosing an appropriate $\delta_0$, however, requires detailed information on the distribution of the eigenvalues of the Dirichlet Laplacian on $\Omega$, which imposes a rather severe restriction on the domains $\Omega$ for which this approach is likely to bear fruit.

\subsubsection{Large damping away from a submanifold}\label{sec:BurZui}

We consider the damped Klein--Gordon equation on the square $\Omega=(0,1)^2$. This is a slight variant of~\eqref{eq:2Dwave}
in which $\Delta$ is replaced by $\Delta-m$ for some $m >0$.
Furthermore, we view $\Omega$ as the 2-torus $\mathbb T^2$ by imposing periodic rather than Dirichlet boundary conditions, thus allowing us to use the results of~\cite{BurZui16}.
We apply our abstract results, setting $H=L^2(\mathbb T^2)$ and $L=-\Delta+m$ with  domain $H_1=H^2(\mathbb T^2)$ in the framework of Section~\ref{second_order_setup}, in order to derive resolvent estimates
 under the assumption that the damping coefficient $b$
satisfies a certain type of lower bound
away from a proper submanifold $\Sigma$ of $\mathbb{T}^2$. 
A typical example would be for $\Sigma$ to be a circle of the form $\Sigma=\setm{(\xi_1,\xi_2)\in\Omega}{\xi_1\in(0,1)}$ for some fixed $\xi_2\in(0,1)$, but the results in~\cite{BurZui16} also apply in a much more general setting than this. The following result is a simple extension of~\cite[Cor.~1.3]{BurZui16} in our special case. The distance referred to here is the geodesic distance on the manifold $\mathbb{T}^2$.

\begin{corollary}\label{cor:dist}
Let $r:\R_+\to\R_+$ be a non-decreasing function satisfying $r(s)>0$ for all $s>0$, and suppose that $b(\xi)^2\ge r(\dist(\xi,\Sigma))$ for all $\xi\in\mathbb{T}^2$. Then $i\R\subseteq\rho(A_{B})$ and  there exist $\varepsilon\in(0,1)$ and  $s_0>0$ such that
$$\|(is-A_{B})^{-1}\|\lesssim r(\varepsilon|s|^{-1/2})^{-1}, \qquad |s|\ge s_0.$$
\end{corollary}

\begin{proof}
The inclusion $i\R\subseteq\rho(A_{B})$ may be obtained for instance by following the argument used in the proof of~\cite[Lem.~4.2]{AnaLea14}. Note in particular that the origin is removed from the spectrum as a result of the shift we apply to the Laplacian.
We now prove the resolvent estimate.
Given $\varepsilon\in(0,1)$ and  $s\in\R\setminus\{0\}$ let $\omega_{\varepsilon,s}=\setm{\xi\in\mathbb{T}^2}{\dist(\xi,\Sigma)< \varepsilon|s|^{-1/2}}$. By~\cite[Thm.~1.1]{BurZui16} (but see also~\cite{Sog88}) there exists $s_0>m$ such that
\begin{equation}\label{eq:BuZu}
\|w\|_{L^2(\omega_{\varepsilon,s})}\lesssim \varepsilon^{1/2}\big( |s|^{-1}\|(s^2-L)w\|_{L^2(\mathbb{T}^2)}+\|w\|_{L^2(\mathbb{T}^2)}\big)
\end{equation}
for all $w\in H^2(\mathbb{T}^2)$, $\varepsilon\in(0,1)$ and $s\in\R$ with $|s|\ge s_0$. 
By  assumption we have $b(\xi)^2\ge r(\varepsilon|s|^{-1/2})$ for all $\xi\in\mathbb{T}^2\setminus\omega_{\varepsilon,s}$. Thus if we let $m_\varepsilon(s)=r(\varepsilon|s|^{-1/2})^{-1}$ for $\varepsilon\in(0,1)$ and $|s|\ge s_0$, then
$$m_\varepsilon(s)\|bw\|_{L^2(\mathbb{T}^2)}^2\ge m_\varepsilon(s)\|bw\|_{L^2(\mathbb{T}^2\setminus\omega_{\varepsilon,s})}^2\ge \|w\|_{L^2(\mathbb{T}^2)}^2-\|w\|_{L^2(\omega_{\varepsilon,s})}^2,$$
and hence by~\eqref{eq:BuZu} and an application of Young's inequality we may choose  $\varepsilon\in(0,1)$ sufficiently small to ensure that
$$\|w\|_{L^2(\mathbb{T}^2)}^2\lesssim |s|^{-2}\|(s^2-L)w\|_{L^2(\mathbb{T}^2)}^2+m_\varepsilon(s)\|bw\|_{L^2(\mathbb{T}^2)}^2$$
for all $w\in H^2(\mathbb{T}^2)$ and all $s\in\R$ such that $|s|\ge s_0$. The result now follows from Proposition~\ref{prp:SchObstoRG} and Remark~\ref{rem:asymp}.
\end{proof}

We may use Corollary~\ref{cor:dist} to study the asymptotic behaviour of solutions of the damped Klein--Gordon equation. In particular, if $r(s)=cs^{2\kappa}$ for some constants $c,\kappa>0$ then Corollary~\ref{cor:dist} yields the estimate $\|(is-A_{B})^{-1}\|\lesssim 1+|s|^{\kappa}$ for $s\in\R$, and hence by Theorem~\ref{thm:sg_decay} any classical solution decays at the rate $t^{-1/\kappa}$. Note that this is worse than the rate obtained under additional assumptions in~\cite{LeaLer17,DatKle20} for the classical damped wave equation~\eqref{eq:2Dwave}, which formally corresponds to the choice $m=0$ in our setting. On the other hand, it is stated in~\cite[Rem.~1.5]{BurZui16} that in general the rate $t^{-1/\kappa}$ cannot be improved. The main value of Corollary~\ref{cor:dist} lies in the fact that it leads to interesting non-polynomial resolvent estimates whenever the function $r$ providing the lower bound is chosen appropriately.

\subsubsection{Suboptimality of the observability and wave\-packet conditions}
  \label{sec:2DwaveWPsuboptimality}

  In this section we discuss certain natural limitations of our results in Section~\ref{sec:freqdomresults}, and in particular describe situations where the non-uniform decay rates obtained by our methods are suboptimal. 
  As shown in~\cite{BurHit07,AnaLea14,LeaLer17,DatKle20,Sun21} in the case of multi-dimensional wave equations with viscous damping, rates of non-uniform decay are dependent not only on the location of the damping but also on the smoothness of the damping coefficient $b$.
By studying the damped wave equation~\eqref{eq:2Dwave} on a square $\Omega = (0,1)^2$ we can illustrate
that the resolvent growth rates in Sections~\ref{sec:freqdomresults} and~\ref{sec:timedomresults} are inherently suboptimal due to the fact that our observability concepts --- the non-uniform Hautus test, the wavepacket condition, the observability of the Schr\"odinger group and the non-uniform observability ---
are unable to detect the degree of smoothness of the damping coefficient $b$.

  For this purpose, let
  $\gw=(0,1/2)\times (0,1)$.
  For any arbitrarily small $\eps\in(0,1/2)$ we may as in~\citel{BurHit07}{Sec.~3} define a smooth non-negative damping coefficient $b_\eps$ such that $\supp b_\eps\subseteq \omega$, $\norm{b_\eps}_{\Lp[\infty]}\leq 1$, and $\norm{(is-A_{B_\eps})\inv}\lesssim 1+\abs{s}^{1+\eps},\ s \in \R,$
  where $B_\eps\in \Lin(\Lp[2](\Omega),X)$  is the damping operator associated with $b_\eps$.
Now consider the damping coefficient
  $b_\chi = \chi_\gw$, and denote the damping operator associated with this function by $B_\chi\in \Lin(\Lp[2](\Omega),X)$.
  For this damping coefficient the optimal order of resolvent growth is known to be $1+\abs{s}^{3/2}$~\cite{Sta17,AnaLea14}, and  in particular $\limsup_{\abs{s}\to\infty} \abs{s}^{-3/2}\norm{(is-A_{B_\chi})\inv}>0$.
  However, since $b_\chi(\xi)\geq b_\eps(\xi)$ for all $\xi\in\Omega$, we clearly have
  \eq{
    \norm{B_\chi^\ast x}\geq \norm{B_\eps^\ast x},  \qquad  x\in X.
  }
  Hence the non-uniform Hautus test~\eqref{eq:NUHautus}, the wavepacket condition~\eqref{eq:WPcond}, observability of the Schr\"odinger group~\eqref{eq:SchObs}, or non-uniform observability~\eqref{eq:NUObs} for the pair $(B_\eps^\ast,A)$ immediately implies the same property for the pair $(B_\chi^\ast,A)$ with the same parameters.
  In particular,  any resolvent estimate of the form $\norm{(is-A_{B_\eps})\inv}\leq N(s),\ s \in \R,$  obtained from Theorem~\ref{thm:HautustoResgrowth}, Theorem~\ref{thm:WPtoRG}, Proposition~\ref{prp:SchObstoRG} or Theorem~\ref{thm:NUObstoResgrowth} also implies that $\norm{(is-A_{B_\chi})\inv}\leq N(s)$ for $ s \in \R$. However, by~\citel{AnaLea14}{Prop.~B.1}
  we then also have $\limsup_{\abs{s}\to\infty}\abs{s}^{3/2}N(s)>0$. This means that $N(s)$ is a suboptimal upper bound for  $\norm{(is-A_{B_\eps})\inv}$ as $\abs{s}\to \infty$.

Comparing the rates of non-uniform decay of~\eqref{eq:2Dwave} with the two damping profiles $b_\eps$ and $b_\chi$ also shows that in the second part of Theorem~\ref{thm:HautustoResgrowth} it is in general impossible to choose functions $M$ and $m$ satisfying $M+m\lesssim N$.
To see this,
let
 $M_\eps$ and $m_\eps$ be functions $M$ and $m$ corresponding to the damping $b_\eps$.
 Then the inequality $b_\chi\geq b_\eps$ implies that  $(B_\chi^\ast,A)$, too, satisfies the Hautus test for the same functions $M_\eps$ and $m_\eps$, and by Theorem~\ref{thm:HautustoResgrowth} we have $\norm{(is-A_{B_\chi})\inv}\lesssim M_\eps(s)+m_\eps(s)$, $s \in \R$. However, since the optimal order of resolvent growth for
the damping $b_\chi$ is $\abs{s}^{3/2}$, the conclusion cannot be true unless
\eq{
  \limsup_{\abs{s}\to \infty}\, \abs{s}^{3/2}\big(M_\eps(s)+m_\eps(s)\big)>0 .
}
Thus $M_\eps+m_\eps$ provides a strictly worse resolvent bound than the estimate $\norm{(is-A_{B_\eps})\inv}\lesssim 1+\abs{s}^{1+\eps}$, $s \in \R,$ obtained in~\citel{BurHit07}{Sec.~3}.

Finally, comparison of the damping coefficients $b_\eps$ and $b_\chi$
further shows that a dissipative perturbation of a generator of a polynomially stable semigroup can strictly worsen the rate of decay.
Indeed, since $b_\chi\geq b_\eps$ by construction, the ``additional damping'' of the difference $b_\Delta = b_\chi-b_\eps\geq 0$
increases the asymptotic rate of resolvent growth as $|s|\to\infty$ from at most $ \abs{s}^{1+\eps}$ to $ \abs{s}^{3/2}$.
In terms of the semigroup generators this means that $A_{B_\eps}$ has a strictly slower asymptotic resolvent growth than $A_{B_\chi}$ even though $A_{B_\chi}$ is a dissipative perturbation of $A_{B_\eps}$.

\subsection{Damped wave equations on one-dimensional domains}\label{dwe1d}

\subsubsection{Damping at a single interior point}
  \label{sec:wave1Dpointwisedamping}

  In this section we consider the one-dimensional wave equation with pointwise damping studied in~\citel{AmmTuc01}{Sec.~5.1}; see also~\cite{RzeSch18} for a closely related problem on the stability of two serially connected strings. Our arguments rely essentially on ideas from~\cite{AmmTuc01}.
Given an irrational number $\xi_0\in (0,1)$, let us consider the problem
\begin{subequations}
    \label{eq:rs_fav}
  \eqn{
&w_{tt}(\xi,t)-w_{\xi\xi}(\xi,t)+w_t(t,\xi_0)\delta_{\xi_0}(\xi) =0, \qquad \xi\in (0,1), ~t>0,\\
     & w(0,t)= 0, \qquad w(1,t)=0, \hspace{2.9cm} t> 0,\\
    & w(\cdot,0)= w_0(\cdot)\in H^2(0,1)\cap H_0^1(0,1), \quad w_t(\cdot,0)=w_1(\cdot)\in H_0^1(0,1).
  }
\end{subequations}

As shown in~\citel{AmmTuc01}{Sec.~5.1}, the system~\eqref{eq:rs_fav} satisfies the assumptions in Section~\textup{\ref{second_order_setup}}
with $H=\Lp[2](0,1)$, 
$L=-\partial_{\xi\xi}$ with domain $H_1=H^2(0,1)\cap H_0^1(0,1)$, and $L$ has positive square root with domain $H_{1/2}=H_0^1(0,1)$.  
The damping operator $D$ is given by $D u=\gd_{\xi_0}u$ for all $u\in U= \C$, where $\gd_{\xi_0}$ is the Dirac delta distribution at $\xi=\xi_0$, and we indeed have $D \in \Lin(\C,H_{-1/2})$ and $D^\ast \in \Lin(H_{1/2},\C)$, where $H_{-1/2}=H^{-1}(0,1)$ and $H_{1/2}=H_0^1(0,1)$.
In order to describe the domain $\Dom(A_{B})$,
note that $A_{-1}\inv B=(-L\inv \gd_{\xi_0},0)=(z,0)$, where $z\in H_0^1(0,1)$ is the solution of the differential equation $z''=\gd_{\xi_0}$ with boundary conditions $z(0)=z(1)=0$ in $H^{-1}(0,1)$. We thus have
\eq{
  z(\xi) =
  \begin{cases}
    \xi(1-\xi_0), \qquad 0<\xi\leq \xi_0,\\
    \xi_0(1-\xi), \qquad \xi_0<\xi\leq 1.
  \end{cases}
}
Since $\Dom(A_{B})=\setm{x\in X_B}{ A_{-1} x-BB^\ast x\in X}$ by
Remark~\ref{rem:ABBAltDomain}, we
deduce that
(cf. ~\citel{AmmTuc01}{Sec.~5.1})
\eq{
  \Dom(A_{B}) = \setm{(u+z(\cdot)v(\xi_0),v)}{ u\in H^2(0,1)\cap H_0^1(0,1), ~ v\in H_0^1(0,1)},
}
and therefore classical solutions of~\eqref{eq:rs_fav} correspond to initial conditions
\eqn{
  \label{eq:1DwavePointwiseICs}
  w_0 = w_{00} + z(\cdot)w_1(\xi_0),
  \quad
  w_{00}\in H^2(0,1)\cap H_0^1(0,1),
  \quad
  w_1\in H_0^1(0,1).
}

Since the eigenvalues $\gl_n^2=n^2\pi^2$, $n\in\N$, and corresponding normalised eigenfunctions $\phi_n(\cdot)=\sqrt{2}\sin(n\pi \cdot)$ of $L$ are known explicitly, we may use the wavepacket condition in Theorem~\ref{thm:WPtoRG_WE} to analyse the stability properties of the damped system~\eqref{eq:rs_fav}.
Indeed, the eigenvalues $\gl_n=n\pi$, $n\in\N$, of $L^{1/2}$ have a uniform gap, so we may choose $\gd(s)\equiv \pi/4$. The non-trivial $(s,\gd(s))$-wavepackets of $L^{1/2}$ are then simply multiples of the eigenfunctions $\phi_n$ for $n\in\N$ such that $ n\pi\in (s-\pi/4,s+\pi/4) $.
  For any $n\in\N$ we have
  \eq{
    \abs{D^\ast \phi_n}
    = \abs{\phi_n(\xi_0)}
    = \sqrt{2}\abs{\sin(n\pi \xi_0)}.
  }

  In order to determine the rate of resolvent growth we need to estimate the coefficients $\abs{D^\ast \phi_n}$ from below.
   This certainly requires $\xi_0$ to be an irrational number, but in fact we shall need to assume more, namely that $\xi_0$ is   \emph{badly approximable} by rationals. It is known, for instance, that given any $\eps>0$ almost every irrational $\xi_0\in(0,1)$ has the property that
\begin{equation}\label{eq:log_lb}
\min_{m\in\N}\Abs{\xi_0-\frac mn}\ge \frac{1}{n^2\log(n)^{1+\eps}}
\end{equation}
  for all sufficiently large $n\ge2$, while simultaneously for almost every irrational $\xi_0\in(0,1)$ there exist rationals $m/n$ with arbitrarily large values of $n\ge2$ such that
 \begin{equation}\label{eq:log_ub}
 \Abs{\xi_0-\frac mn}\le \frac{1}{n^2\log(n)};
 \end{equation}
see for instance~\cite[Thm.~32]{Khi64book}. A rather special class of irrationals $\xi_0\in(0,1)$ is the set of irrationals that have \emph{constant type}. These are commonly defined to be those irrational numbers which have uniformly bounded coefficients in their partial fractions expansions. Irrationals of constant type include all irrational \emph{quadratic numbers}, that is to say irrational solutions of quadratic equations with integer coefficients. As shown in~\citel{Lan95book}{Ch.~II, Thm.~6}, an irrational number $\xi_0\in(0,1)$ has constant type if and only if there is a constant $c_{\xi_0}>0$ such that
\begin{equation}\label{eq:Dir_lb}
\min_{m\in\N}\Abs{\xi_0-\frac mn}\ge \frac{c_{\xi_0}}{n^2},\qquad n\in\N.
\end{equation}
  It follows from the Dirichlet approximation theorem~\citel{Lan95book}{Ch.~II,Thm.~1} that for any irrational number $\xi_0\in(0,1)$ there exist rationals $m/n$ with arbitrarily large values of $n\in\N$ such that
  \begin{equation}\label{eq:Dir_ub}
  \Abs{\xi_0-\frac mn}\le \frac{1}{n^2}.
  \end{equation}
The following result yields (essentially) sharp rates of decay for the energy of our damped system for  irrational numbers $\xi_0\in(0,1)$ of different nature.

\begin{corollary}   \label{cor:1DwavePointwiseDamping}
Let $w$ be the (classical) solution of~\eqref{eq:rs_fav} corresponding to initial conditions as in~\eqref{eq:1DwavePointwiseICs}.
 	\begin{itemize}
 \item [(a)] Fix $\eps>0$. For almost every irrational number $\xi_0\in(0,1)$ there exists $C_\eps>0$ such that
        \begin{equation}\label{eq:pw_damp_ub}
      \norm{(w(\cdot,t),w_t(\cdot,t))}_{H^1\times \Lp[2]}\leq  C_\eps\frac{\log(t)^{1+\eps}}{{t}^{1/2}}\norm{(w_{00},w_1)}_{H^2\times H^1},\quad t\ge2.
    \end{equation}
    Moreover, the rate is almost optimal in the sense that if $r:\R_+\to(0,\infty)$ is any function such that $r(t)=o(t^{-1/2}\log(t))$ as $t\to\infty$, then there exist $w_0,w_1$ as in~\eqref{eq:1DwavePointwiseICs} for which $r(t)\inv \norm{(w(\cdot,t),w_t(\cdot,t))}_{H^1\times \Lp[2]}$ is unbounded as $t\to\infty$.
\item [(b)]
If $\xi_0\in(0,1)$ is an irrational number of constant type then there exists $C>0$ such that
    $$     \norm{(w(\cdot,t),w_t(\cdot,t))}_{H^1\times \Lp[2]}\leq  \frac{C}{{t}^{1/2}}\norm{(w_{00},w_1)}_{H^2\times H^1},\qquad t\ge1.$$
    Moreover, the rate is optimal in the sense that if $r:\R_+\to(0,\infty)$ is any function such that $r(t)=o(t^{-1/2})$ as $t\to\infty$, then there exist $w_0,w_1$ as in~\eqref{eq:1DwavePointwiseICs} for which $r(t)\inv \norm{(w(\cdot,t),w_t(\cdot,t))}_{H^1\times \Lp[2]}$ is unbounded as $t\to\infty$.
		\end{itemize}
  \end{corollary}

  \begin{proof}
    The form of the estimates follows from Theorem~\ref{thm:sg_decay} and the property that for initial conditions as in~\eqref{eq:1DwavePointwiseICs} we have
\eq{
  \norm{A_{B}(w_0,w_1)}_X^2
= \norm{A (w_{00},w_1)}_X^2
= \norm{w_{00}''}_{\Lp[2]}^2 + \norm{w_1}_{H^1}^2.
}

In order to prove (a),
we will use
 Theorem~\ref{thm:WPtoRG_WE}. 
 As shown in~\cite[Lem.~5.3]{AmmTuc01}, 
we have 
 $\abs{s}\,\norm{D^\ast ( (1+is)^2+L_{-1})\inv D}\lesssim 1$, $s\in\R$.
To verify the wavepacket condition, 
 let $\xi_0$ be such that~\eqref{eq:log_lb} holds. 
For a given $n\ge2$, choose $m\in\N$ in such a way that $C_n\in\R$ defined by
      \begin{align*}
	\xi_0 = \frac{m}{n} + \frac{C_n}{n^2\log(n)^{1+\eps}}
      \end{align*}
    has minimal absolute value. By~\eqref{eq:log_lb} we have
  $1 \leq \abs{C_n} \leq n\log(n)^{1+\eps}/2$
for all sufficiently large $n\ge2$, and since $2 r/\pi\leq \sin(r)\leq r$ for $0\leq r\leq \pi/2$ it follows that
  \eq{
    \abs{D^\ast \phi_n}=\sqrt{2}\abs{\sin(n\pi \xi_0)}
= \sqrt{2}\Abs{\sin\left(\frac{C_n\pi}{n\log(n)^{1+\eps}}\right)}
\geq \frac{2\sqrt{2}}{n\log(n)^{1+\eps}}
  }
 for all sufficiently large $n\ge2$. 
Thus by Theorem~\ref{thm:WPtoRG_WE} we have $\|(is-A_{B})\inv\|\lesssim s^2\log(|s|)^{2+2\eps}$, $|s|\ge2$,  and hence~\eqref{eq:pw_damp_ub} follows from Theorem~\ref{thm:sg_decay}; see also~\cite[Thm.~1.3]{BatChi16}.

In order to prove the optimality statement, note that by~\eqref{eq:log_ub} there exist infinitely many $n\ge2$ for which $|C_n|\le\log(n)^\eps$ and therefore also
\eq{
    \abs{D^\ast \phi_n}= \sqrt{2}\Abs{\sin\left(\frac{C_n\pi}{n\log(n)^{1+\eps}}\right)}
\le \frac{\sqrt{2}\pi}{n\log(n)}.
  }
Now Proposition~\ref{prp:RGLowerBound} shows that
$$\limsup_{|s|\to\infty}\frac{\|(is-A_{B})\inv\|}{|s|^2\log(|s|)^2}>0,$$
and it follows from Proposition~\ref{prp:opt}  that
$$\limsup_{t\to\infty}\frac{\log(t)}{t^{-1/2}} \|T_{B}(t)A_{B}\inv\|>0.$$
Now the optimality statement follows from a simple application of the uniform boundedness principle.

  The argument for part (b) is entirely analogous and slightly simpler. It uses~\eqref{eq:Dir_lb} and~\eqref{eq:Dir_ub} in place of~\eqref{eq:log_lb} and~\eqref{eq:log_ub}, respectively.
  \end{proof}

\subsubsection{Weak damping}\label{subs_weak}
  \label{sec:wave1Dweakdamping}
  In this section we consider
  a \emph{weakly damped} wave equation on $(0,1)$, namely
  \begin{subequations}
    \label{eq:1DwaveWeakDamping}
    \eqn{
    \label{eq:1DwaveWeakDamping1}
    &w_{tt}(\xi,t)-w_{\xi\xi}(\xi,t)+b(\xi)\int_0^1 \hspace{-.1cm} b(r)w_t(r,t)dr =0, ~\; \xi\in (0,1), ~t>0,\\
    \label{eq:1DwaveWeakDamping2}
    & w(0,t)= 0, ~ w(1,t)=0, \hspace{4.25cm} t> 0,\\
    \label{eq:1DwaveWeakDamping3}
    & w(\cdot,0)= w_0(\cdot)\in H^2(0,1)\cap H_0^1(0,1), ~\; w_t(\cdot,0)=w_1(\cdot)\in H_0^1(0,1),
    }
  \end{subequations}
  where $b\in \Lp[2](0,1;\R)$ is the damping coefficient.
The wave equation has the form considered in Section~\textup{\ref{second_order_setup}} 
with $H=\Lp[2](0,1)$, 
 $L=-\partial_{\xi\xi}$ with domain $H_1=H^2(0,1)\cap H_0^1(0,1)$, and $L$ has positive square root with domain $H_{1/2}=H_0^1(0,1)$.
Moreover,  $U=\C$ and   $D\in \Lin(\C,H)$ is the rank-one operator defined by $D u=bu$ for all $u\in \C$.

The operator $L$ is the same as in
Section~\ref{sec:wave1Dpointwisedamping}.
Hence if we define $\gd(s)\equiv \pi/4$ then the non-trivial $(s,\gd(s))$-wavepackets of $\Azhalf $ are multiples of the normalised eigenfunctions $\phi_n$ for $n\in\N$ such that $ n\pi\in (s-\pi/4,s+\pi/4) $.
For any $n\in\N$ we have
\eq{
    \abs{D^\ast \phi_n} = \sqrt{2}\Abs{\int_0^1 b(\xi)\sin(n\pi \xi)d\xi}.
}
  For a large class of functions $b$ these Fourier sine series coefficients have explicit expressions.
  In order to have  $i\R\subseteq \rho(A_{B})$ we require that $D^\ast \phi_n\neq 0$ for all $n\in\N$, and the rate at which $\abs{D^\ast \phi_n}$ decays to zero as $n\to \infty$ determines the rate of resolvent growth. In the following we summarise the conclusions of Theorem~\ref{thm:WPtoRG} for a class of dampings.

  \begin{corollary}
    \label{cor:1DwaveWeakDamping}
    Assume that $\abs{D^\ast \phi_n}\gtrsim f(n\pi),\  n \in \N,$ for a continuous strictly decreasing function $f:\R_+\to(0,\infty)$ such that $f(\cdot)\inv$ has positive increase.
    Then there exist $C,t_0>0$ such that
    for all $w_0\in H^2(0,1)\cap H_0^1(0,1)$ and $w_1\in H_0^1(0,1)$ the (classical) solution $w$ of~\eqref{eq:1DwaveWeakDamping}
satisfies
    \eqn{
      \label{eq:WEweakDecay}
      \norm{(w(\cdot,t),w_t(\cdot,t))}_{H^1\times \Lp[2]}\leq  \frac{C}{N\inv(t)}\norm{(w_0,w_1)}_{H^2\times H^1}, \qquad t\geq t_0,
    }
    where $N\inv$ is the inverse function of $N(\cdot):=f(\cdot)^{-2}$.
Moreover, if there exists an increasing sequence
 $(n_k)_{k\in\N}\subseteq \N$
such that
$\abs{D^\ast\phi_{n_k}}\lesssim f(n_k\pi)$ for all $k\in\N$,
then the decay rate is optimal in the sense of Theorem~\textup{\ref{thm:OptimMain}}.
  \end{corollary}

  \begin{proof}
If $\abs{D^\ast \phi_n}\gtrsim f(n\pi),\ n \in \N,$
    		then the wavepacket condition in~\eqref{eq:WE_WPform} is satisfied for
    $\delta_0=\pi/4$ and
$\gg_0(s)= f(s+\pi/4)$. 
Moreover, since $D\in \Lin(\C,H)$, we have 
 $\abs{s}\,\norm{D^\ast ( (1+is)^2+L)\inv D}\lesssim 1$, $s\in\R$.
Thus
Theorem~\ref{thm:WPtoRG_WE} implies that
$\|(is -A_{B})^{-1}\|\lesssim f(\abs{s}+\pi/4)^{-2}$, $s \in \R,$
and  Theorem~\ref{thm:sg_decay} yields~\eqref{eq:WEweakDecay} 
with the function $N_0$ defined by $N_0(s)= f(s+\pi/4)^{-2}$ for $s>0$. The claim now follows from the fact that
$N\inv=N_0\inv+\pi/4$.
  \end{proof}

   For the particular damping functions $b$ defined by $b(\xi) = 1-\xi$, $b(\xi) = \xi^2(1-\xi)$ and $b(\xi)=\chi_{(0,\xi_0)}(\xi)$, where $\xi_0\in (0,1)$ is an irrational of {constant type}, the optimal decay rates are given
   by (writing $b_n=D^\ast \phi_n$ for brevity)
  \begin{subequations}
    \label{eq:1DwaveWeakDampingCoeffs}
    \eqn{
      b(\xi) &= 1-\xi,  && b_n = \frac{\sqrt{2}}{n\pi},  &~ & N\inv(t)\inv\lesssim t^{-1/2},\\
      b(\xi) &= \xi^2(1-\xi),  && b_n = \frac{2 \sqrt{2}(2(-1)^n-1)}{n^3\pi^3},  & ~ & N\inv(t)\inv\lesssim t^{-1/6},\\
      b(\xi) &= \chi_{(0,\xi_0)}(\xi) , && b_n = \frac{\sqrt{2}(1-\cos(n\pi \xi_0))}{n\pi},  & ~ & N\inv(t)\inv\lesssim t^{-1/6}.
    }
  \end{subequations}

The required upper and lower bounds for $|D^\ast \phi_n|$ in the third example follow by arguments similar to those used in the proof of Corollary~\ref{cor:1DwavePointwiseDamping}, once again using~\eqref{eq:Dir_lb} and~\eqref{eq:Dir_ub}. Optimality in all three examples is a consequence of Theorem~\ref{thm:OptimMain}.

\begin{remark}\label{rem:DampingClass}
The above discussion implies that the Fourier sine series coefficients
$b_n=D^\ast \phi_n$
of the damping $b$
determine the resolvent growth and thus the rate of energy
decay in~\eqref{eq:1DwaveWeakDamping}. So it is natural
to try to relate the energy decay to the properties of $b$ and $(b_n)_{n\in\N}$ directly.
However, it is difficult to give a succinct answer here without specifying a precise class of functions $b$.
First note that since $b \in L^2(0,1)$, we have
$(b_n)_{n\in\N}\in\ell^2$.
On the other hand,  the results in~\cite{Naz98} show that for any $(c_n)_{n\in\N}\in \ell^2$ with $ c_n \ge 0$ there exists $b \in C[0,1]$
such that
$|b_n|\ge c_n$ for all $n\in\N$,
and thus any rate of decay that can be achieved with a damping function $b \in L^2(0,1)$
can also be achieved with a more regular function $b \in C[0,1]$.
However,
imposing
further regularity properties on $b$, such as
H\"older type conditions, changes the situation substantially.

In general, finer estimates for decay of $(b_n)_{n\in\N}$ depend heavily on the modulus of continuity
(or the integral modulus of continuity) of $b$, and conversely for $(b_n)_{n\in\N}$ close in a sense to being monotone
one may infer regularity properties of $b$ from the sequence $(b_n)_{n\in\N}$; see for instance
\cite[Ch.~7]{Edw79book}, \cite[Ch.~5]{Zyg02book}, \cite{DyaMuk19} and references therein.

Note finally that
any polynomial rate of decay
$t^{-\alpha}$ with $\ga\in(0,1)$
can be achieved by choosing the damping function $b\in\Lp[2](0,1)$ such that $b_n=n^{-1/(2\alpha)}$ for $n\in\N$.
Moreover, by~\cite{Naz98} the same scale of polynomial rates can be realised by means of continuous damping functions.
It would be interesting to consider similar statements about other scales of decay rates, for instance of regularly varying functions, but we do not pursue this here.
\end{remark}

\subsection{A damped fractional Klein--Gordon equation}

In this example we consider a ``fractional Klein--Gordon equation'' with viscous damping studied in~\cite{MalSta20}; see also~\cite{Gre20}.
For a fixed $\ga\in(0,1]$ this system has the form
\eq{
&w_{tt}(\xi,t)+(-\partial_{\xi\xi})^\ga w(\xi,t)+mw(\xi,t) + b(\xi)^2w_t(\xi,t) =0, \qquad \xi\in \R,\ t>0\\
&\qquad\quad w(\cdot,0)= w_0(\cdot)\in H^{2\ga}(\R), \qquad w_t(\cdot,0)=w_1(\cdot)\in H^\ga(\R),
}
where $m>0$ and $b\in \Lp[\infty](\R)$ is the non-negative damping coefficient.
We assume that  $\essinf_{\xi\in\omega} b(\xi)>0$ for some non-empty open set $\omega\subseteq\R$ which is invariant under translation by  $2\pi$.

Polynomial stability of this equation was studied e.g.\ in~\cite{MalSta20}.
In the following proposition we use the wavepacket condition~\eqref{eq:WE_WPform} to derive the same resolvent estimate under the above assumptions on $b$ (strictly weaker conditions on the damping were also considered recently in~\cite{Gre20}).
    The fractional Klein--Gordon equation is again of the form studied in Section~\textup{\ref{second_order_setup}},
 now with $H =U= \Lp[2](\R)$, $L=(-\partial_{\xi\xi})^\ga+m>0$ with domain $H_1= H^{2\ga}(\R)$ and $H_{1/2}=H^{\ga}(\R)$. The damping operator $D\in \Lin(\Lp[2](\R))$ is the multiplication operator defined by $D u = bu$ for all $u\in\Lp[2](\R)$. 
 
\begin{proposition}
Let $0<\ga<1$. There exists $C>0$ such that
for every $w_0\in H^{2\ga}(\R)$ and $w_1\in H^\ga(\R)$ the solution $w$ 
 of the fractional Klein--Gordon equation
 satisfies
\eq{
\norm{(w(\cdot,t),w_t(\cdot,t))}_{H^\ga\times \Lp[2]}\leq  \frac{C}{t^{\ga/(2-2\ga)}}\norm{(w_0,w_1)}_{H^{2\ga}\times H^\ga}, \qquad t>0.
}
\end{proposition}

\begin{proof}
Let us begin by showing that the classical Klein--Gordon equation corresponding to $\ga=1$ is exponentially stable.
Due to the properties of the damping coefficients we may choose a smooth and $2\pi$-periodic function $b_1$ such that $0\leq b_1\leq b$ on $\R$ and $\inf_{\xi\in\gw_1}b_1(\xi)>0$ for a non-empty open set $\gw_1\subseteq \gw$.
By~\citel{BurJol16}{Thm.~1.2} the Klein--Gordon equation with damping coefficient $b_1$ is exponentially stable. If we define $D_1\in \Lin(\Lp[2](\R))$ so that $D_1u = b_1u$ for all $u\in\Lp[2](\R)$,
and define
$B_1=\pmatsmall{0\\ D_1}$,
 then
$(B_1^\ast,A)$ is exactly observable, and
by~\citel{Mil12}{Cor.~2.17} the pair $(\mathcal  B_1^\ast, A)$ satisfies the wavepacket condition~\eqref{eq:WPcond} for constant functions $\gd(s)\equiv \gd>0$ and $\gg(s)\equiv \gg>0$.
However, since $b(\xi)\geq b_1(\xi)$ for all $\xi\in\R$ we see that also $(B^\ast,A)$ satisfies the wavepacket condition for the same functions $\gd$ and $\gg$.

Let us temporarily write $L_\alpha$ for the operator $(-\partial_{\xi\xi})^\ga+m$, $0<\alpha\le1$, accepting that this entails a minor abuse of notation. Since $\sigma(L_\alpha)\subseteq[m,\infty)$ for $0<\alpha\le 1$, we obtain from Lemma~\ref{lem:WE_WPs} that
\eqn{
\label{eq:KGEqnWPcond}
\norm{D^\ast w}_U\geq \gg_1\norm{w}_H
}
for all $(s,\gd_1)$-wavepackets $w$ of $L_{\alpha}^{1/2}$, where  $\gd_1,\gg_1>0$ are suitable constants.

For $0<\alpha\le1$ and any bounded function $\delta_0:\R_+\to(0,\infty)$ the $(s,\delta_0(s))$-wavepackets of $L_\alpha^{1/2}$ are precisely the elements of
 $\ran(\chi_{I_{s,\delta_0(s)}}(L_\alpha^{1/2}))$, 
where $I_{s,\delta_0(s)}=(s-\gd_0(s),s+\gd_0(s))$.
Using the spectral theorem we see  that if  $I\subseteq[\sqrt{m},\infty)$ is a bounded interval then 
$\ran(\chi_I(L_\alpha^{1/2}))=\ran(\chi_{J_\alpha}(L_1^{1/2}))$, where $J_\alpha=((I^2-m)^{1/\alpha}+m)^{1/2}$. 
Now fix $\alpha\in(0,1)$ and let $\delta_0(s)=c(1+s^{\alpha\inv-1}),\ s \ge 0$, where $c>0$ is a constant. Straightforward estimates show that the images of the intervals $I_{s,\delta_0}\cap[\sqrt{m},\infty)$ under the map $I\mapsto J_\alpha$ have length bounded by some constant multiple of $c$. It follows that~\eqref{eq:KGEqnWPcond} holds also for all $(s,\gd_0(s))$-wavepackets $w$ of $L_\alpha^{1/2}$ provided that $c$ is sufficiently small. (Here the form of the function $\delta_0$ can either be guessed or alternatively derived by considering the images of constant-width intervals  under the inverse of the map $I\mapsto J_\alpha$.)
Moreover, since $D\in \Lin(\Lp[2](\Omega))$ we have 
 $\abs{s}\,\norm{D^\ast ( (1+is)^2+L)\inv D}\lesssim 1$, $s\in\R$.
Thus we deduce from Theorem~\ref{thm:WPtoRG_WE} that $\|(is-A_B)^{-1}\|\lesssim 1+|s|^{2(\alpha^{-1}-1)^{-1}}$ for $s \in \R$. The claim now follows directly from Theorem~\ref{thm:sg_decay}.
\end{proof}

\subsection{A weakly damped beam equation}

In this section we consider the stability of the following Euler--Bernoulli beam equation with weak damping,
  \eq{
    &w_{tt}(\xi,t)+w_{\xi\xi\xi\xi}(\xi,t)+b(\xi)\int_0^1 b(r)w_t(r,t)dr =0, \qquad \xi\in (0,1), ~t>0,\\
    & w(0,t)= 0, \qquad w_{\xi\xi}(0,t)=0, \hspace{4.2cm} t>0, \\
    & w(1,t)= 0, \qquad w_{\xi\xi}(1,t)=0, \hspace{4.2cm} t>0,\\
    & w(\cdot,0)= w_0(\cdot)\in H^4(0,1)\cap H_0^1(0,1), \\
    & w_t(\cdot,0)=w_1(\cdot)\in H^2(0,1)\cap H_0^1(0,1),
  }
  where $b\in \Lp[2](0,1;\R)$ is the damping coefficient. The boundary conditions describe a situation in which the beam is \emph{simply supported}.

The beam equation fits into the framework of Section~\ref{second_order_setup} with the choices $H=\Lp[2](0,1)$ and
\eq{
L = \partial_{\xi\xi\xi\xi}, \quad H_1= \setm{x\in H^4(0,1)}{x(0)=x''(0)=x(1)=x''(1)=0}.
}
The operator $L$ is invertible and positive and its positive square root is given by $\Azhalf =-\partial_{\xi\xi}$ with domain $H_{1/2}=H^2(0,1)\cap H_0^1(0,1)$.  The eigenvalues and normalised eigenfunctions of $\Azhalf $ are given by $\gl_n=n^2\pi^2$ and $\phi_n(\cdot)=\sqrt{2}\sin(n\pi \cdot)$, respectively,  for $n\in\N$.  As in Section~\ref{subs_weak}, $U=\C$ and 
$D\in \Lin(\C,H)$ is the rank-one operator defined by $D u=bu $ for all $u\in \C$.

Our aim is to study the asymptotic behaviour of the solutions of the damped beam equation using the wavepacket condition in Theorem~\ref{thm:WPtoRG_WE}.
Since the eigenvalues $\gl_n=n^2\pi^2$, $n\in\N$, have a uniform gap, we may choose $\gd(s)\equiv \pi^2/4$. The non-trivial $(s,\gd(s))$-wavepackets of $\Azhalf $ are then multiples of the eigenfunctions $\phi_n$ for $n\in\N$ such that $ n^2\pi^2\in (s-\pi^2/4,s+\pi^2/4) $.
For any $n\in\N$ we have
\eq{
\abs{D^\ast \phi_n} = \sqrt{2}\Abs{\int_0^1 b(\xi)\sin(n\pi \xi)d\xi}.
}
These Fourier sine series coefficients are identical to the ones in
Section~\ref{sec:wave1Dweakdamping}.
However, the locations of the eigenvalues of $A$ now result in a slower rate of resolvent growth than in the case of the wave equation.
In order to have $i\R\subseteq \rho(A_{B})$ it is again necessary that $D^\ast \phi_n\neq 0$ for all $n\in\N$.
However, since the gaps between the eigenvalues $n^2\pi^2$ of $\Azhalf $ grow without bound as $n\to\infty$, the same damping has a greater relative effect for the beam equation than for the wave equation.

\begin{corollary}
\label{cor:1DbeamWeakDamping}
Assume that $\abs{D^\ast \phi_n}\gtrsim f(n^2\pi^2)$ for a continuous strictly decreasing function $f:\R_+\to(0,\infty)$ such that $f(\cdot)\inv$ has positive increase.
Then there exist $C,t_0>0$ such that
for every $w_0\in H_1$ and $w_1\in H_{1/2}$ the (classical) solution of
the weakly damped beam equation
satisfies
\eq{
\norm{(w(\cdot,t),w_t(\cdot,t))}_{H^2\times \Lp[2]}\leq  \frac{C}{N\inv(t)}\norm{(w_0,w_1)}_{H^4\times H^2}, \qquad t\geq t_0,
}
where $N\inv$ is the inverse function of $N(\cdot):=f(\cdot)^{-2}$.
Moreover, if there exists an increasing sequence
 $(n_k)_{k\in\N}\subseteq \N$
such that
$\abs{D^\ast\phi_{n_k}}\lesssim f(n_k\pi)$ for all $k\in\N$,
then the decay rate is optimal in the sense of Theorem~\textup{\ref{thm:OptimMain}}.
\end{corollary}

The coefficients $\abs{D^\ast \phi_n}$ for the functions $b$ defined by $b(\xi)=1-\xi$, $b(\xi)=\xi^2(1-\xi)$ and $b(\xi)=\chi_{(0,\xi_0)}(\xi)$ (with $\xi_0\in (0,1)$ an irrational number of constant type) are presented in~\eqref{eq:1DwaveWeakDampingCoeffs}, and for these functions Corollary~\ref{cor:1DbeamWeakDamping} implies the asymptotic rates
      $t^{-1}$, $t^{-1/3}$ and $t^{-1/3}$ as $t\to\infty$, respectively.
Note finally that Remark~\ref{rem:DampingClass} applies also in the setting of this section.


\begin{thebibliography}{100}

\bibitem{AmmBch17}
K.~Ammari, A.~Bchatnia, and K.~El~Mufti.
\newblock Non-uniform decay of the energy of some dissipative evolution
  systems.
\newblock {\em Z. Anal. Anwend.}, 36(2):239--251, 2017.

\bibitem{AmmTuc01}
K.~Ammari and M.~Tucsnak.
\newblock Stabilization of second order evolution equations by a class of
  unbounded feedbacks.
\newblock {\em ESAIM Control Optim. Calc. Var.}, 6:361--386, 2001.

\bibitem{AmmNic15}
K. Ammari and S. Nicaise, 
\newblock Stabilization of elastic systems by collocated feedback. 
\newblock Lecture Notes in Mathematics, 2124. Springer, Cham, 2015.


\bibitem{AnaLea14}
N.~Anantharaman and M.~L{\'e}autaud.
\newblock Sharp polynomial decay rates for the damped wave equation on the
  torus.
\newblock {\em Anal. PDE}, 7(1):159--214, 2014.
\newblock With an appendix by St{\'e}phane Nonnenmacher.

\bibitem{AnaLea16}
N.~Anantharaman, M.~L\'{e}autaud, and F.~Maci\`a.
\newblock Wigner measures and observability for the {S}chr\"{o}dinger equation
  on the disk.
\newblock {\em Invent. Math.}, 206(2):485--599, 2016.

\bibitem{AreBat11book}
W.~Arendt, C.~J.~K. Batty, M.~Hieber, and F.~Neubrander.
\newblock {\em Vector-Valued Laplace Transforms and Cauchy Problems}.
\newblock Birkh{\"a}user, Basel, second ed., 2011.

\bibitem{BatChi16}
C.~J.~K. Batty, R.~{Chill}, and Y.~{Tomilov}.
\newblock {Fine scales of decay of operator semigroups}.
\newblock {\em J. Europ. Math. Soc.}, 18(4):853--929, 2016.

\bibitem{BatDuy08}
C.~J.~K. Batty and T.~Duyckaerts.
\newblock Non-uniform stability for bounded semi-groups on {B}anach spaces.
\newblock {\em J. Evol. Equ.}, 8:765--780, 2008.

\bibitem{Ben78a}
C.~Benchimol.
\newblock Feedback stabilizability in {H}ilbert spaces.
\newblock {\em Appl. Math. Optim.}, 4(3):225--248, 1978.

\bibitem{BorTom10}
A.~Borichev and Y.~Tomilov.
\newblock Optimal polynomial decay of functions and operator semigroups.
\newblock {\em Math. Ann.}, 347(2):455--478, 2010.

\bibitem{BurHit07}
N.~Burq and M.~Hitrik.
\newblock Energy decay for damped wave equations on partially rectangular
  domains.
\newblock {\em Math. Res. Lett.}, 14(1):35--47, 2007.

\bibitem{BurJol16}
N.~Burq and R.~Joly.
\newblock Exponential decay for the damped wave equation in unbounded domains.
\newblock {\em Commun. Contemp. Math.}, 18(6):1650012, 27, 2016.

\bibitem{BurZui16}
N.~Burq and C.~Zuily.
\newblock Concentration of {L}aplace eigenfunctions and stabilization of weakly
  damped wave equation.
\newblock {\em Comm. Math. Phys.}, 345(3):1055--1076, 2016.

\bibitem{BurZwo04}
N.~Burq and M.~Zworski.
\newblock Geometric control in the presence of a black box.
\newblock {\em J. Amer. Math. Soc.}, 17(2):443--471, 2004.

\bibitem{BurZwo19}
N.~Burq and M.~Zworski.
\newblock Rough controls for {S}chr\"{o}dinger operators on 2-tori.
\newblock {\em Ann. H. Lebesgue}, 2:331--347, 2019.

\bibitem{CavMa19arxiv}
M.~M. {Cavalcanti}, T.~F. {Ma}, P.~{Mar{\'\i}n-Rubio}, and P.~N.
  {Seminario-Huertas}.
\newblock {Dynamics of Riemann waves with sharp measure-controlled damping}.
\newblock {\em arXiv e-prints}, page arXiv:1908.04814, Aug 2019.

\bibitem{Chen}
G. Chen, S. A. Fulling,  F. J. Narcowich, S. Sun,  
\newblock Exponential decay of energy of evolution equations with locally distributed damping. 
\newblock {\em SIAM J. Appl. Math.} 51(1):266--301, 1991.


\bibitem{Chill_et_al}
R. Chill, D. Seifert, and Yu. Tomilov. 
\newblock Semi-uniform stability of operator semigroups and energy decay of damped waves. 
\newblock {\em Philos. Trans. Roy. Soc. A} 378 (2020), no. 2185, 24 pp.

\bibitem{CurWei06}
R.~F. Curtain and G.~Weiss.
\newblock Exponential stabilization of well-posed systems by colocated
  feedback.
\newblock {\em SIAM J. Control Optim.}, 45(1):273--297, 2006.

\bibitem{CurWei19}
R.~F. Curtain and G.~Weiss.
\newblock Strong stabilization of (almost) impedance passive systems by static output feedback. 
\newblock \emph{Math. Control Relat. Fields} 9(4): 643--671, 2019. 

\bibitem{DatKle20}
K.~{Datchev} and P.~{Kleinhenz}.
\newblock Sharp polynomial decay rates for the damped wave equation with
  {H}\"older-like damping.
\newblock \emph{Proc. Amer. Math. Soc.} 148(8): 3417–-3425, 2020.

\bibitem{DebSei19b}
G.~{Debruyne} and D.~{Seifert}.
\newblock {Optimality of the quantified Ingham-Karamata theorem for operator
  semigroups with general resolvent growth}.
\newblock {\em Arch. Math.},   113(6): 617--627,  2019.

\bibitem{DelPat21}
F.~{Dell'Oro} and V.~{Pata}.
\newblock {Second order linear evolution equations with general dissipation}.
\newblock  {\em Appl. Math. Optim.} 83(3): 1877--1917, 2021.

\bibitem{Duy07}
T.~Duyckaerts.
\newblock Optimal decay rates of the energy of a hyperbolic-parabolic system
  coupled by an interface.
\newblock {\em Asymptot. Anal.}, 51(1):17--45, 2007.

\bibitem{DyaMuk19}
M.~I. D'yachenko, A.~B. Mukanov, and S.~Yu. Tikhonov.
\newblock Smoothness of functions and {F}ourier coefficients.
\newblock {\em Mat. Sb.}, 210(7):94--119, 2019.


\bibitem{Edw79book}
R.~E. Edwards.
\newblock {\em Fourier series. {A} modern introduction. {V}ol. 1},
 volume~64 of {\em Grad. Texts in Math.}.
\newblock Springer-Verlag, New York-Berlin, second ed., 1979.


\bibitem{GohSig72}
I. C. {Gohberg} and E. I. {Sigal}.
\newblock {An operator generalization of the logarithmic residue theorem and
  the theorem of Rouche.}
\newblock {\em {Math. USSR, Sb.}}, 13:603--625, 1972.

\bibitem{Gre20}
W.~{Green}.
\newblock On the energy decay rate of the fractional wave equation on
  {$\mathbb{R}$} with relatively dense damping.
\newblock {\em Proc. Amer. Math. Soc.} 148(11): 4745--4753, 2020. 

\bibitem{GuoLuo02a}
B.~Z.~Guo and Y.~H.~Luo.
\newblock Controllability and stability of a second-order hyperbolic system
  with collocated sensor/actuator.
\newblock {\em Systems Control Lett.}, 46(1):45--65, 2002.


\bibitem{HasHil09}
A.~Hassell, L.~Hillairet, and J.~Marzuola.
\newblock Eigenfunction concentration for polygonal billiards.
\newblock {\em Comm. Partial Differential Equations}, 34(4-6):475--485, 2009.

\bibitem{Jaf90}
S.~Jaffard.
\newblock Contr\^{o}le interne exact des vibrations d'une plaque rectangulaire.
\newblock {\em Portugal. Math.}, 47(4):423--429, 1990.

\bibitem{JolLau20}
R.~{Joly} and C.~{Laurent}.
\newblock {Decay of semilinear damped wave equations: cases without geometric
  control condition}.
\newblock  {\em Ann. H. Lebesgue} 3: 1241--1289, 2020.

\bibitem{Kat61}
T.~Kato.
\newblock A generalization of the {H}einz inequality.
\newblock {\em Proc. Japan Acad.}, 37:305--308, 1961.


\bibitem{Khi64book}
A.~Ya. Khinchin.
\newblock {\em Continued fractions}.
\newblock The University of Chicago Press, 1964.

\bibitem{Lan95book}
S.~{Lang}.
\newblock {\em Introduction to {D}iophantine Approximations.}
\newblock Springer-Verlag New York, new exp. ed. edition, 1995.

\bibitem{LasTri81}
I.~Lasiecka and R.~Triggiani.
\newblock A cosine operator approach to modeling {$L_{2}(0,\,T;\ L_{2}(\Gamma
  ))$}---boundary input hyperbolic equations.
\newblock {\em Appl. Math. Optim.}, 7(1):35--93, 1981.

\bibitem{LaTri00}
 I. Lasiecka, R. Triggiani, 
\newblock {\em Control Theory for Partial Differential Equations: Continuous and Approximation Theories. Volume II. Abstract hyperbolic-like systems over a finite time horizon},
volume~75 of  {\em Encyclopedia Math. Appl.}.
\newblock
 Cambridge University Press, Cambridge, 2000. 


\bibitem{LasTri03}
I.~Lasiecka and R.~Triggiani.
\newblock {$L_2(\Sigma)$}-regularity of the boundary to boundary operator
  {$B^\ast L$} for hyperbolic and {P}etrowski {PDE}s.
\newblock {\em Abstr. Appl. Anal.}, (19):1061--1139, 2003.


\bibitem{LatShv01}
Y.~Latushkin and R.~Shvydkoy.
\newblock Hyperbolicity of semigroups and {F}ourier multipliers.
\newblock In {\em Systems, approximation, singular integral operators, and
  related topics ({B}ordeaux, 2000)}, volume~129 of {\em Oper. Theory Adv.
  Appl.}, pages 341--363. Birkh\"auser, Basel, 2001.

\bibitem{LauLea21}
C.~Laurent and M.~L\'eautaud.
\newblock Logarithmic decay for linear damped hypoelliptic wave and {S}chr\"odinger equations, {\em SIAM J. Control Optim.}, 59(3):1881--1902, 2021. 
 
\bibitem{LeaLer17}
M.~L\'{e}autaud and N.~Lerner.
\newblock Energy decay for a locally undamped wave equation.
\newblock {\em Ann. Fac. Sci. Toulouse Math. (6)}, 26(1):157--205, 2017.

\bibitem{Leb96}
G.~Lebeau.
\newblock \'{E}quation des ondes amorties.
\newblock In {\em Algebraic and geometric methods in mathematical physics
  ({K}aciveli, 1993)}, volume~19 of {\em Math. Phys. Stud.}, pages 73--109.
  Kluwer Acad. Publ., Dordrecht, 1996.

\bibitem{LetSun21}
C.~Letrouit and C.~Sun.
\newblock Observability of  {B}aouendi--{G}rushin-type equations through resolvent estimates.
\newblock {\em J. Inst. Math. Jussieu}, published online, 2021.


\bibitem{LiuRao05}
Z. Liu and B. Rao.
\newblock Characterization of polynomial decay rate for the solution of linear
  evolution equation.
\newblock {\em Z. Angew. Math. Phys.}, 56(4):630--644, 2005.

\bibitem{LiuZha15}
Z. Liu and Q. Zhang.
\newblock A note on the polynomial stability of a weakly damped elastic
  abstract system.
\newblock {\em Z. Angew. Math. Phys.}, 66(4):1799--1804, 2015.

\bibitem{MalSta20}
S.~{Malhi} and M.~{Stanislavova}.
\newblock {On the energy decay rates for the 1D damped fractional Klein-Gordon
  equation}.
\newblock  {\em Math. Nachr.} 293(2):363--375, 2020.

\bibitem{Mil12}
L.~Miller.
\newblock Resolvent conditions for the control of unitary groups and their
  approximations.
\newblock {\em J. Spectr. Theory}, 2(1):1--55, 2012.

\bibitem{Naz98}
F.~L. Nazarov.
\newblock The {B}ang solution of the coefficient problem.
\newblock {\em Algebra i Analiz}, 9(2):272--287, 1997.
\newblock Translation in {\em St. Petersburg Math. J.} 9(2):407--419, 1998.

\bibitem{Oos00}
J.  Oostveen.
\newblock {\em Strongly stabilizable distributed parameter systems}, 
 volume 20 of {\em Frontiers in Applied Mathematics}.
\newblock SIAM,  Philadelphia, PA, 2000. 


\bibitem{Pau17b}
L.~Paunonen.
\newblock Robust controllers for regular linear systems with
  infinite-dimensional exosystems.
\newblock {\em SIAM J. Control Optim.}, 55(3):1567--1597, 2017.

\bibitem{RTTT05}
K. Ramdani, T. Takahashi, G. Tenenbaum, M. Tucsnak, 
\newblock A spectral approach for the exact observability of infinite-dimensional systems with skew-adjoint generator.
\newblock {\em J. Funct. Anal.} 226(1):193--229, 2005.

\bibitem{Rauch} 
J. Rauch, X. Zhang, and E. Zuazua,
\newblock Polynomial decay for a hyperbolic-parabolic coupled system.
\newblock \emph{J. Math. Pures Appl.} 84(4):407--470, 2005.

\bibitem{RozSei19}
J.~Rozendaal, D.~Seifert, and R.~Stahn.
\newblock Optimal rates of decay for operator semigroups on {H}ilbert spaces.
\newblock {\em Adv. Math.}, 346:359--388, 2019.

\bibitem{Rus75}
D.~L. Russell.
\newblock Decay rates for weakly damped systems in {H}ilbert space obtained
  with control-theoretic methods.
\newblock {\em J. Differential Equations}, 19(2):344--370, 1975.

\bibitem{RzeSch18}
{\L}.~Rzepnicki and R.~Schnaubelt.
\newblock Polynomial stability for a system of coupled strings.
\newblock {\em Bull. Lond. Math. Soc.}, 50(6):1117--1136, 2018.

\bibitem{Sal87a}
D.~Salamon.
\newblock Infinite-dimensional linear systems with unbounded control and observation: {A} functional analytic approach.
\newblock {\em Trans. Amer. Math. Soc.},
300(2):383--431, 1987.


\bibitem{Sle74}
M.~Slemrod.
\newblock A note on complete controllability and stabilizability for linear
  control systems in {H}ilbert space.
\newblock {\em SIAM J. Control}, 12:500--508, 1974.


\bibitem{Sog88}
C.~D. Sogge.
\newblock Concerning the {$L^p$} norm of spectral clusters for second-order
  elliptic operators on compact manifolds.
\newblock {\em J. Funct. Anal.}, 77(1):123--138, 1988.

\bibitem{Sta02}
O. Staffans. 
\newblock Passive and conservative continuous-time impedance and scattering systems. Part I: Well-posed systems. 
\newblock {\em Math. Control Signals Systems}, 15(4):291–-315, 2002.

\bibitem{Sta17}
R.~Stahn.
\newblock Optimal decay rate for the wave equation on a square with constant
  damping on a strip.
\newblock {\em Z. Angew. Math. Phys.}, 68(2):36, 2017.


\bibitem{SuTuc20}
P.~Su, M.~Tucsnak and G.~Weiss.
\newblock Stabilizability properties of a linearized water waves system.
\newblock {\em Systems Control Lett.}, 139:104672, 2021.

\bibitem{Sun21}
C.~Sun.  
\newblock Polynomial stablization for the wave equation with convex-shaped damping.
\newblock {\em arXiv e-prints}, page arXiv:2106.11782, Jun 2021.

\bibitem{TucWei09book}
M.~Tucsnak and G.~Weiss.
\newblock {\em Observation and Control for Operator Semigroups}.
\newblock Birkh\"auser Basel, 2009.

\bibitem{TucWei14}
M.~Tucsnak and G.~Weiss.
\newblock Well-posed systems---the {LTI} case and beyond.
\newblock {\em Automatica J. IFAC}, 50(7):1757--1779, 2014.

\bibitem{Wei03}
G.~Weiss.
\newblock Optimal control of systems with a unitary semigroup and with collocated control and observation.
\newblock {\em Systems Control Lett.}, 48(3-4):329--340, 2003.

\bibitem{WeiTuc03}
G.~Weiss and M.~Tucsnak.
\newblock How to get a conservative well-posed linear system out of thin air. Part {I}. {W}ell-posedness and energy balance
\newblock {\em ESAIM Control Optim. Calc. Var.}, 9:247--274, 2003.


\bibitem{Zyg02book}
A.~Zygmund.
\newblock {\em Trigonometric series. {V}ol. {I}, {II}}.
\newblock Cambridge Mathematical Library. Cambridge University Press, Cambridge, third edition, 2002.

\end{thebibliography}
\end{document}